\documentclass[11pt]{article}
\usepackage[margin=.75in]{geometry}
\usepackage{amsfonts,amsmath,amssymb,amsthm,booktabs,graphicx,cleveref,url,float}
\usepackage{algorithm,algpseudocode}
\usepackage{tikz}
\usetikzlibrary{arrows.meta,positioning,shapes.geometric}
\newtheorem{theorem}{Theorem}[section]

\newtheorem{proposition}[theorem]{Proposition}
\newtheorem{corollary}[theorem]{Corollary}
\numberwithin{algorithm}{section}
\theoremstyle{definition}
\theoremstyle{remark}
\newcommand{\R}{\mathbb R}
\long\def\secondarynumericsA#1\endsecondarynumericsA{\gdef\storednumericsA{#1}}
\long\def\secondarynumericsB#1\endsecondarynumericsB{\gdef\storednumericsB{#1}}
\long\def\secondarynumericsC#1\endsecondarynumericsC{\gdef\storednumericsC{#1}}
\long\def\secondarynumericsD#1\endsecondarynumericsD{\gdef\storednumericsD{#1}}
\long\def\secondarynumericsE#1\endsecondarynumericsE{\gdef\storednumericsE{#1}}
\long\def\secondarynumericsF#1\endsecondarynumericsF{\gdef\storednumericsF{#1}}
\long\def\secondarytheoryA#1\endsecondarytheoryA{\gdef\storedtheoryA{#1}}
\long\def\secondarytheoryB#1\endsecondarytheoryB{\gdef\storedtheoryB{#1}}
\long\def\secondarytheoryC#1\endsecondarytheoryC{\gdef\storedtheoryC{#1}}
\long\def\secondarytheoryD#1\endsecondarytheoryD{\gdef\storedtheoryD{#1}}
\title{Schur--Riesz Refinement for Variational Approximation}
\author{Matthew Dixon}
\date{August 2026}
\begin{document}
\maketitle

\begin{abstract}
Width theory identifies economical spaces for compact PDE solution families
\cite{kolmogorov1936,babuskalipton2011}, but does not provide a stable adaptive
selection rule.  We introduce Schur--Riesz refinement, which compares ordinary
h/p refinement with operator-informed functions on a common variational scale.
Projection removes content already represented by the incumbent, Riesz bounds
test coefficient stability, and exact Schur gain ranks the surviving
directions.  New non-regression, bulk-contraction, and mixed near-oracle results
provide finite-run error, dimension, and work certificates for coercive
Galerkin and noncoercive minimum-residual formulations.

We applied Schur--Riesz refinement across coercive and noncoercive PDEs.  At
matched dimension, automatic modes reduce held-out error by
$50.2\%$ for 2D heterogeneous Helmholtz and $30.9\%$ for 2D Darcy.  At matched Darcy error, the selected space uses
$38.2\%$ fewer coordinates, reducing deployment memory by $38.9\%$ and online
time by $30.1\%$.  The additional offline construction is reusable and can be
amortized across multiple right-hand sides for a fixed operator.
On locked 3D heterogeneous Helmholtz tests, the method reduces mean error by
$25.1$--$47.3\%$ against matched spectral-polynomial spaces and meets target
errors with 24--96 coordinates, versus 1,331--6,859 for native MFEM
hp-refinement.  In an adaptive audit, 48 coordinates attain mean error
$.1092$, compared with $.0967$ using 878 polynomial coordinates, while making
online solves $6.8\times$ faster.
Thus, the experiments demonstrate that Schur--Riesz refinement can construct
smaller stable spaces with lower deployment memory and repeated-solve cost,
while preserving a certified incumbent when richer functions do not help.
\end{abstract}

\section{Introduction}
Adaptive finite elements usually enlarge a trial space through mesh refinement
or higher polynomial degree.  These choices are effective for broad Sobolev
classes: Bramble--Hilbert theory explains their algebraic rates, and graded
hp-spaces can converge exponentially for piecewise analytic solutions
\cite{ciarlet1978,schwab1998}.  A different opportunity arises when the PDE
supplies additional structure---for example waves, interfaces, multiscale
correctors, or transported profiles---that local polynomials may reproduce
only through many degrees of freedom.  Trefftz, partition-of-unity, multiscale,
and transfer-operator methods exploit precisely this opportunity
\cite{hiptmairmoiolaperugia2016,babuskamelenk1997,
efendievgalvishou2013,babuskalipton2011}.

Width theory explains why this can work but does not by itself provide an
adaptive decision.  For a specified compact solution family and norm, the
Kolmogorov $r$-width gives the best worst-case error attainable by an
$r$-dimensional linear space \cite{kolmogorov1936,pinkus1985}.  When the
family is generated by a compact PDE transfer operator, its leading singular
functions are dimension-optimal and generally nonpolynomial
\cite{babuskalipton2011}.  This conclusion is conditional on the chosen
family and norm; it does not displace the generic rate-optimality of
polynomials.  Nor does it determine, from a finite computable collection,
which proposed function is new, stable, and useful for the current
variational problem.

We build on the successive quotient--Riesz stability theory developed in the
companion approximation paper \cite{dixon2026schurriesz}.  The present paper
turns that abstract stability result into a variational refinement method with
PDE error control, adaptive selection, and safe transfer.  We introduce
\emph{Schur--Riesz refinement} to make that decision.  Classical
h/p functions and operator-informed functions are proposed to a common
incumbent space.  Each candidate block is projected against what the
incumbent already represents; its Schur complement measures the remaining
innovation, a quotient--Riesz bound tests stable realizability, and an exact
residual identity measures variational gain.  Thus neither polynomial nor
operator-informed refinement is prescribed in advance: the method retains
whichever supplies the most stable and valuable new direction.  Rejected
automatic transfers return the incumbent exactly.

\paragraph{Literature review.}
Convergence and quasi-optimality of adaptive h-refinement are classical
\cite{binevdahmendevore2004,casconetal2008}; D\"orfler marking converts local
indicators into a bulk refinement rule \cite{dorfler1996}, and hp residual
estimation is well developed \cite{melenkwohlmuth2001}.  Richer trial spaces
also have a substantial history.  Partition-of-unity and extended finite
elements admit problem-specific local functions
\cite{belytschkoblack1999,babuskamelenk1997}; Trefftz methods use local PDE
solutions for oscillatory wave problems \cite{hiptmairmoiolaperugia2016}; and
GMsFEM and localized orthogonal decomposition construct multiscale spaces from
local operator information \cite{efendievgalvishou2013,malqvistpeterseim2014}.
These methods establish the value of nonpolynomial approximation, but their
candidate construction and acceptance criteria are method-dependent.

Kolmogorov-width and transfer-operator theory provide a more intrinsic account
of dimension efficiency \cite{pinkus1985,babuskalipton2011}.  Randomized local
transfer spaces make such constructions computable for parameterized,
time-dependent, and advection-dominated problems
\cite{schleusssmetanatermaat2023}.  For moving features, one fixed linear space
may have slowly decaying widths, motivating nonlinear unions or manifolds
\cite{cohendahmenmula2022}.  These results characterize or approximate strong
spaces for a specified solution family; they do not supply a common adaptive
comparison among h-, p-, and heterogeneous operator-informed proposals.

The stability literature addresses a complementary issue.  Beyond coercivity,
DPG and minimum-residual methods place stability in the trial--test pair and
construct near-optimal test norms \cite{demkowiczgopalakrishnan2011,chanheuer2021}.
Equilibrated reconstruction supplies reliable error control for coercive
problems \cite{braessschoberl2008,ernvohralik2015}.  Finally, the companion
successive Schur--Riesz theorem controls redundant coefficient blocks through
quotient stability \cite{dixon2026schurriesz}.  What remains missing across
these strands is a single decision rule that removes incumbent content,
tests stable realizability, measures variational gain, and safely compares
polynomial with operator-informed refinement.  That is the gap addressed here.

The new contribution is the resulting width-to-decision framework: a
common-incumbent Schur comparison for heterogeneous candidate blocks;
prospective certificates for approximately constructed operator coordinates;
an independently audited non-regression rule; and mixed-pool contraction,
finite-run, work, and perturbed near-oracle bounds.  Classical singular-value
optimality, Galerkin identities, concentration inequalities, and the abstract
quotient--Riesz theorem are cited where invoked rather than claimed anew.

The paper asks three questions.  Can operator-guided candidates be generated
without access to the audit solution and compared fairly with h/p refinement?
Can transfer be accepted while separately controlling novelty, coefficient
stability, and variational value?  Can these tests remain computable for both
coercive and noncoercive problems without forming a dense complement?
Sections~2--4 answer these questions through conditional-width identities,
Schur--Riesz gain and stability tests, safe transfer, and matrix-free
complementary elimination.  The numerical studies then test the claims on
Darcy, Helmholtz, and Maxwell problems; supporting transport, DPG, scaling,
and negative-control studies are deferred to the supplementary material.

The experiments are not intended to establish universal superiority over
hp-FEM.  They test the sharper conditional claim predicted by the theory:
operator-informed functions can improve a matched-size approximation when
they contain stable innovation relevant to the PDE family, while the same
acceptance rule must preserve a stronger incumbent when they do not.

Figure~\ref{fig:sr-refinement} gives the refinement logic before the formal
notation is introduced.  Candidate generation is separated from acceptance:
projection removes represented content, stability and gain test the surviving
innovation, and an independent audit either authorizes transfer or preserves
the incumbent.

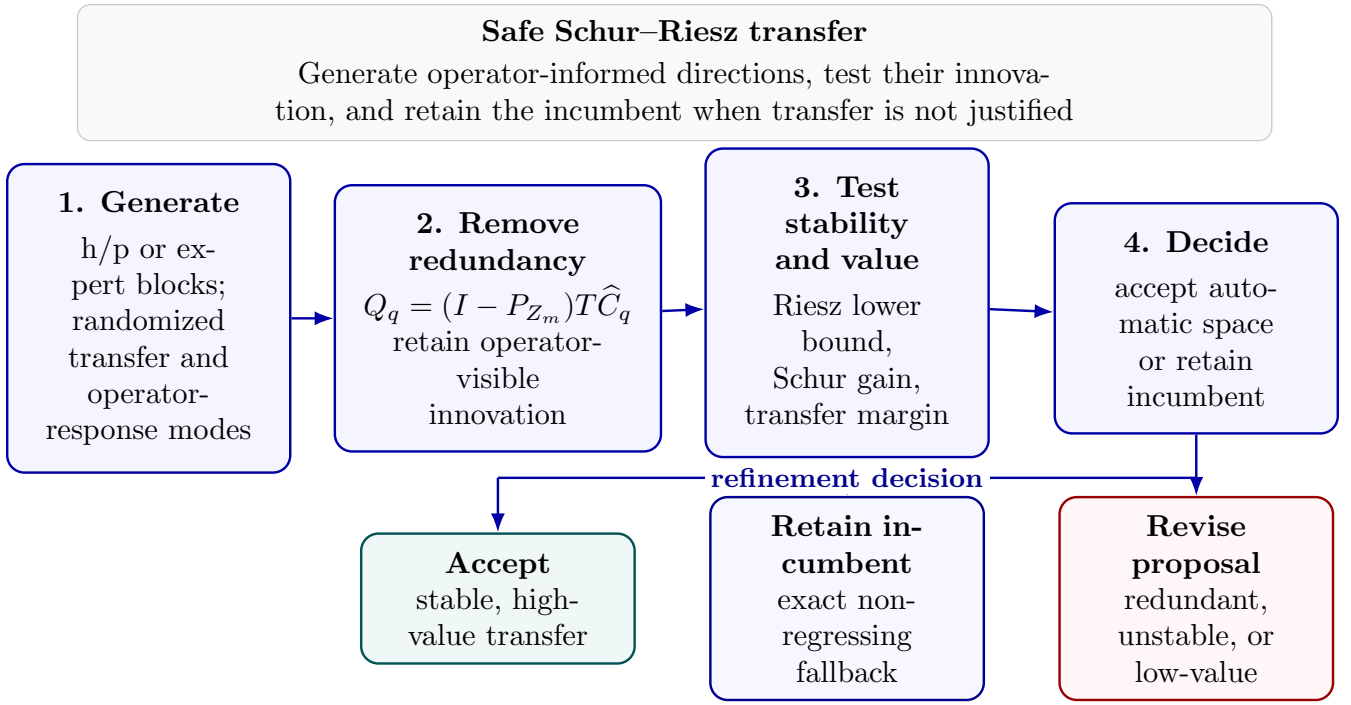
\begin{figure}[H]
\centering
\resizebox{\textwidth}{!}{%
\begin{tikzpicture}[
  font=\small,
  stage/.style={rounded corners=2mm,draw=blue!65!black,thick,fill=blue!4,
                minimum height=1.85cm,align=center,inner sep=9pt},
  banner/.style={rounded corners=1.5mm,draw=gray!45,fill=gray!5,
                 align=center,inner sep=5pt},
  outcome/.style={rounded corners=2mm,thick,align=center,inner sep=6pt,
                  minimum height=1.05cm},
  arr/.style={-{Latex[length=2.2mm]},thick,draw=blue!65!black}
]
\node[banner,text width=13.5cm] at (6.10,2.85)
 {\textbf{Safe Schur--Riesz transfer}\\[2pt]
  Generate operator-informed directions, test their innovation, and retain the incumbent when transfer is not justified};

\node[stage,text width=2.65cm] (cand) at (0,0)
 {\textbf{1. Generate}\\[3pt]
  h/p or expert blocks;\\
  randomized transfer and\\
  operator-response modes};
\node[stage,text width=3.15cm] (schur) at (4.05,0)
 {\textbf{2. Remove redundancy}\\[3pt]
  \mbox{$Q_q=(I-P_{Z_m})T\widehat C_q$}\\
  retain operator-visible\\innovation};
\node[stage,text width=2.65cm] (riesz) at (8.10,0.18)
 {\textbf{3. Test stability and value}\\[3pt]
  Riesz lower bound,\\
  Schur gain, transfer margin};
\node[stage,text width=2.65cm] (gain) at (12.15,0)
 {\textbf{4. Decide}\\[3pt]
  accept automatic space\\
  or retain incumbent};
\draw[arr] (cand)--(schur);
\draw[arr] (schur)--(riesz);
\draw[arr] (riesz)--(gain);

\node[outcome,draw=teal!65!black,fill=teal!6,text width=2.75cm] (accept) at (4.05,-3.25)
 {\textbf{Accept}\\stable, high-value transfer};
\node[outcome,draw=blue!55!black,fill=blue!3,text width=2.75cm] (revise) at (8.10,-3.25)
 {\textbf{Retain incumbent}\\exact non-regressing fallback};
\node[outcome,draw=red!60!black,fill=red!3,text width=2.75cm] (reject) at (12.15,-3.25)
 {\textbf{Revise proposal}\\redundant, unstable, or low-value};
\draw[thick,draw=blue!65!black] (gain.south) -- (12.15,-1.85) -- (4.05,-1.85);
\draw[arr] (4.05,-1.85) -- (accept.north);
\draw[arr] (8.10,-1.85) -- (revise.north);
\draw[arr] (12.15,-1.85) -- (reject.north);
\node[fill=white,inner sep=2pt,font=\footnotesize\bfseries,text=blue!55!black]
 at (8.10,-1.85) {refinement decision};
\end{tikzpicture}
}
\caption{Safe automatic-transfer logic.  Here $T$ is the trial-to-test map,
$Z_m=T(V_m)$ is the represented operator image, and $\widehat C_q$ is a
computed candidate block. Projection forms the genuinely new operator-visible
content $Q_q$.  Its Schur gain measures task value and its Riesz lower bound
measures stable realizability.  An independently evaluated transfer margin
authorizes replacement; otherwise the incumbent is returned unchanged.  For
coercive Galerkin problems, $T$ is the energy Riesz representation, while for
noncoercive problems it is the stable trial-to-test map.}
\label{fig:sr-refinement}
\end{figure}

Figure~\ref{fig:sr-refinement} is deliberately kept simple to emphasize the
refinement logic, and therefore suppresses one central detail: the distinction
between coercive and noncoercive problems. In a coercive problem, the energy Riesz map supplies a
canonical norm and local residual information is often effective; the method
therefore provides a common comparison and certification layer over mature
h/p technology. Beyond coercivity, stability belongs to the trial--test pair,
and local forcing can generate a global response. Here the same framework is
more consequential: it determines the correct geometry, detects when local
enrichment is structurally inadequate, and admits compressed nonlocal
operator responses.  The experiments show material gains in both regimes,
while also showing that fine-grid compatible Maxwell can favor the incumbent.

\section{Schur--Riesz transfer in variational spaces}
To turn width information into a usable numerical method, three questions must
be resolved: how candidate directions are computed from the operator; which
part is absent from the incumbent space; and, after unstable innovations are
rejected, which admissible block has the greatest certified variational value.
We first describe a
general transfer construction, then record conformity and nesting requirements
for finite-element enrichment, and finally derive quotient stability and exact
gain.  This separates candidate generation, representation, stability, and
task utility before any particular PDE or function family is chosen.

At stage $m$, let $C_{
m hp}$ and $C_{\rm op}$ denote proposed polynomial and
operator-derived blocks and set $Z_m=T(V_m)$.  Their residual-space innovations
are defined symmetrically by
\begin{equation}
 Q_{\rm hp}=(I-P_{Z_m}^Y)(T\circ C_{\rm hp}),\qquad
 Q_{\rm op}=(I-P_{Z_m}^Y)(T\circ C_{\rm op}).
 \label{eq:common-incumbent-innovations}
\end{equation}
Both use the same incumbent projector.  This removes what either family
already shares with the incumbent, allowing their remaining operator-visible
innovations to be compared on the same variational scale.  Riesz admissibility
is then checked for each nonzero quotient block, after which certified gain
ranks the survivors.
If one block is accepted, $V_m$ is enlarged and Eq.~\eqref{eq:common-incumbent-innovations}
is recomputed.  Consequently expressions such as
$(I-P_{T(V_m+C_{\rm hp})}^Y)(T\circ C_{\rm op})$ describe a later sequential
factorization, not an assumption that h/p precedes operator enrichment; the
reverse expression is equally valid.

\subsection{Operator-generated transfer candidates}
Let $\mathcal G$ describe admissible loads, boundary traces, initial data, or
local source probes, and let $K:\mathcal G\to X$ be the corresponding
solution, transfer, or response operator.  Examples include an elliptic inverse
restricted to a patch, a Helmholtz trace-to-interior map, and an evolution map
from initial data to a trajectory space.  Width theory identifies the leading
left singular functions of $K$ as optimal for the compact image
$K(B_{\mathcal G})$.  They can be approximated without an audit solution by
applying $K$ to a deterministic quadrature or randomized probe matrix $\Omega$
and compressing the response block
\begin{equation}
 \widehat C=K\Omega.
 \label{eq:automatic-transfer-block}
\end{equation}
Local oversampling, randomized range finding, transfer eigenproblems, and
empirical trajectory covariance are different ways to construct the response
block in Eq.~\eqref{eq:automatic-transfer-block}.

The raw range of $\widehat C$ is not yet an admissible refinement.  Relative to
an incumbent space $V_m$, form
\[
 Q=(I-P_{T(V_m)}^Y)(T\circ\widehat C),\qquad D=Q^*Q.
\]
The singular vectors of $Q$ are the automatically generated directions that
remain after incumbent content has been removed.  Proposition~\ref{prop:conditional-width}
identifies the spectrum of $D$ with the conditional width spectrum of this
computable response family.  Theorem~\ref{thm:main-exact-gain} measures its
variational value, and Theorem~\ref{thm:safe-automatic-transfer} decides
whether it may replace the incumbent.  Thus width theory proposes directions;
Schur--Riesz transfer turns them into a stable, non-regressing decision.

The construction deliberately parallels hierarchical p-refinement.  If
$V_{h,p}$ is a conforming degree-$p$ finite element space and $B_{p+1}$ is its
next hierarchical polynomial block, with coefficient space $E_{p+1}$, then
the familiar p-step is
\begin{equation}
 V_{h,p+1}=V_{h,p}+B_{p+1}(E_{p+1}).
 \label{eq:p-step-analogy}
\end{equation}
A general enrichment step has the same nested form,
\begin{equation}
 V_{m+1}=V_m+C_m(E_m),
 \label{eq:c-step-analogy}
\end{equation}
where $C_m$ may be an h/p block, a compressed response from
Eq.~\eqref{eq:automatic-transfer-block}, or a nonpolynomial function.  This
greater freedom creates two questions that the standard polynomial ladder
largely avoids: whether $C_m$ merely reproduces functions already in $V_m$,
and whether its genuinely new component is useful for the present PDE.
Projection and exact variational gain answer these questions below.
Composition-driven c-refinement is the special case in which composition
generates $C_m$; it changes the candidate class, not the Galerkin principle.

\subsection{Conforming enriched spaces}
We begin with nesting, since innovation and gain are meaningful across
refinement levels only when previously accepted functions are preserved.
Let $\mathcal T_m$ be a conforming newest-vertex-bisection mesh and let
$M_m$ be a finite collection of compactly represented, conforming physical
macro-functions.  These functions need not be polynomial and need not be
compositions.  Polynomial edge
modes are defined on the two-triangle patch sharing their parent edge;
polynomial and non-polynomial interior modes have zero trace on their parent
triangle.  Set
\[
 V_m=V_1(\mathcal T_m)+\operatorname{span}M_m\subset H_0^1(\Omega).
\]
This formula generalizes the usual hierarchical decomposition of an hp space.
The nodal space $V_1(\mathcal T_m)$ provides the conforming finite
element backbone, while $M_m$ records accepted enrichment functions just as a
hierarchical p-basis records accepted polynomial modes.  The difference is
that membership of $M_m$ is determined by conformity, nonredundancy, and
variational value rather than by polynomial degree alone.

The first requirement for stable adaptive enrichment is exact
preservation: changing the mesh must not alter
an enrichment that has already been accepted.  Otherwise successive errors
would compare different functions rather than nested approximation spaces.
The required preservation is the standard nested-space fact used in conforming
adaptive finite elements \cite{binevdahmendevore2004,ciarlet1978}.  If
$\mathcal T_{m+1}$ refines $\mathcal T_m$ and every function in $M_m$ is
retained as the same physical function, then
$V_m\subset V_1(\mathcal T_{m+1})+\operatorname{span}M_m$.  Hence the old
Galerkin function remains exactly representable and minimization over the
larger space cannot increase the objective.  We record this classical fact to
state the implementation requirement, not as a new lemma.

Thus h-refinement enlarges the space without erasing earlier non-polynomial
content.  This is the nesting property on which the later monotonicity argument
depends.

\subsection{Quotient stability and exact variational gain}
Once nesting is secured, an admissible addition must pass two complementary
tests: its genuinely new component must remain stably distinguishable from
earlier h-, p-, and non-polynomial contributions, and that component must
reduce the Galerkin error.  The first test concerns the geometry of the growing
space; the second depends on the PDE.  We establish both in the symmetric
coercive setting
used by the principal elliptic examples.  Proofs are supplied only for the
new results and nontrivial specializations; established identities are cited
at their point of use.  Minimum-residual extensions, patchwise reconstruction,
quadrature control, and conditional complexity results are provided in the
Supplementary Material.

Suppose the approximation space is generated successively by bounded maps
$S_\ell:E_\ell\to H_0^1(\Omega)$, with
$S^{(m)}c=\sum_{\ell\leq m}S_\ell c_\ell$.  Coefficients are identified when
they represent the same function.  Let
\[
 R_\ell=(I-P_{V_{\ell-1}}^a)S_\ell,\qquad D_\ell=R_\ell^*R_\ell.
\]
Thus $R_\ell$ is the genuinely new part of level $\ell$, and $D_\ell$ is its
energy Gram operator.

Write $R^{(m)}d=\sum_{\ell=1}^mR_\ell d_\ell$. Successive block elimination
gives a compatible factorization
\[
 S^{(m)}=R^{(m)}L_m
\]
on the quotient coefficient space. Here $L_m$ converts the original block
coefficients into innovation coefficients; it is the block analogue of the
unit triangular factor in an $LDL^*$ factorization.

Because different formulas can represent identical or nearly identical
functions, stability must be measured after the old space and exact null
directions have been removed.  The next theorem is the energy-space
specialization of the Successive-Space Riesz Theorem in the companion
approximation paper.  It is stated without proof because the present paper
uses, rather than reproves, that result.

\begin{theorem}[Quotient--Riesz stability
  {\cite[Theorem~2.2]{dixon2026schurriesz}}]
\label{prop:main-successive-stability}
Assume that, on the quotient coefficient spaces,
\[
 \delta I\preceq D_\ell\preceq\Delta I
\]
uniformly in $\ell$, and that the associated block elimination maps and their
inverses have norms bounded by $M$ and $K$.  Then
\begin{equation}
 \frac{\delta}{K^2}\|[c]\|^2
 \leq \|S^{(m)}[c]\|_a^2
 \leq \Delta M^2\|[c]\|^2
 \label{eq:main-successive-stability}
\end{equation}
for every level $m$.  Exact duplicate functions disappear in the quotient,
and the constants do not deteriorate merely because further stable levels are
added.
\end{theorem}
This is the direct point at which the companion Schur--Riesz theory enters.
It does not claim that a stable new function is useful for the PDE.  Rather,
it guarantees that its contribution can be separated from earlier levels and
that coefficient perturbations have a controlled functional effect.  PDE
relevance enters through the residual.

Let $a(\cdot,\cdot)$ be the energy inner product, let $u_m$ be the Galerkin
solution in $V_m$, and let $C:E_C\to H_0^1(\Omega)$ generate a finite candidate
block. Denote its candidate space by $\mathcal C=C(E_C)$. Remove what is
already represented by setting
\[
 Q=(I-P_{V_m}^a)C,\qquad D=Q^*Q,
 \qquad b(z)=F(Qz)-a(u_m,Qz).
\]
The corresponding innovation space is $\mathcal I_C=Q(E_C)$.
Coefficients in $\ker Q$ represent no new function and are identified.

The residual $b$ measures PDE relevance, while $D$ measures the geometry of
the genuinely new direction.  The following proposition is the classical
Hilbert-ellipsoid width identity \cite{kolmogorov1936,pinkus1985} applied after
quotienting the incumbent space; the identity itself is not new.

\begin{proposition}[Projected-ellipsoid width identity
  {\cite{kolmogorov1936,pinkus1985}}]
\label{prop:conditional-width}
Let $B_E$ be the unit ball of the candidate coefficient space after
quotienting $\ker Q$, and let
$\mathcal K_C=Q(B_E)\subset V_m^{\perp_a}$.  If
$\lambda_1(D)\geq\lambda_2(D)\geq\cdots>0$, then
\begin{equation}
 \underbrace{\inf_{\dim W=k}\sup_{z\in\mathcal K_C}
              \inf_{w\in W}\|z-w\|_a}_{d_k(\mathcal K_C;V_m^{\perp_a})}
 =\sigma_{k+1}(Q)=\sqrt{\lambda_{k+1}(D)}.
 \label{eq:conditional-width}
\end{equation}
Hence the Schur spectrum is exactly the Kolmogorov-width spectrum of the
candidate family after removing the current approximation space.
\end{proposition}

Proposition~\ref{prop:conditional-width} is not itself a new bridge theorem:
it is the classical width formula applied to the projected candidate
ellipsoid $\mathcal K_C$.  It supplies the width-theoretic input to the new
framework.  The operational bridge is the subsequent combination of this
projected spectrum with residual gain, quotient stability, and certified
adaptive acceptance.  It is conditional because $\mathcal K_C$ is generated
by the available candidate block, not by the unknown full PDE solution class.

The gain formula is likewise the standard Galerkin Pythagorean identity
written in quotient--Schur coordinates \cite{ciarlet1978}.  We state it to
define the common h/p/operator ranking functional, not as a new orthogonality
theorem.

\begin{theorem}[Galerkin gain identity
  {\cite{ciarlet1978}}]
\label{thm:main-exact-gain}
The decrease in squared energy error obtained by adjoining the block is
\begin{equation}
 \|u-u_m\|_a^2-\|u-P_{V_m+\mathcal C}^au\|_a^2
 =\langle b,D^\dagger b\rangle.
 \label{eq:main-exact-gain}
\end{equation}
It is computable without knowing $u$.  If
$A_CI\preceq D\preceq B_CI$ on the quotient, then the gain lies between
$\|b\|^2/B_C$ and $\|b\|^2/A_C$.
\end{theorem}
The identity is Galerkin Pythagoras: $Q$ is precisely the part of the candidate
orthogonal to the old space, and $D^\dagger b$ is the best correction in that
new direction.  Thus $D$ measures genuinely new functional content while
$b$ measures whether the unresolved PDE residual uses it.

\secondarytheoryA
For repeated solves, Eq.~\eqref{eq:main-exact-gain} also identifies the
operator quantity that replaces a small collection of task residuals.  Let
the load $F$ be random with covariance operator $\Sigma_F$, and construct
$Q$ using the fixed PDE operator and current space.
The following proposition is an elementary covariance specialization of the
classical Galerkin gain identity, rather than a new operator theorem.
\begin{proposition}[Covariance form of Galerkin gain
  {\cite{ciarlet1978}}]
\label{prop:operator-covariance-gain}
The expected squared energy-error reduction from the candidate block is
\begin{equation}
 \mathbb E_F[G_C(F)]
 =\operatorname{tr}\!\left(D^\dagger Q^*\Sigma_F Q\right).
 \label{eq:operator-covariance-gain}
\end{equation}
For an energy-normalized scalar innovation $q$, this reduces to
$\mathbb E_F[G_q(F)]=\langle q,\Sigma_Fq\rangle$.
\end{proposition}
Thus a quadrature or model for admissible loads can rank transferable
directions without solving the PDE for those loads.  The PDE still enters
essentially through the energy projection $Q$ and normalization $D$; the load
covariance specifies which operator-visible directions must transfer.  A
finite pilot average and an operator-covariance score answer different
questions: the former optimizes the observed tasks, whereas the latter
optimizes expected reduction under a declared load law.

The same statement extends beyond coercivity once trial functions are measured
in a stable test-space residual norm.  Let $B:X\to Y^*$ be the variational
operator, let $R_Y:Y\to Y^*$ be the Riesz map, and set
$T=R_Y^{-1}B$ and $g=R_Y^{-1}F$.  For the current residual-image space
$Z_m=T(V_m)$ and candidate image $Z_C=T(\mathcal C)$, define
\[
 Q=(I-P_{Z_m}^Y)(T\circ C),\qquad D=Q^*Q.
\]
The next statement is the minimum-residual specialization of the same
classical projection and covariance identity; its role here is to extend the
ranking functional to stable noncoercive formulations.
\begin{theorem}[Minimum-residual covariance form
  {\cite{demkowiczgopalakrishnan2011}}]
\label{thm:minres-operator-covariance}
If the transformed load $g$ has covariance $\Sigma_g$, then the expected
decrease in squared minimum-residual error caused by adjoining $C$ is
\begin{equation}
 \mathbb E_g\!\left[
 \|g-P_{Z_m}g\|_Y^2-
 \|g-P_{Z_m+Z_C}g\|_Y^2\right]
 =\operatorname{tr}(D^\dagger Q^*\Sigma_gQ).
 \label{eq:minres-operator-covariance}
\end{equation}
If $B$ has inf--sup constant $\beta>0$, the residual norm also controls trial
error through $\beta\|u-u_m\|_X\leq\|B(u-u_m)\|_{Y^*}$.
\end{theorem}
The theorem covers transport, Helmholtz, convection--diffusion, and mixed
problems provided a stable test norm, graph norm, symmetrizer, or
minimum-residual formulation is supplied.  It also explains why covariance in
an arbitrary raw residual norm can be misleading: the trace formula is exact
for that norm, but the norm may have poor equivalence to the desired solution
error.
\endsecondarytheoryA

Automatic transfer construction raises one additional question that does not
arise when a candidate block is merely appended: may an automatically generated
space safely replace an established incumbent of the same dimension?  The next
result gives a population-level answer.  It applies to coercive Galerkin
approximation after identifying a load with its energy representer, and to
noncoercive minimum-residual approximation after transforming the load into the
test space.

The concentration step below is Hoeffding's classical inequality
\cite{hoeffding1963}.  The new result is the selection theorem obtained by
coupling that margin to projection gain, quotient--Riesz admissibility, and
exact incumbent recovery.

\begin{theorem}[Safe automatic-transfer selection]
\label{thm:safe-automatic-transfer}
Let $Y$ be a Hilbert space, let $g$ be a random element of $Y$, and let
$Z_I,Z_A\subset Y$ be respectively an incumbent and an automatically generated
finite-dimensional approximation space.  Define
\[
 \mathcal E(Z)=\mathbb E\|g-P_Zg\|_Y^2,
 \qquad J_Z(g)=\|P_Zg\|_Y^2.
\]
Suppose that $\|g\|_Y^2\leq M$ almost surely and that
$g_1,\ldots,g_n$ are independent marking samples, generated independently of
$Z_I$ and $Z_A$.  Set
\[
 \widehat\Delta_n=\frac{1}{n}\sum_{i=1}^n
       \bigl(J_{Z_A}(g_i)-J_{Z_I}(g_i)\bigr),\qquad
 \tau_{n,\alpha}=M\sqrt{\frac{2\log(1/\alpha)}{n}}.
\]
Accept $Z_A$ only if $\widehat\Delta_n>\tau_{n,\alpha}$ and its quotient
synthesis map has a prescribed positive Riesz lower bound; otherwise return
$Z_I$.  Then, with probability at least $1-\alpha$,
\begin{equation}
 \mathcal E(Z_{\rm safe})\leq \mathcal E(Z_I),
 \quad\hbox{and, if accepted,}\quad
 \mathcal E(Z_A)\leq\mathcal E(Z_I)
       -(\widehat\Delta_n-\tau_{n,\alpha}).
 \label{eq:safe-transfer-nonregression}
\end{equation}
Rejection recovers the incumbent space and approximant exactly.
\end{theorem}

The complete proof is given in Appendix~\ref{app:safe-transfer-proof}.  The
theorem is deliberately a non-regression result for a declared probability
measure, not a claim that one space dominates for every right-hand side.  The
gain comparison can be evaluated without reference solutions: in the coercive
case $g$ is the energy representer of the load, while in a minimum-residual
formulation $g=R_Y^{-1}F$.  The Riesz condition plays a different role from
the statistical margin: it prevents an accepted improvement from being
represented by unstable coefficients.

If the declared load law is a finite quadrature and every quadrature point is
evaluated, $\Delta$ is computed exactly and the margin
$\tau_{n,\alpha}$ is unnecessary.  For a continuously distributed load, the
displayed Hoeffding margin is rigorous but can be conservative.  Consequently,
the small marking batches in the numerical mechanism tests are reported as
prospective gates, not retroactively promoted to population certificates.
Sharper variance-sensitive or deterministic operator-covariance bounds may be
substituted without changing the selection argument.

The two results have complementary roles: quotient stability is an
admissibility gate, while exact gain ranks the surviving directions by their
relevance to the current residual.  They do not determine whether the current
error is small or whether relevant directions are absent from the proposed
families.  That requires an independent global error bound.  The construction
below is stated first for a general variational operator and only later
specialized to familiar flux and graph-residual estimators.

For nonsymmetric or indefinite problems, trial-space innovation alone is not
enough: the discrete test space must also see the new direction. In a
minimum-residual formulation, the Supplementary Material proves that
the exact gain remains a Schur quotient and that stability of an enlarged
trial space $W$ is measured by
\begin{equation}
 \beta_R(W)^2=\lambda_{\min}(H_W^{-1/2}G_WH_W^{-1/2}),
 \label{eq:main-discrete-infsup}
\end{equation}
where $H_W$ is the trial-norm Gram matrix and $G_W$ is the Gram matrix of the
discrete trial-to-test images. Enlarging the test space cannot decrease
$\beta_R(W)$; a Fortin map with norm $C_F$ gives
$\beta_R(W)\geq\beta_{\mathcal B}/C_F$. Thus the innovation calculation also
identifies when a useful trial addition requires paired test enrichment. The
complete statement is given in the Supplementary Material.

The discrete inf--sup constant also supplies the missing bridge from an
operator residual to physical error.  The following enclosure is the standard
triangle-inequality and inf--sup argument used in reliable Galerkin and
minimum-residual analysis \cite{demkowiczgopalakrishnan2011,ciarlet1978}.  It
is stated here because neither a sharp Schur solve nor
a stable coefficient quotient by itself controls error to the continuum
solution.

\begin{theorem}[Discrete-to-continuum enclosure
  {\cite{demkowiczgopalakrishnan2011,ciarlet1978}}]
\label{thm:discrete-continuum-enclosure}
Let $B:X\to Y'$ satisfy
\[
 \beta\|w\|_X\leq \|Bw\|_{Y'}\leq M\|w\|_X,
 \qquad 0<\beta\leq M,
\]
on the relevant error space.  Let $u$ solve $Bu=F$, let $u_h$ be a fixed full
discrete solution, and suppose a computable reliable estimator satisfies
$\|u-u_h\|_X\leq\eta_{\rm disc}$.  For a reduced approximation $v$, suppose
$\underline d_v\leq\|u_h-v\|_X\leq\overline d_v$ has been certified.  Then
\begin{equation}
 \max\{0,\underline d_v-\eta_{\rm disc}\}
 \leq \|u-v\|_X \leq
 \overline d_v+\eta_{\rm disc}.
 \label{eq:continuum-combined-interval}
\end{equation}
If instead only a residual enclosure
$\underline r_v\leq\|F-Bv\|_{Y'}\leq\overline r_v$ is available, valid
endpoints are $\underline r_v/M$ and $\overline r_v/\beta$.  Candidate $v_1$
is therefore certified to improve upon $v_0$ whenever its upper endpoint is
strictly below the lower endpoint for $v_0$.
\end{theorem}

\secondarytheoryB
One common way to supply $\eta_{\rm disc}$ is the classical two-level
saturation estimate \cite{melenkwohlmuth2001}; the required assumption must
remain visible.
\begin{proposition}[Two-level saturation estimate
  {\cite{melenkwohlmuth2001}}]
\label{prop:two-level-saturation}
Let $u_h,u_+\in X$ be two discrete approximations.  If a known $q<1$ satisfies
$\|u-u_+\|_X\leq q\|u-u_h\|_X$, then
\begin{equation}
 {\|u_+-u_h\|_X\over1+q}
 \leq\|u-u_h\|_X\leq
 {\|u_+-u_h\|_X\over1-q}.
 \label{eq:two-level-saturation}
\end{equation}
\end{proposition}
The proposition is a guarantee only when $q$ is established independently;
estimating it from the same two solutions is not a proof.
\endsecondarytheoryB

The saturation-free principle is not specific to transport.  The following
factorized reconstruction lemma is derived in this paper to isolate the only
property required by the subsequent PDE-specific majorants.  Its ingredients
are Galerkin orthogonality and the Riesz representation theorem; novelty is
claimed only for this reusable formulation and its integration into the
refinement certificate.

\begin{theorem}[Factorized residual majorant]
\label{thm:factorized-residual-majorant}
Let $V$ and $Z$ be Hilbert spaces and let $C:V\to Z$ be an isometry, so that
$\|v\|_V=\|Cv\|_Z$.  Let $R\in V'$ have Riesz representative $z$, and let
$z_h$ be its Galerkin projection onto a finite-dimensional subspace of $V$.
If a computable $g_h\in Z$ satisfies
\begin{equation}
 R(v)-(z_h,v)_V=(g_h,Cv)_Z\qquad(v\in V),
 \label{eq:factorized-residual-reconstruction}
\end{equation}
then
\begin{equation}
 \|z_h\|_V\leq\|R\|_{V'}
 \leq\bigl(\|z_h\|_V^2+\|g_h\|_Z^2\bigr)^{1/2}.
 \label{eq:factorized-residual-majorant}
\end{equation}
More generally, if Eq.~\eqref{eq:factorized-residual-reconstruction} holds
up to a functional of known dual norm $\delta_h$, then $\|g_h\|_Z$ in the
upper endpoint may be replaced by $\|g_h\|_Z+\delta_h$.
\end{theorem}

\begin{proof}
Galerkin orthogonality gives
$\|R\|_{V'}^2=\|z_h\|_V^2+\|z-z_h\|_V^2$.
Equation~\eqref{eq:factorized-residual-reconstruction} and Cauchy--Schwarz
give $\|z-z_h\|_V\leq\|g_h\|_Z$.  Combining the two identities proves
Eq.~\eqref{eq:factorized-residual-majorant}.  The perturbed statement follows
from the triangle inequality in $V'$.
\end{proof}

This theorem separates the reusable analysis from the PDE-dependent
reconstruction.  For coercive diffusion, $C$ may be the gradient and $g_h$ an
equilibrated flux.  For first-order minimum-residual formulations,
$Cv=(v,A^*v)$ is the graph embedding.  Trefftz, wave, and mixed formulations
may use their corresponding operator graph or product-space factorization.
Only the construction of $g_h$ changes; the Schur--Riesz refinement and the
error logic do not.

\secondarytheoryC
For transport, saturation can instead be avoided by reconstructing the
unresolved graph residual.  Let
$A^*v=c,v-\boldsymbol\beta\cdot\nabla v$, where
$c=\gamma-\nabla\cdot\boldsymbol\beta$, and equip the broken test space with
$\|v\|_V^2=\|v\|^2+\|A^*v\|^2$.

The next majorant is the paper's transport-specific specialization of the
preceding factorized reconstruction lemma.  Its integration-by-parts lifting
and saturation-free bound are derived here; no general novelty is claimed for
graph norms or residual reconstruction themselves
\cite{demkowiczgopalakrishnan2011}.

\begin{theorem}[Saturation-free graph-residual majorant]
\label{thm:graph-residual-majorant}
Suppose a residual functional has the cellwise representation
\[
 R(v)=(p,v)+(q,A^*v)+\langle\lambda,v\rangle_{\partial\mathcal T}.
\]
Let $z_h$ be any computed approximation to its $V$-Riesz representative.
For every cell $K$, choose a lifting $\ell_K$ satisfying
\[
 -(\boldsymbol\beta\cdot n)\ell_K=\lambda
 \quad\hbox{on }\partial K,
\]
with the equality imposed only on faces carrying the trace residual, and set
\[
 a_{\lambda,K}=-c\ell_K-\nabla\cdot(\boldsymbol\beta\ell_K).
\]
Then, with broken fields
\[
 a=p-z_h+a_\lambda,
 \qquad b=q-A^*z_h+\ell,
\]
the exact dual residual satisfies
\begin{equation}
 \|R\|_{V'}^2
 \leq \|z_h\|_V^2+\|a\|^2+\|b\|^2,
 \label{eq:graph-residual-majorant}
\end{equation}
provided $z_h$ is the Galerkin Riesz projection.  Without Galerkin
orthogonality, the always-valid bound is
$\|R\|_{V'}\leq\|z_h\|_V+(\|a\|^2+\|b\|^2)^{1/2}$.
Consequently, if the continuous inf--sup constant is bounded below by
$\beta_{\mathcal B}>0$, the right side of
Eq.~\eqref{eq:graph-residual-majorant}, divided by
$\beta_{\mathcal B}^2$, is a saturation-free upper bound for the squared
trial-space error.
\end{theorem}

\begin{proof}
Cellwise integration by parts gives
\[
 (a_\lambda,v)+(\ell,A^*v)
 =-\langle(\boldsymbol\beta\cdot n)\ell,v\rangle_{\partial\mathcal T}
 =\langle\lambda,v\rangle_{\partial\mathcal T}.
\]
Hence the unresolved functional $R-(z_h,\cdot)_V$ equals
$(a,\cdot)+(b,A^*\cdot)$.  Cauchy--Schwarz in the product space
$L^2\times L^2$ bounds its dual norm by
$(\|a\|^2+\|b\|^2)^{1/2}$.  Galerkin orthogonality makes the computed and
unresolved Riesz representatives orthogonal, so Pythagoras proves
Eq.~\eqref{eq:graph-residual-majorant}.  The triangle inequality proves the
version without orthogonality, and the continuous inf--sup inequality gives
the final statement.
\end{proof}

The following conversion is the classical continuous inf--sup estimate and is
stated only to identify the norm controlled by the certificate
\cite{demkowiczgopalakrishnan2011}.
\begin{corollary}[Physical-norm conversion
  {\cite{demkowiczgopalakrishnan2011}}]
\label{cor:physical-norm-conversion}
Let $e=u-u_h\in X$ and suppose the continuous operator satisfies
$\beta_X\|w\|_X\leq\|Bw\|_{Y'}$ for every $w$ in the relevant error
space, with an independently established $\beta_X>0$.  If
$\overline\eta_G$ encloses the graph-residual norm furnished by
Theorem~\ref{thm:factorized-residual-majorant}, then
\begin{equation}
 \|u-u_h\|_X\leq {\overline\eta_G\over\beta_X}.
 \label{eq:physical-norm-conversion}
\end{equation}
For the induced energy norm $\|w\|_E:=\|Bw\|_{Y'}$, the constant is exactly
$\beta_E=1$.  For an $L^2$, trace, or other physical norm, the corresponding
continuous lower bound is indispensable; a discrete singular value is not,
by itself, a certified substitute.
\end{corollary}

The constant in Corollary~\ref{cor:physical-norm-conversion} need not be an
unknown calibration parameter.  It can be built into the test norm or bounded
through the adjoint problem.

The adjoint-lifting and optimal-test-norm statement below is standard DPG
stability theory; it is stated without proof to make the required calibration
assumption explicit.
\begin{theorem}[Adjoint-lifting inf--sup calibration
  {\cite{demkowiczgopalakrishnan2011}}]
\label{thm:adjoint-lifting-infsup}
Let $R_X:X\to X'$ be the Riesz map and suppose that, for every $x\in X$, an
adjoint lifting $Qx\in Y$ satisfies
\begin{equation}
 B^*Qx=R_Xx,
 \qquad \|Qx\|_Y\leq C_Q\|x\|_X.
 \label{eq:adjoint-lifting}
\end{equation}
Then the continuous inf--sup constant of $B:X\to Y'$ is at least $C_Q^{-1}$.
If $Y$ is equipped with the ideal physical test norm
$\|v\|_{Y,\mathrm{opt}}:=\|R_X^{-1}B^*v\|_X$, after quotienting
$\ker B^*$, then $C_Q=1$ and the inf--sup constant is exactly one.

For the adjoint graph norm
$\|v\|_Y^2=\|A^*v\|_{L^2}^2+\|v\|_{L^2}^2$, if $A^*$ has a right inverse
$Q_0$ with $\|Q_0x\|_{L^2}\leq C_0\|x\|_{L^2}$, then
\begin{equation}
 \beta_{L^2}\geq(1+C_0^2)^{-1/2}.
 \label{eq:graph-to-l2-infsup}
\end{equation}
\end{theorem}
For large noncoercive discretizations, repeatedly applying the ideal physical
test norm may be as expensive as solving the full PDE.  We therefore use an
independently randomized Petrov approximation.  Draw isotropic probes $s_j$
and solve $B^*z_j=s_j$.  Since
\begin{equation}
 z_j^*(Bw-f)=s_j^*(w-u),
 \label{eq:adjoint-petrov-sketch}
\end{equation}
least squares in the functionals $z_j^*(Bw-f)$ sketches physical solution
error rather than the generally misaligned raw residual.  Separate probe
families select the space and fit its reduced coefficients.  The construction
uses only the operator and declared loads; reference solutions are reserved
for the final audit.  Standard subspace-embedding estimates quantify sampling
error, while the quotient--Riesz gate removes redundant or unstable sketched
directions.
This result gives three operational choices.  A global optimal test norm
gives the desired physical error with constant one.  A localizable graph norm
gives the explicit conversion in Eq.~\eqref{eq:graph-to-l2-infsup}.  Between
them, any spectrally equivalent approximation to the optimal Riesz map gives
a certified constant from its norm-equivalence bounds.  Those bounds can be
verified by the same quotient and complementary-Schur estimates used for
candidate blocks, so the calibration does not require a raw smallest
singular value of the full discretization.

For divergence-free transport whose backward characteristics meet the
adjoint boundary within time $T_{\max}$, the characteristic Volterra operator
satisfies $C_0\leq2T_{\max}/\pi$.  Hence
\begin{equation}
 \beta_{L^2}\geq
 \left[1+\left({2T_{\max}\over\pi}\right)^2\right]^{-1/2}.
 \label{eq:transport-explicit-infsup}
\end{equation}
This bound uses only the prescribed velocity and domain.  It remains valid
without a discrete reference solution, although a worst-case residence time
can be pessimistic.  Verified characteristic enclosures or a
Schur-certified approximation of the optimal test Riesz map sharpen it while
preserving a continuum guarantee.

The lifting is not an oracle: it depends only on the assembled trace residual
and the prescribed transport field.  Minimizing $\|a\|^2+\|b\|^2$ over an
$H(\operatorname{div};K)$ lifting space gives a small independent coercive
problem on each cell.  Any admissible lifting is reliable; richer lifting
spaces improve effectivity without changing validity.

For an ideal minimum-residual method, $\beta=M=1$ in the induced energy norm.
For practical DPG, a Fortin operator gives a computable lower bound for
$\beta$, while $M$ follows from continuity of the bilinear form.  A standard
residual-reliability route is available in the DPG setting
\cite{demkowiczgopalakrishnan2011}.  In the notation used here, it takes the
form
\[
 \eta_{\rm disc}={C_F\over\beta_{\mathcal B}}
 \bigl(\eta_{V^+}+\operatorname{osc}_h\bigr),
\]
provided the enriched-test residual $\eta_{V^+}$ plus data oscillation
$\operatorname{osc}_h$ is a certified upper bound for the full dual residual.
Without a Fortin, localization, or saturation argument, an enriched residual
is only an indicator and must not be presented as a continuum guarantee.
\endsecondarytheoryC

Theorems~\ref{thm:discrete-continuum-enclosure} and
\ref{thm:factorized-residual-majorant} provide the main error-control bridge:
the first combines discretization and reduced-space errors, while the second
gives a reusable saturation-free majorant once the unresolved functional is
reconstructed.  The two-level alternative, transport graph reconstruction,
physical-norm conversion, and adjoint calibration are supporting error-control
results provided in Appendix~\ref{app:secondary-theory}.

\section{Safe automatic transfer and adaptive selection}
\label{sec:main-certificate}
The preceding analysis ranks an exact candidate inside an exact variational
problem.  An adaptive method must additionally decide from computed quantities
whether the current error is acceptable and whether a proposed block remains
beneficial after candidate, quadrature, and algebraic errors are included.
This section supplies that prospective error-control layer and then assembles
it into the adaptive algorithm.
\subsection{The central prospective certificate}
We first isolate the acceptance decision in the paper's main result.  Let $X$ and $Y$ be the trial and test
Hilbert spaces, let $T=R_Y^{-1}B:X\to Y$ be the trial-to-test operator, and let
$r_m=g-Tu_m$ be the current minimum residual. Set $Z_m=T(V_m)$. An ideal operator-derived block
$C:E_C\to X$ is available only through a computed approximation
$\widehat C$. Write $\mathcal C=C(E_C)$ and
$\widehat{\mathcal C}=\widehat C(E_C)$. After removing the current operator
image, set
\[
 Q=(I-P_{Z_m}^Y)(T\circ C),
 \qquad \widehat Q=(I-P_{Z_m}^Y)(T\circ\widehat C).
\]
The following prospective construction-error estimate is new to this paper;
unlike the classical exact-space identities above, it compares an ideal block
with the block that is actually assembled before the enlarged solve.

\begin{theorem}[Prospective operator-coordinate enrichment certificate]
\label{thm:prospective-certificate}
Work on the quotient by $\ker Q$.  Suppose
\[
 \sigma_{\min}(Q)\geq\sqrt{A_C}>0,
 \qquad
 \|\widehat Q-Q\|\leq\varepsilon_C<\sqrt{A_C}.
\]
Let $G_C$ be the exact squared residual reduction obtained by the ideal block
and put $\gamma=G_C/\|r_m\|_Y^2$.  Then the computed block is stable, with
\[
 \sigma_{\min}(\widehat Q)\geq\sqrt{A_C}-\varepsilon_C,
\]
and its enriched residual satisfies
\begin{equation}
 \frac{\displaystyle
 \min_{w\in V_m+\widehat{\mathcal C}}\|g-Tw\|_Y}
 {\|r_m\|_Y}
 \leq
 \rho_{\rm cert}:=
 \sqrt{1-\gamma}
 +\frac{\varepsilon_C}{\sqrt{A_C}}\sqrt{\gamma}.
 \label{eq:main-prospective-certificate}
\end{equation}
In particular, $\rho_{\rm cert}<1$ guarantees improvement before the computed
enriched problem is solved.  If the continuous operator has inf--sup constant
$\beta>0$, division by $\beta$ converts the residual estimate into an
$X$-error estimate.
\end{theorem}

The proof first transfers the lower singular-value bound from $Q$ to
$\widehat Q$, then tests the computed space with the coefficient of the best
ideal correction. The triangle inequality separates ideal residual reduction
from construction error.  The Supplementary Material gives the
complete seven-step argument and the final inf--sup conversion.

The three quantities have noninterchangeable meanings.  The fraction
$\gamma$ measures task relevance of the ideal block; $A_C$ measures stable
functional novelty after the current space and exact redundancies are
removed; and $\varepsilon_C$ measures loss caused by numerical construction
of the operator coordinate.  A large candidate family can therefore be rejected
because its weak tail makes $A_C$ small, even when that tail adds a little
ideal gain.  Conversely, refining the coordinate can reduce
$\varepsilon_C$ and turn the same stable block from uncertified to certified.
This is the mechanism tested prospectively in Section~\ref{sec:main-numerics}.

\subsection{Variational discovery of compositional candidates}
The certificate above evaluates a supplied block; it does not yet explain how
to construct one.  For composition-generated candidates, the corresponding
proposal rule follows from differentiating the same variational gain.  Let
$\Theta\subset\R^p$ parameterize a
family $C_\theta:E_C\to X$ generated by successive maps
\[
 h_0=x,\qquad h_\ell=\Phi_\ell(h_{\ell-1};\theta_\ell),
 \qquad C_\theta e=\Psi(h_L,e).
\]
This includes transported profiles and nested local maps as well as ordinary
fixed functions.  On the frozen current space define
\[
 Q_\theta=(I-P_{Z_m}^Y)(T\circ C_\theta),\quad
 D_\theta=Q_\theta^*Q_\theta,\quad
 b_\theta=Q_\theta^*r_m,\quad
 G(\theta)=\langle b_\theta,D_\theta^\dagger b_\theta\rangle.
\]
Thus the discovery objective is not coefficient size or correlation with an
unprojected residual.  It is the exact residual reduction of the new
functional subspace after old content and exact redundancy have been removed.

The derivative below is the classical variable-projection/envelope formula
\cite{golubpereyra1973} specialized to the quotient variational objective.
What is specific here is the projection of the differentiated range through
$(I-P_{Z_m}^Y)T$ and its coupling to the later stability certificate.

\begin{theorem}[Variable projection in quotient
  coordinates {\cite{golubpereyra1973}}]
\label{thm:compositional-gain-calculus}
Suppose $Q_\theta$ is continuously differentiable and has constant quotient
rank near $\theta$.  Let
\[
 z_\theta=D_\theta^\dagger b_\theta,\qquad
 e_\theta=r_m-Q_\theta z_\theta.
\]
For every parameter variation $\dot\theta$,
\begin{equation}
 DG(\theta)[\dot\theta]
 =2\langle e_\theta,
          DQ_\theta[\dot\theta]z_\theta\rangle_Y.
 \label{eq:compositional-gain-gradient}
\end{equation}
The value $G(\theta)$ is unchanged under every invertible reparameterization
$Q_\theta\mapsto Q_\theta R_\theta$ of the same candidate range.  Moreover,
$DQ_\theta[\dot\theta]$ is obtained by the ordinary chain rule through the
maps $\Phi_1,\ldots,\Phi_L$, followed by the fixed operator projection
$(I-P_{Z_m}^Y)T$.
\end{theorem}

Equation~\eqref{eq:compositional-gain-gradient} is an adjoint or
backpropagation rule, but for an intrinsic variational objective.  It moves
the composition only when the induced \emph{new operator-visible function}
reduces the frozen residual; changes that merely relabel or remix the same
range have zero value.  At a rank change the pseudoinverse is nonsmooth, so
discovery is carried out on a fixed spectral quotient and the rank is changed
only between solve--discover--certify cycles.

For computation, let $\widetilde G(\theta)$ and $\Xi(\theta)$ be the computed
gain and its perturbation radius from
Proposition~\ref{prop:main-safe-ranking}.  A stable discovery problem is
\begin{equation}
 \max_{\theta\in\Theta}
 \bigl\{\widetilde G(\theta)-\Xi(\theta)
       -\lambda\mathcal K(\theta)\bigr\}
 \quad\text{subject to}\quad
 A_C(\theta)\geq A_{\min},\qquad
 \varepsilon_C(\theta)<\sqrt{A_C(\theta)},
 \label{eq:certified-compositional-discovery}
\end{equation}
where $\mathcal K$ is an optional representation-cost functional.  On a
compact parameter set with continuous data and fixed quotient rank, a
maximizer exists.  This optimization only \emph{proposes} a block.  Final
acceptance still requires
Theorem~\ref{thm:prospective-certificate}; consequently an imperfect
search cannot invalidate the residual guarantee.

This yields a two-level variational principle.  Inner coefficient elimination
computes the best use of a fixed compositional range.  Outer differentiation
uses Eq.~\eqref{eq:compositional-gain-gradient} to improve that range.  The
prospective theorem then audits the finite computed map.  Composition is
therefore unified with operator-compatible enrichment at the level of both
optimization and certification, while remaining optional: if no
composition-generated block passes the audit, the method returns to h- or
p-refinement.

The preceding section determines whether a proposed direction is new and useful, but it
does not say how much error remains outside all tested candidates.  Adaptive
refinement therefore requires a second, independent object: a computable upper
bound for the current PDE error.  The quotient-flux reconstruction supplies
this bound; perturbation margins make candidate comparisons reliable; and the
resulting algorithm stops or refines using PDE data alone.

As one optional coercive specialization, consider the Poisson energy form
$a(u,v)=(\nabla u,\nabla v)$ with $F(v)=(q,\nabla v)$, and define
\[
 \mathcal N=\{s\in L^2(\Omega)^d:(s,\nabla v)=0
 \text{ for every }v\in H_0^1(\Omega)\}.
\]
The next identity is the corresponding classical Hilbert
projection/hypercircle error bound.  It supplies a computable stopping criterion for this
specialization; it is not an assumption of the general refinement framework.
The paper's contribution is its separation from, and coupling to,
Schur--Riesz action selection.

\begin{theorem}[Equilibrated-flux identity
  {\cite{braessschoberl2008,ernvohralik2015}}]
\label{thm:main-flux-bound}
Every conforming approximation $u_h$ satisfies
\begin{equation}
 \|u-u_h\|_a
 =\inf_{s\in\mathcal N}\|q-\nabla u_h-s\|_{L^2}.
 \label{eq:main-flux-bound}
\end{equation}
Consequently every computable $s_h\in\mathcal N$ gives the guaranteed upper
bound $\|u-u_h\|_a\leq\|q-\nabla u_h-s_h\|_{L^2}$, without a saturation
assumption or exact reference solution.
\end{theorem}
This is the Hilbert-space distance from the raw flux residual to fields that
are invisible to all gradients.  The Supplementary Material records the
patchwise reconstruction used in the experiments.  In short,
Eq.~\eqref{eq:main-exact-gain} chooses the next action and
Eq.~\eqref{eq:main-flux-bound} supplies the computable stopping criterion.

\subsection{Reliable comparison with inexact variational quantities}
Discovery and acceptance have so far used exact operators.  To preserve their
ordering after discretization, the exact variational-gain identity in
Eq.~\eqref{eq:main-exact-gain} and the
quotient-flux identity in Eq.~\eqref{eq:main-flux-bound} assume exact
variational quantities, whereas assembly, quadrature, and algebraic solution
are finite.  A practical comparison therefore requires a margin that
shows when the ordering of two proposed enrichments survives these errors.
For a proposed block $C$, write
$G_C=\langle b_C,D_C^\dagger b_C\rangle$ and let
$(\widetilde D_C,\widetilde b_C)$ denote the computed quantities.

The following perturbation estimate is derived in this paper.  It combines a
standard inverse-perturbation argument with the quotient lower Riesz bound to
produce an operational acceptance and ranking margin for competing refinement
blocks.

\begin{proposition}[Perturbation-safe ranking]
\label{prop:main-safe-ranking}
Suppose $D_C\succeq A_CI$ on the quotient and
\[
 \|\widetilde D_C-D_C\|\leq\varepsilon_{D,C}<A_C,
 \qquad
 \|\widetilde b_C-b_C\|\leq\varepsilon_{b,C}.
\]
Then
\begin{align}
 |G_C-\widetilde G_C|&\leq\Xi_C,\qquad
 \widetilde G_C=\langle\widetilde b_C,
                         \widetilde D_C^\dagger\widetilde b_C\rangle,\\
 \Xi_C&=
 \frac{2\|\widetilde b_C\|\varepsilon_{b,C}+\varepsilon_{b,C}^2}{A_C}
 +\frac{\|\widetilde b_C\|^2\varepsilon_{D,C}}
 {A_C(A_C-\varepsilon_{D,C})}.
 \label{eq:main-ranking-margin}
\end{align}
Consequently $\widetilde G_C-\Xi_C>0$ proves positive improvement, and
$\widetilde G_C-\widetilde G_{C'}>\Xi_C+\Xi_{C'}$ proves that $C$ has greater
true variational gain than $C'$.
\end{proposition}
The lower quotient bound $A_C$ has an operational role: nearly dependent
functions enlarge the uncertainty margin and are accepted only if their
residual gain is correspondingly stronger.  The Supplementary Material gives
the proof and separates the
quadrature, map, and algebraic contributions to the perturbation radii.

\subsection{Global impact of local decisions}
The preceding margin protects a single comparison, but a sequence of locally
positive choices could still remove an arbitrarily small fraction of the
remaining residual.  The missing global quantity is therefore \emph{coverage}:
how much of the current residual is captured by the joint innovation space of
the entire admissible candidate pool.  This leads to a contraction test that
is computed from the same Schur quantities as the refinement decision and does
not assume estimator reduction.

D\"orfler bulk marking is the classical principle behind selecting a fixed
fraction of an error indicator \cite{dorfler1996}.  The theorem below is the
new quotient--Schur specialization used here: coverage is measured by the
joint operator-image innovation of a mixed h/p/operator candidate pool, and
the retained fraction yields computable residual, dimension, and work
certificates.

\begin{theorem}[Bulk-coverage contraction]
\label{thm:bulk-coverage}
Let $u_m$ minimize $\|g-Tv\|_Y$ over $V_m$, let
$r_m=g-Tu_m$, and let $W_m\subset Z_m^\perp$ be the joint operator-image
innovation of the certified candidate pool.  Denote its exact aggregate gain
by
\[
 G_m^{\rm pool}=\|P_{W_m}r_m\|_Y^2
\]
and suppose a computable lower bound satisfies
\begin{equation}
 \underline G_m^{\rm pool}\leq G_m^{\rm pool},
 \qquad
 \underline G_m^{\rm pool}\geq\alpha\|r_m\|_Y^2
 \quad\text{for some }\alpha>0.
 \label{eq:coverage-condition}
\end{equation}
If the selected block or batch has a certified gain
$\underline G_m^{\rm sel}\geq
\theta\underline G_m^{\rm pool}$, where $0<\theta\leq1$, then
\begin{equation}
 \|r_{m+1}\|_Y^2\leq(1-\theta\alpha)\|r_m\|_Y^2.
 \label{eq:bulk-contraction}
\end{equation}
Consequently, uniform positive lower bounds for $\alpha$ and $\theta$ imply
geometric convergence in the induced residual norm.
\end{theorem}

\begin{corollary}[Finite-run convergence and complexity certificate]
\label{cor:finite-run-complexity}
For any accepted sequence of $N$ batches, let $\alpha_m$ and $\theta_m$ be
the computed lower coverage and retained-gain fractions in
Eq.~\eqref{eq:coverage-diagnostics}. Then
\begin{equation}
 \|r_N\|_Y^2\leq \|r_0\|_Y^2
 \prod_{m=0}^{N-1}(1-\alpha_m\theta_m).
 \label{eq:finite-run-envelope}
\end{equation}
If $\alpha_m\geq\alpha_0>0$ and $\theta_m\geq\theta_0>0$ until
$\|r_m\|_Y\leq\tau$, the required number of accepted batches satisfies
\begin{equation}
 N\leq
 \left\lceil
 \frac{\log(\|r_0\|_Y^2/\tau^2)}
      {-\log(1-\alpha_0\theta_0)}
 \right\rceil .
 \label{eq:finite-run-batch-bound}
\end{equation}
If batch $m$ adds $d_m$ quotient coordinates and costs $\kappa_m$, the
certified representation dimension and accumulated refinement work are
$\dim V_0+\sum_{m<N}d_m$ and $\sum_{m<N}\kappa_m$, respectively.
\end{corollary}
This corollary is a direct finite-run consequence of the new contraction
theorem: iterate Eq.~\eqref{eq:bulk-contraction}, then solve the resulting
geometric inequality.  No separate proof is needed.

Contraction alone does not show whether the accepted sequence uses its
dimensions effectively.  To compare it with the best mixed refinement of a
given size, fix a finite, prospectively specified pool $\mathcal Q$ of h-, p-,
and operator-derived blocks.  For $S\subset\mathcal Q$, let
\begin{equation}
 \mathcal F(S)=\|r_0\|_Y^2-
 \min_{v\in V_0+\sum_{q\in S}\mathcal C_q}\|g-Tv\|_Y^2
 \label{eq:captured-gain-set-function}
\end{equation}
be the total residual energy captured after refitting all selected blocks.
Weak-submodular greedy bounds are classical
\cite{elenbergetal2018}; the contribution below is their quotient--Schur
specialization to mixed variational refinement, including computable
finite-assembly margins.

\begin{theorem}[Certified mixed-refinement near-oracle bound]
\label{thm:mixed-near-oracle}
Fix an integer $k\geq1$.  Suppose that, after quotienting exact redundancy,
every conditional joint Schur Gram operator formed from at most $2k$ blocks
has lower bound $A_{2k}>0$, while every conditional one-block Schur Gram
operator has upper bound $B_1<\infty$.  Set
$\gamma=A_{2k}/B_1\leq1$.

Starting with $S_0=\varnothing$, at step $m$ choose the block maximizing its
computed marginal gain $\widetilde G_m(q)$, refit in the enlarged space, and
assume Proposition~\ref{prop:main-safe-ranking} supplies
$|G_m(q)-\widetilde G_m(q)|\leq\Xi_m(q)$.  Define
$\delta_m=2\max_{q\notin S_m}\Xi_m(q)$ and
$\rho=1-\gamma/k$.  Then, for $0\leq m\leq k$ and every comparator
$O\subset\mathcal Q$ with $|O|\leq k$,
\begin{equation}
 \mathcal F(S_m)\geq
 (1-\rho^m)\mathcal F(O)
 -\sum_{j=0}^{m-1}\rho^{m-1-j}\delta_j .
 \label{eq:mixed-near-oracle}
\end{equation}
In particular, after $k$ accepted blocks,
\begin{equation}
 \mathcal F(S_k)\geq
 (1-e^{-\gamma})\mathcal F(O)
 -\sum_{j=0}^{k-1}\rho^{k-1-j}\delta_j .
 \label{eq:mixed-near-oracle-k}
\end{equation}
Equivalently, the squared residual of the computed mixed refinement is at
most the oracle residual plus
$e^{-\gamma}\mathcal F(O)$ and the displayed, fully computable inexactness
penalty.
\end{theorem}

The theorem separates approximation richness from numerical reliability.
The best admissible mixed sequence enters only through $\mathcal F(O)$;
$\gamma$ measures how strongly candidate innovations interfere after exact
redundancy is removed; and $\delta_m$ records quadrature, assembly, and solve
uncertainty.  Thus a poor result is attributable to a weak candidate pool, an
ill-conditioned joint innovation, or insufficient numerical accuracy.  The
comparison is deliberately made against a fixed prospective pool.  If new
candidates are generated from the observed residual, the same conclusion
holds conditionally for the realized pool only when generation and auditing
are separated; no unrestricted adaptive-oracle claim is made.
The proof is given in Appendix~\ref{app:mixed-near-oracle-proof}.

This certificate is deliberately a posteriori: it makes no uniform richness
claim for an arbitrary candidate generator. Every completed run instead
returns a computable residual envelope, dimension, and work account. If the
coverage floor fails, Algorithm~\ref{alg:hpc} returns
\textsc{NoCoverage}; insufficient candidate richness is then diagnosed rather
than hidden inside an asymptotic convergence claim.

The two observed ratios
\begin{equation}
 \alpha_m=\frac{\underline G_m^{\rm pool}}{\|r_m\|_Y^2},
 \qquad
 \theta_m=\frac{\underline G_m^{\rm sel}}
 {\underline G_m^{\rm pool}}
 \label{eq:coverage-diagnostics}
\end{equation}
separate two failure modes.  A small $\alpha_m$ says that the current library
does not cover the unresolved residual and must be enlarged or spatially
remarked; a small $\theta_m$ says that selection is too timid and should admit
a block rather than one function.  Neither diagnosis can be obtained from a
positive one-candidate gain alone.  In a coercive energy formulation the same
argument applies after the Riesz identification; for a general minimum-
residual formulation, an inf--sup lower bound converts residual convergence to
the corresponding physical norm.  The proof and the older estimator-reduction
alternative are given in the Supplementary Material.

\secondarytheoryD
The following is the standard AFEM quasi-error contraction argument under
estimator reduction and reliability \cite{casconetal2008}; it is retained only
as an alternative to the new bulk-coverage theorem.
\begin{proposition}[Conditional quasi-error contraction
  {\cite{casconetal2008}}]
\label{prop:main-contraction}
Let $\eta_m$ be reliable,
$\|u-u_m\|_a^2\leq C_{\rm rel}\eta_m^2$, and suppose the accepted h-, p-, or
c-action satisfies
\[
 \eta_{m+1}^2\leq q\eta_m^2+C_{\rm stab}
 \|u_{m+1}-u_m\|_a^2,\qquad 0<q<1.
\]
For $0<\gamma<C_{\rm stab}^{-1}$, the quasi-error
$\mathcal E_m^2=\|u-u_m\|_a^2+\gamma\eta_m^2$ obeys
\begin{equation}
 \mathcal E_{m+1}^2\leq
 \left(1-\frac{\gamma(1-q)}{C_{\rm rel}+\gamma}\right)
 \mathcal E_m^2.
 \label{eq:main-contraction}
\end{equation}
\end{proposition}
\endsecondarytheoryD

\subsection{Safe Schur--Riesz transfer and refinement algorithm}
These ingredients now combine into one adaptive loop.  The quotient-flux bound answers
whether the current approximation is sufficiently accurate; the innovation
identity answers which admissible enlargement gives the largest guaranteed
decrease.  Keeping these questions separate prevents a small local residual
from being confused with a useful new basis function.

\begin{algorithm}[H]
\caption{Safe Schur--Riesz transfer and adaptive refinement}\label{alg:hpc}
\begin{algorithmic}[1]
\Require $T:X\to Y$, $g\in Y$, incumbent $V_0$, optional transfer operator
$K$ and independent marking data $\{g_i\}_{i=1}^n$, generators
$\mathfrak D_h,\mathfrak D_p,\mathfrak D_c$, tolerance $\tau>0$,
$0<\theta\leq1$, coverage floor $\alpha_0>0$, costs $\kappa(C)$,
certified construction data
\Ensure $(u_m,\eta_m)$ with $\eta_m\leq\tau$, or a diagnosed failure
\If{an automatic transfer proposal is requested}
  \State $\widehat C\gets\Call{Compress}{K\Omega}$ using
    Eq.~\eqref{eq:automatic-transfer-block}
  \State quotient $(I-P_{T(V_0)}^Y)(T\circ\widehat C)$ and construct
    the matched candidate space $V_A$
  \State $V_0\gets\Call{SafeTransfer}{V_0,V_A,\{g_i\},\alpha}$ using
    Theorem~\ref{thm:safe-automatic-transfer}
  \Comment{accept $V_A$ or retain $V_0$ exactly}
\EndIf
\For{$m=0,1,\ldots$}
  \State $u_m\gets\arg\min_{v\in V_m}\|g-Tv\|_Y$
  \State $Z_m\gets T(V_m)$; $r_m\gets g-Tu_m$
  \State $(\eta_m,\{\eta_{m,K}\}_{K\in\mathcal T_m})
     \gets\Call{Estimate}{u_m}$
  \If{$\eta_m\leq\tau$}
    \State \Return $(u_m,\eta_m)$
  \EndIf
  \State $\mathcal M_m\gets\Call{Dorfler}
     {\{\eta_{m,K}\},\theta}$
  \State $\mathfrak D_m\gets\Call{Generate}{\mathcal M_m;h,p}
  \cup\Call{Discover}{r_m;c}$ using Eq.~\eqref{eq:certified-compositional-discovery}
  \ForAll{$\widehat C:E_C\to X$ in $\mathfrak D_m$}
    \State $\widehat Q_C\gets(I-P_{Z_m}^{Y})(T\circ\widehat C)$
    \State $\widetilde D_C\gets\widehat Q_C^*\widehat Q_C$;
      $\widetilde b_C\gets\widehat Q_C^*r_m$
    \State quotient $\ker\widehat Q_C$ and set
      $(A_C,\varepsilon_C,\Xi_C)\gets\Call{Audit}{\widehat C}$
    \State $\underline G_C\gets
      \langle\widetilde b_C,\widetilde D_C^\dagger
      \widetilde b_C\rangle-\Xi_C$
    \State $s_C\gets[\underline G_C>0]\wedge[A_C>0]$
    \If{certified ideal-coordinate data are declared}
      \State $\rho_C\gets\Call{Prospective}
        {A_C,\varepsilon_C,\gamma_C}$ using Eq.~\eqref{eq:main-prospective-certificate}
      \State $s_C\gets s_C\wedge
        [\varepsilon_C<\sqrt{A_C}]\wedge[\rho_C<1]$
    \EndIf
  \EndFor
  \State $\mathfrak A_m\gets
    \{\widehat C\in\mathfrak D_m:s_{\widehat C}=\mathrm{true}\}$
  \If{$\mathfrak A_m=\varnothing$}
    \State \Return \textsc{NoGain}
  \EndIf
  \State jointly quotient $\{\widehat Q_C:C\in\mathfrak A_m\}$ and
    compute $\underline G_m^{\rm pool}$
  \State $\alpha_m\gets\underline G_m^{\rm pool}/\|r_m\|_Y^2$
  \If{$\alpha_m<\alpha_0$}
    \State \Return \textsc{NoCoverage} with diagnostic $\alpha_m$
  \EndIf
  \State $\mathcal B_m\gets\Call{BulkSelect}{\mathfrak A_m,\theta}$:
    $\underline G_{\mathcal B_m}\geq
    \theta\underline G_m^{\rm pool}$ at least cost
  \State $V_{m+1}\gets V_m+
    \sum_{\widehat C\in\mathcal B_m}\widehat C(E_{\widehat C})$
\EndFor
\end{algorithmic}
\end{algorithm}

The initial transfer stage is optional: without it the algorithm recovers the
established certified refinement loop exactly.  With it, operator responses
are compressed and compared with the incumbent before adaptive enlargement
begins.  The same safe decision can be repeated after a material change in the
operator or declared load law.  This optional matched-space replacement is
distinct from the adaptive loop, where every h-, p-, and operator-derived
proposal is projected against the same $V_m$ as in
Eq.~\eqref{eq:common-incumbent-innovations}.  Here c denotes a compact, problem-adapted
enrichment generated by composition; other operator-derived blocks enter
through the same discovery and certification interface.
Algorithm~\ref{alg:hpc} does not prescribe tent maps or any single
compositional family.  Composed tent functions are used only to illustrate how
depth can encode progressively finer scales compactly.  By themselves they
remain piecewise affine, may require many orientations to resolve curved or
anisotropic structure, and contain no information about the PDE operator.
They are therefore unsuitable as a universal enrichment family.  The principal
experiments instead use harmonic or localized spectral functions; the scope
studies also consider singular, interface, asymptotic, and transported
profiles.  Location, scale, orientation, profile family, phase, transport
history, and, where relevant, composition depth may be continuous candidate
parameters.  The theorem compares the resulting
functions only after projection onto the current space.  Candidate generation
may therefore be approximate, but acceptance remains variational: an optimizer
proposes a direction and the lower-gain calculation decides whether it enters
the space.  Bulk selection can also reduce offline work: one joint coverage
calculation supports several accepted functions before the next global solve,
assembly, and audit.  A scalable implementation should generate candidates
lazily, maintain incremental Schur/QR lower bounds, and stop enlarging the
pool once the required coverage has been certified; forming a complete dense
candidate family is not required by the theorem.  The prospective gate $\rho_C<1$ is imposed when a computed block
approximates a declared ideal operator coordinate and certified
$(A_C,\varepsilon_C,\gamma_C)$ are available.  For an ordinary h/p block
assembled directly in the discrete space, $\varepsilon_C=0$ and the
perturbation-safe lower gain is the operative test.  Thus the algorithm does
not demand an artificial ideal map for classical finite-element actions.
One-dimensional actions are not forced: an h midpoint can lie at an
oscillatory zero, and one member of a structured enrichment family can be weak
although a small block is useful.  Comparing exact nested blocks avoids both
false stopping mechanisms.

For factor-aware integration, the active maps generate a finite filtration
$\mathcal G_0\subset\cdots\subset\mathcal G_L$.  Gaussian rules operate on its
atoms; only the join of factors occurring in a particular Gram entry is
needed.  The rigorous margin is
$2\sum_A\mu(A)\epsilon_A$, where $\epsilon_A$ certifies uniform polynomial
approximation on atom $A$; the corresponding corollary is given in the
Supplementary Material.  The numerical implementation reports an
order-doubling estimate for this margin except where piecewise-polynomial degree
counting proves exactness.

\secondarynumericsD
\subsection{Full supporting PDE study}
This section evaluates the theory and the resulting adaptive method.  We ask whether richer functions reduce
Galerkin error and intrinsic dimension under the same reference-independent
stopping rule as hp; whether predicted gain equals realized decrease; and
whether the flux majorant bounds the independently audited error.  Solver
cost is reported separately because multigrid solves in a fixed space, whereas
h, p, and c change that space.
The evidence must test the theory rather than merely exhibit small errors.  We
therefore record whether the Schur quotient removes redundancy, whether the
predicted gain equals the realized Galerkin decrease, whether the flux
majorant bounds an independently audited error, and whether the accepted space
improves on hp at a common stopping rule.  Fine or exact solutions are used
only for these post-hoc audits.  Controlled scope and falsification examples
are reported in the Supplementary Material.

\subsection{Controlled trial--test stability beyond coercivity}
We first isolate the mechanism in Eq.~\eqref{eq:main-discrete-infsup} on the
advection--reaction problem
\[
 \mathcal Bu=u'+2u=f_d,\qquad u(0)=0,\qquad
 f_d(x)=1+\cos(2^d\pi x).
\]
With the graph norm $\|u\|_X=\|\mathcal Bu\|_{L^2}$, the continuous inf--sup
constant is one. The base trial function is the exact response to the constant
load. For the remaining oscillatory load, propose the operator-derived
zero-inflow response $c_d$, so that
\[
 \mathcal Bc_d=\cos(2^d\pi x).
\]
These functions have a depth-$d$ description through
$\sin(2t)=2\sin(t)\cos(t)$ and $\cos(2t)=2\cos^2(t)-1$. We compare the
unchanged one-dimensional test space, shifted Legendre test spaces, and the
paired test addition $\cos(2^d\pi x)$. Every entry in
Table~\ref{tab:paired-infsup-pilot} is obtained from the generalized
eigenvalue in Eq.~\eqref{eq:main-discrete-infsup}; no reference solution or
fitted indicator is used.

\begin{table}[t]
\centering
\caption{Discrete stability in the controlled advection--reaction problem. The
unchanged test space is blind to each oscillatory trial innovation. A paired
test direction restores $\beta_R=1$ with two test coordinates; the polynomial
test dimension needed for $\beta_R\geq .99$ grows with frequency and hence
exponentially in composition depth. The base residual is $1/\sqrt2$ in every
row and the enriched residual is zero to quadrature precision.}
\begin{tabular}{rrrrr}
\toprule
depth & frequency & unchanged $\beta_R$ & polynomial dimension & paired $\beta_R$\\
\midrule
1 & 2   & $2.61\,10^{-13}$ & 5   & 1.000\\
2 & 4   & $2.62\,10^{-13}$ & 7   & 1.000\\
3 & 8   & $2.60\,10^{-13}$ & 15  & 1.000\\
4 & 16  & $2.61\,10^{-13}$ & 27  & 1.000\\
5 & 32  & $2.58\,10^{-13}$ & 53  & 1.000\\
6 & 64  & $2.47\,10^{-13}$ & 103 & 1.000\\
7 & 128 & $2.41\,10^{-13}$ & 203 & 1.000\\
\bottomrule
\end{tabular}
\label{tab:paired-infsup-pilot}
\end{table}

Thus trial-only enrichment is insufficient beyond coercivity: an economical
trial direction can make the discrete equations singular when the tests do
not observe it. Paired enrichment recovers stability at constant dimension
when the operator image shares the recursive structure; generic polynomial
tests recover it only after resolving the oscillation. This is a controlled
PDE mechanism test, not yet a claim against mature DPG or Trefftz methods on a
natural benchmark. Moreover, the omitted component's ordinary $L^2$ norm
decreases from $.101$ at depth one to $.00176$ at depth seven even though its
graph-norm residual remains $1/\sqrt2$. The present advantage is therefore
sharp residual control and stable representation, not a uniform claim about
$L^2$ error. A multidimensional convection- or wave-dominated benchmark with
matched DPG/Trefftz controls is the next required experiment.

\paragraph{Impedance Helmholtz control.}
We next consider unit-square scattering,
$-\Delta u-k^2u=0$, with impedance data chosen so that
$u(x,y)=\exp(ik(d_1x+d_2y))$ and
$d=(\cos(.37\pi),\sin(.37\pi))$. This is a natural Trefftz positive control:
the plane wave is one exact trial direction and its impedance traces are
paired test directions on the four edges. Table~\ref{tab:helmholtz-paired}
compares tensor Legendre trial approximation with polynomial boundary tests.
The degrees are selected independently for one-percent $L^2$ error and for a
discrete visibility constant of $.99$.

\begin{table}[t]
\centering
\caption{Polynomial and paired Trefftz representations for the impedance
Helmholtz control. Polynomial trial dimension grows quadratically with
wavenumber at fixed accuracy, while boundary-test dimension grows linearly.
The exact Trefftz trial and four paired traces have constant dimension.}
\begin{tabular}{rrrrrr}
\toprule
$k$ & polynomial degree & trial DOFs & test degree & test DOFs & paired trial/test DOFs\\
\midrule
4  & 4  & 25   & 2  & 12  & 1/4\\
8  & 7  & 64   & 4  & 20  & 1/4\\
16 & 11 & 144  & 8  & 36  & 1/4\\
32 & 19 & 400  & 15 & 64  & 1/4\\
64 & 35 & 1296 & 30 & 124 & 1/4\\
\bottomrule
\end{tabular}
\label{tab:helmholtz-paired}
\end{table}

This control separates established structure from the new methodology. The
one-coordinate approximation is the classical Trefftz advantage, not a new
consequence of Schur--Riesz analysis. The present framework instead compares
that direction intrinsically with the current trial space, computes its exact
residual gain, and rejects it unless the test space supplies the required
inf--sup visibility. At $k=64$, 124 generic polynomial edge tests are needed
to reach $.99$, whereas four matched traces give one. A heterogeneous medium,
where no exact plane wave is available and candidates must compete, remains
the decisive next test.

We performed that test on a locked finite-difference discretization with a
smooth heterogeneous refractive inclusion and a localized incident-wave
source. Candidate sets comprised 64 tensor sine modes, 72 windowed plane waves
over prescribed directions and scales, their union, and a 40-vector Krylov
control. Every adaptive decision used quotient-orthogonalized operator images
and exact residual gain; the direct solution was reserved for auditing.

\begin{table}[t]
\centering
\caption{Forty-direction heterogeneous Helmholtz competition. The union finds
complementary polynomial and wave innovations and gives the smallest operator
residual, but the unpreconditioned residual does not sharply control relative
$L^2$ solution error.}
\begin{tabular}{lrrrr}
\toprule
space & dimension & polynomial/wave directions & relative residual & relative $L^2$ error\\
\midrule
polynomial & 40 & 40/0  & .697 & .925\\
windowed waves & 40 & 0/40  & .683 & .964\\
union & 40 & 22/18 & .439 & .858\\
Krylov & 33 & 0/0 & .526 & .961\\
\bottomrule
\end{tabular}
\label{tab:heterogeneous-helmholtz}
\end{table}

The mixed result is a useful falsification of residual-only reasoning. The
nearest absolute eigenvalue is $6.66$, whereas the operator norm is about
$1.56\,10^4$; moreover
$\|f\|/(\sigma_{\min}\|u\|)=10.60$. Hence the union's residual gives only the
vacuous relative $L^2$ bound $10.60(.439)=4.65$. Schur selection successfully
finds complementary innovations, but it cannot repair a globally unsuitable
test norm.

We therefore repeated the locked competition under two operator-adapted
geometries. The ideal test norm whitens the residual by $|A|^{-1}$ and makes
its norm equal to solution error; it is an oracle ceiling because its exact
application requires solving the original problem. The practical variant
inverts $|A|$ only on the 24 eigenvectors nearest resonance and applies one
bulk scale on the complement. This is a standard deflation principle, but its
use here is dictated by the measured inf--sup failure rather than introduced
as an unrelated solver heuristic.

\begin{table}[t]
\centering
\caption{Effect of the residual geometry on 40-direction selection. Entries
are relative $L^2$ solution errors; the practical norm uses 24 deflated
near-resonant modes.}
\begin{tabular}{lrrr}
\toprule
candidate space & raw residual & deflated test norm & ideal test norm\\
\midrule
polynomial & .925 & .550 & .506\\
windowed waves & .964 & .713 & .493\\
polynomial--wave union & .858 & .317 & .209\\
Krylov control & .961 & .950 & .923\\
\bottomrule
\end{tabular}
\label{tab:heterogeneous-helmholtz-norms}
\end{table}

The practical geometry closes about $84\%$ of the raw-to-ideal error gap for
the union: $(.858-.317)/(.858-.209)=.834$. It selects 19 polynomial and 21
wave directions, confirming that neither family alone explains the result.
Thus operator-adapted stability and complementary approximation structure are
both necessary in this test. The remaining questions are whether a local or
matrix-free DPG norm reproduces the deflated result, whether the advantage
persists across media and wavenumbers, and whether its setup cost is competitive.

\paragraph{Locked noncoercive suites.}
We froze the candidate and geometry rules and varied the PDE data. The
Helmholtz suite contains three wavenumbers ($16,24,32$), Gaussian and layered
refractive media, and two source directions, giving twelve cases. The
convection--diffusion suite contains three diffusion parameters
($.02,.01,.005$) and two source locations. Its practical geometry uses a fixed
incomplete factorization rather than an exact inverse. Every run uses forty
directions; reference solves enter only the reported $L^2$ audit.

\begin{table}[t]
\centering
\caption{Locked noncoercive suites. Values are median relative $L^2$ errors,
with ranges in parentheses. The practical mixed space beats both corresponding
single-family controls in all 18 cases.}
\begin{tabular}{lrr}
\toprule
selection rule & Helmholtz (12) & convection--diffusion (6)\\
\midrule
raw mixed & .981 $(.889,1.002)$ & .110 $(.0845,.205)$\\
practical polynomial only & .452 $(.0876,.998)$ & .414 $(.330,.489)$\\
practical structured only & .768 $(.297,.922)$ & 1.000 $(1.000,1.000)$\\
practical mixed & .361 $(.0805,.910)$ & .0746 $(.0540,.150)$\\
ideal mixed & .181 $(.0524,.737)$ & .0455 $(.0355,.0611)$\\
\bottomrule
\end{tabular}
\label{tab:locked-noncoercive-suites}
\end{table}

The result separates geometry from approximation class. Operator-adapted
geometry is insufficient by itself because each single family remains worse
than the union; a rich union is insufficient under the raw geometry. The
Helmholtz mixed spaces contain on average 20.7 polynomial and 19.3 wave
directions. The convection spaces contain 32--36 polynomial and four--eight
outflow-layer directions: layers are ineffective alone but provide useful
conditional innovations after the global response is represented. Thus the
Schur quotient measures complementarity rather than simply preferring a
favored function class.

These suites establish repeatability across controlled PDE variations, not
application-scale generality. The data are procedural, the Helmholtz deflation
uses explicitly computed eigenvectors, and the convection test uses an
assembled incomplete factorization. Three-dimensional benchmarks, local test
computation, and mature DPG and Trefftz controls remain required.

\paragraph{A parabolic width audit.}
To test the width-theoretic motivation beyond a stationary problem, we applied
the same exact Schur-innovation ranking to a locked heterogeneous-diffusion
snapshot family.  Training and audit centers and times were disjoint, every
method retained twelve directions, and the reference snapshots entered only
the post-selection error audit.  The candidates were Fourier functions,
eigenfunctions of the heterogeneous diffusion operator, or their union.

\begin{table}[t]
\centering
\caption{Parabolic snapshot audit at twelve retained directions.  The mixed
space provides the smallest held-out error, showing that operator modes and
Fourier functions contain complementary innovations.}
\begin{tabular}{lrr}
\toprule
candidate functions & training error & audit error\\
\midrule
Fourier & .0110 & .0131\\
operator eigenfunctions & .0196 & .0079\\
mixed & .0048 & .0036\\
\bottomrule
\end{tabular}
\label{tab:cross-class-width-audit}
\end{table}

Table~\ref{tab:cross-class-width-audit} supports a deliberately limited
conclusion.  The Schur criterion selects a compact fixed space for this
parabolic family, and the mixed space reduces held-out error by $72\%$ relative
to Fourier functions and by $54\%$ relative to the operator modes.  This does
not establish a general parabolic convergence advantage; it verifies that the
innovation criterion can identify complementary structure on unseen
parameters when a compact fixed linear representation is appropriate.

\paragraph{Flow-composed enrichment under SUPG stabilization.}
To separate c-refinement from approximate DPG test geometry, we next use a
conforming fine-grid $P_1$ reference discretization of
\[
 -\epsilon\Delta u+b\cdot\nabla u+.05u=f,
 \qquad b(x,y)=(1,.55\sin(2\pi x)),
\]
with streamline-upwind/Petrov--Galerkin stabilization. All reduced trial
spaces are embedded in this common reference space and use the same incomplete
factorization residual transform. The h candidates are hierarchical local
tents, p candidates are smooth cell bubbles, and c candidates have the form
\[
 \chi(x,y)\,\rho\!\left(y-
   \frac{\alpha}{2\pi}\{1-\cos(2\pi x)\}\right),
\]
where the inner argument is a first integral of the prescribed velocity when
$\alpha=.55$. It is obtained from the PDE coefficient, not from the reference
solution. Six cases combine three diffusion parameters and two source
locations; every policy starts from the same nine-dimensional backbone and is
limited to 50 accepted enrichments.

\begin{table}[t]
\centering
\caption{Locked embedded-SUPG enrichment study (six cases). Every row except
the exhausted p-only candidate family has terminal dimension 59. The last three
rows isolate whether the composition follows the prescribed flow.}
\label{tab:supg-flow-composition}
\begin{tabular}{lrrr}
\toprule
policy or ablation & median error & worst error & median cells\\
\midrule
h & .246 & .359 & 166\\
hp & .237 & .305 & 53.5\\
c & .175 & .227 & 16\\
hc & \textbf{.118} & .160 & 62.5\\
hpc & .123 & \textbf{.151} & 34\\
\midrule
hpc, correct flow composition & .123 & .151 & 34\\
hpc, straight composition & .191 & .271 & 32.5\\
hpc, reversed-flow composition & .218 & .290 & 34\\
\bottomrule
\end{tabular}
\end{table}

The paired comparisons are uniform: hc improves upon h in all six cases and
hpc improves upon hp and every single-family policy in all six. Their median
error ratios are $.499$ and $.526$, respectively. Correct flow composition
also wins all six structural ablations; straight and reversed maps increase
median error by factors $1.56$ and $1.78$. Median stability constants remain
of the same order ($.032$, $.030$, and $.035$), so conditioning does not explain
the separation. This is the clearest evidence here that c-refinement can use
known operator geometry to reduce approximation error and mesh subdivision in
a noncoercive problem. It remains a controlled reduced-space experiment on a
fixed fine SUPG reference mesh, not yet a comparison with production adaptive
SUPG software or a proof of asymptotic complexity.

We then replace the abstract h-cells by a conforming adaptive triangular mesh.
The fine $P_1$ grid is retained only as a common integration and independent
audit space.  Each h-action refines a marked triangle conformingly; p and c
functions are independent enrichment coordinates.  Before every solve, exact
discrete redundancies are quotiented by a rank-revealing factorization.  A
proposed h-block or single p/c function is accepted according to transformed
residual reduction per added quotient dimension, subject to a common stability
floor.  Thus the 80-coordinate budget counts represented directions rather
than labeled, possibly redundant functions.

\begin{table}[t]
\centering
\caption{Conforming adaptive-mesh SUPG study on the six locked cases.  Errors,
triangles, and stability ratios are medians; every h-containing policy has
median terminal dimension 80 (79.5 for h).}
\label{tab:adaptive-supg-mesh}
\begin{tabular}{lrrrr}
\toprule
policy & median error & worst error & triangles & stability\\
\midrule
h & .0945 & .286 & 182.5 & .427\\
hp & .0920 & .286 & 158.5 & .423\\
c & .317 & .368 & 32 & .479\\
hc & .0533 & .0631 & 151.5 & .421\\
hpc & \textbf{.0499} & .0747 & \textbf{141} & .428\\
\bottomrule
\end{tabular}
\end{table}

Hpc improves upon h and hp in all six cases; its median error is smaller by
factors $1.89$ and $1.84$, while it uses $23\%$ fewer triangles than h.  Hc
improves upon h in five cases and has the smallest worst-case error.  The
nearly identical median stability ratios show that quotienting has separated
the approximation gain from coordinate redundancy and conditioning.  The
present implementation is deliberately transparent rather than optimized:
after caching operator images and using incremental rank-one Schur tests,
median measured time is $1.43$ seconds for hpc and $.37$ seconds for h.  This is
a $2.62$-fold hpc speedup over the initial explicit implementation, with maximum
terminal-error change below $9,10^{-16}$; candidate auditing remains the main
overhead.

The closed-form flow label is not required.  In a final anti-oracle audit, we
compute it automatically from the supplied velocity coefficient by tracing
the characteristic equation
\[
 \frac{dY}{dx}=\frac{b_y(x,Y)}{b_x(x,Y)},\qquad Y(0)=\eta,
\]
and numerically inverting the inflow-to-current flow map.  The same fixed
profile centers, widths, quotient threshold, and selection rule are then used
without access to a reference solution.  Four prescribed fields include the
original sinusoid, a two-scale field, a state-dependent field, and a field with
variable horizontal speed.  Combining two diffusion values and two source
locations gives 16 locked cases.

\begin{table}[t]
\centering
\caption{Automatic-characteristic audit. Entries are median relative $L^2$
errors over four diffusion/source cases for each velocity field.}
\label{tab:automatic-characteristics}
\begin{tabular}{lrrrr}
\toprule
velocity field & h & hp & hc & hpc\\
\midrule
sinusoidal & .0945 & .0920 & .0517 & \textbf{.0499}\\
two-scale & .0930 & .0924 & \textbf{.0472} & .0505\\
state-dependent & .0957 & .0953 & \textbf{.0471} & .0515\\
variable horizontal speed & .0897 & .114 & \textbf{.0466} & .0525\\
\midrule
all 16 cases & .0931 & .0978 & \textbf{.0476} & .0511\\
\bottomrule
\end{tabular}
\end{table}

The better of hc and hpc improves upon the better of h and hp in all 16 cases,
with median error ratio $.567$.  Hc improves upon h in all 16, and hpc improves
upon both h and hp in all 16.  Median triangle counts are 178.5, 163.5, 154,
and 140 for h, hp, hc, and hpc, while median stability ratios remain
$.487$--$.500$.  This removes the closed-form-coordinate concern: the operator
coefficient $b$, which is part of the PDE data and is already used by SUPG,
is sufficient to construct the c-coordinate algorithmically.  The remaining
generality questions concern recirculating fields without a global inflow
label and higher-dimensional characteristic geometry, not solution leakage.

Theorem~\ref{thm:characteristic-coordinate-error} makes the numerical-flow
dependence explicit.  If the characteristic map remains invertible with
inverse constant $m^{-1}$, its coordinate error is bounded by
$m^{-1}\delta_\Phi+\delta_I$.  Lipschitz composition transfers this error to
the c-family, and discrete inf--sup stability then gives
\[
 \text{total error}
 \ \lesssim\ \text{best h/c approximation}
 +\text{flow-coordinate error}
 +\text{SUPG consistency error}
 +\text{algebraic error}.
\]
The proof also retains exact monotone Schur gain for every accepted nested
enrichment.  This is a conditional quasi-optimality result, not an unconditional
claim that c-refinement always has lower complexity; the latter requires the
solution to possess transported regularity.

We validate the new term independently on the variable-speed field.  With ODE
tolerance $10^{-7}$, increasing the number of inflow labels from 31 to
61, 121, and 241 reduces the maximum coordinate error from
$1.22,10^{-3}$ to $3.11,10^{-4}$, $7.48,10^{-5}$, and
$1.79,10^{-5}$, consistent with second-order inverse interpolation.  Over the
same sweep the audited solution error changes only from $.04970$ to $.04967$.
At 241 labels, tightening ODE tolerance from $10^{-3}$ to $10^{-5}$ reduces
coordinate error from $2.35,10^{-3}$ to $2.71,10^{-5}$; further tightening
has negligible effect.  Thus characteristic error is controlled and the
observed floor comes from the remaining approximation/discretization terms.

\begin{figure}[t]
\centering
\includegraphics[width=.84\textwidth]{figures/characteristic_error_decomposition.pdf}
\caption{Validation of the characteristic-coordinate term in
Theorem~\ref{thm:characteristic-coordinate-error}.  Left: coordinate error
separates inverse-label resolution from ODE tolerance and follows the predicted
second-order interpolation slope once integration error is small.  Right: the
total adaptive solution error reaches a stable floor as the characteristic
term is reduced.}
\label{fig:characteristic-error-decomposition}
\end{figure}

The coordinate estimate becomes operational when combined with the
operator-compatible Schur--Riesz theorem.  We freeze the initial polynomial
space, form the ideal characteristic block using a highly resolved flow map,
and compute its captured residual fraction $\gamma$.  Before solving any
enriched problem, the numerical flow map supplies the block perturbation
$\varepsilon_C$ and quotient SVD supplies $A_C$.  The resulting predicted
residual ratio is
\[
 \rho_{\rm cert}=\sqrt{1-\gamma}
 +\frac{\varepsilon_C}{\sqrt{A_C}}\sqrt{\gamma}.
\]
Table~\ref{tab:operator-compatible-certificate} compares this prospective
quantity with an independent enriched solve.  A spectral cutoff of $.1$
retains 33 stable innovations and certifies even the coarsest computed flow;
with 241 labels its bound $.4806$ is close to the observed ratio $.4786$.
Retaining 56 directions at coarse flow resolution is correctly left
uncertified because transformation error is large relative to the weakest
innovation.  Refining the flow restores certification.  Straight and
deliberately incorrect flow coordinates are rejected.  They may still reduce
the finite-dimensional residual, but the theorem makes no transferable
guarantee for them.

\begin{table}[t]
\centering
\caption{Prospective operator-compatible enrichment audit.  The decision uses
only the frozen baseline, the supplied velocity field, and the numerical flow
map.  Acceptance requires $\rho_{\rm cert}<1$; ``observed'' is computed only
after that decision.}
\label{tab:operator-compatible-certificate}
\begin{tabular}{lrrrrrrr}
\toprule
coordinate & labels & cutoff & rank & $\gamma$ &
$\varepsilon_C/\sqrt{A_C}$ & $\rho_{\rm cert}$ & observed\\
\midrule
computed & 31  & .10 & 33 & .7710 & .1484 & .6089 & .4785\\
computed & 241 & .10 & 33 & .7710 & .0023 & .4806 & .4786\\
computed & 31  & .03 & 53 & .7782 & .4810 & .8952 & .4710\\
computed & 241 & .03 & 53 & .7782 & .0075 & .4775 & .4709\\
computed & 31  & .01 & 56 & .7789 & .9245 & 1.2861 & .4703\\
computed & 241 & .01 & 56 & .7789 & .0144 & .4830 & .4702\\
straight & 241 & .10 & 33 & .7710 & 13.6689 & 12.4804 & .5906\\
wrong flow & 241 & .10 & 33 & .7710 & 14.1400 & 12.8941 & .9435\\
\bottomrule
\end{tabular}
\end{table}

This audit also explains why testing an entire rich neighborhood at once was
previously ineffective.  With all 78 quotient directions retained,
$A_C=1.87\,10^{-7}$ and no flow resolution in the sweep certifies the block.
The discarded tail contributes only about $.009$ of additional captured
residual fraction, yet dominates the perturbation amplification.  Quotient
spectral truncation therefore removes unstable redundancy while preserving
the useful operator-aligned subspace; here it is part of the theorem rather
than an empirically tuned enrichment heuristic.

The same principle extends beyond transport, but its certificate must respect
the operator class.  We test the symmetrizable nonnormal eigenproblem
$-u''+\mathrm{Pe}\,u'+12\sin(2\pi x)u=\lambda u$ with centered differences.
Its diagonal symmetrizer is determined from the two operator off-diagonals.
The c-space applies the corresponding exponential gauge to sine functions;
the p-control uses the same number of untransformed sine functions.  No
eigenvector enters either construction.  Table~\ref{tab:nonnormal-spectrum}
shows the 16-dimensional audit.  The weighted residual from
Theorem~\ref{thm:operator-compatible-spectrum} identifies the target
eigenvalue for every c-space, whereas the p-space loses identification as
nonnormality increases.  The generic Euclidean certificate is much looser,
showing that the operator-induced metric, not enrichment alone, carries the
guarantee.

\begin{table}[t]
\centering
\caption{Nonnormal eigenvalue audit at dimension 16.  ``Error'' is absolute
eigenvalue error; the weighted and Euclidean columns are rigorous enclosure
radii.  The target spectral gap is approximately $33.14$, so a weighted radius
below half this value uniquely identifies the target.}
\label{tab:nonnormal-spectrum}
\begin{tabular}{rrr r r c}
\toprule
$\mathrm{Pe}$ & family & error & weighted & Euclidean & identified\\
\midrule
4  & p & $8.30\,10^{-4}$ & 3.081 & 37.66 & yes\\
4  & c & $8.0\,10^{-6}$  & .021 & .217 & yes\\
8  & p & .00549 & 17.96 & 937.4 & no\\
8  & c & $8.1\,10^{-5}$ & .060 & 2.208 & yes\\
12 & p & .1793 & 184.4 & $1.56\,10^4$ & no\\
12 & c & .00105 & .271 & 18.98 & yes\\
16 & p & 3.579 & 1632 & $2.16\,10^5$ & no\\
16 & c & .00886 & 1.225 & 135.9 & yes\\
\bottomrule
\end{tabular}
\end{table}

The linear enclosure is reliable but not a sharp eigenvalue-error estimator:
its median effectivity for c-spaces is about $1.16\,10^3$.  Its role is
spectral identification.  The complementary Schur enclosure in
Theorem~\ref{thm:complementary-schur-spectrum} repairs this loss by retaining
the unresolved spectral distribution.  We test meshes with 120 and 240
interior points, four potentials (including a discontinuous step and a held-out
smooth potential), four Peclet numbers, and dimensions 8, 16, 24, and 32.
Table~\ref{tab:schur-spectrum} summarizes all 128 matched pairs.  Above the
numerical floor, median upper-bound effectivity is essentially one and its
99th percentile remains below 1.017.  The c-space wins every matched case,
with median c-to-p eigenvalue-error ratio $7.10\,10^{-5}$.

\begin{table}[t]
\centering
\caption{Complementary Schur spectral audit.  Effectivity statistics exclude
28 cases whose eigenvalue error is below $10^{-9}$; all 256 complementary
problems satisfy the required positive complement gap.}
\label{tab:schur-spectrum}
\begin{tabular}{lrrrr}
\toprule
family & matched wins & median error & median effectivity & 99th percentile\\
\midrule
p & 0/128   & $.00507$ & $1.000001$ & $1.0166$\\
c & 128/128 & $9.00\,10^{-7}$ & $1.000000$ & $1.0026$\\
\bottomrule
\end{tabular}
\end{table}

This result distinguishes the source of each improvement.  The exponential
composition produces the approximation gain; the operator-induced metric
turns the nonnormal problem into a self-adjoint one; and complementary Schur
elimination makes the certificate sharp.  None of the three alone supplies
the complete conclusion.

Finally, forming the dense complement is unnecessary.  We apply
$C-\rho I$ through the projected operator
$(I-yy^*)(H-\rho I)(I-yy^*)$ and solve for the Schur correction by conjugate
gradients.  A sparse unit-shifted operator preconditioner regularizes the
target mode and has linear storage.  Equation~\eqref{eq:matrix-free-schur-radius}
adds the iterative residual to the upper certificate.  The complement-gap
input is computed independently by Proposition~\ref{prop:sturm-secular-gap}:
Sturm bisection brackets the first two operator eigenvalues and scalar secular
bisection brackets the first compressed eigenvalue.  Thus no dense complement
or assumed gap enters Table~\ref{tab:matrix-free-schur}.  Four iterations
suffice at every resolution.  At 64,000 unknowns the method uses 13.0 MB,
whereas explicitly storing the complement would require about 31.25 GB; the
computed gap lies in $[33.4602890508,33.4602890586]$.

\begin{table}[t]
\centering
\caption{Fully matrix-free complementary Schur scaling.  Time includes the
Sturm--secular gap certificate, sparse preconditioner construction, and Schur
solve; it excludes construction of the eight-dimensional Ritz vector and the
independent reference spectrum.  Dense memory is the storage avoided.}
\label{tab:matrix-free-schur}
\begin{tabular}{rrrrrrr}
\toprule
$N$ & CG its. & effectivity & time (s) & matrix-free MB & dense MB & saving\\
\midrule
500   & 4 & 1.000004 & .020 & .102 & 1.90 & $18.7\times$\\
4,000 & 4 & 1.000003 & .171 & .815 & 122.0 & $149.8\times$\\
16,000& 4 & 1.000004 & .701 & 3.26 & 1,952.9 & $599.2\times$\\
64,000& 4 & 1.000026 & 2.926 & 13.04 & 31,249.0 & $2,397.1\times$\\
\bottomrule
\end{tabular}
\end{table}

This experiment establishes a complete linear-memory certificate for the
structured tridiagonal class: gap verification, Schur correction, and
algebraic-error control are all matrix-free with respect to the complement.
For general sparse operators, Proposition~\ref{prop:bordered-inertia-gap} supplies the corresponding
dimension-independent formulation.  We test it on a two-dimensional
symmetrized convection--diffusion operator with a variable potential.  A
Until now reverse Cuthill--McKee ordering supplied a baseline.  The final
implementation uses geometric nested dissection, reuses one symbolic analysis
for every bisection shift, and applies SuiteSparse LDL to the bordered
inertia.  The bordered sparse input contains only the original five-point
operator plus two vectors.  Through 4,096 unknowns, an independent projected
eigensolve lies inside every bracket.  The larger cases test factorization
scaling without constructing that reference complement.

\begin{table}[t]
\centering
\caption{Optimized sparse bordered-inertia scaling.  All runs reuse one
symbolic analysis for 38 numeric LDL factorizations and produce brackets of
width $7.45\,10^{-9}$.  Memory is estimated from sparse $L$ and diagonal
storage; speedup compares nested dissection with the RCM baseline.}
\label{tab:bordered-inertia-2d}
\begin{tabular}{rrrrrrr}
\toprule
grid & $N$ & RCM time & ND time & speedup & RCM MB & ND MB\\
\midrule
$32^2$  & 1,024  & .0079 & .0063 & $1.25\times$ & .279 & .202\\
$64^2$  & 4,096  & .0716 & .0400 & $1.79\times$ & 2.12 & 1.09\\
$128^2$ & 16,384 & 1.039 & .292 & $3.56\times$ & 16.47 & 5.61\\
$192^2$ & 36,864 & 5.078 & 1.018 & $4.99\times$ & 55.05 & 15.30\\
$256^2$ & 65,536 & 15.994 & 2.390 & $6.69\times$ & 129.87 & 27.78\\
\bottomrule
\end{tabular}
\end{table}

This closes both the mathematical and prototype implementation gap between the
structured Sturm experiment and general sparse FEM matrices.  At 65,536
unknowns, nested dissection and symbolic reuse reduce time by $6.69\times$ and
factor memory by $4.68\times$, to 2.39 seconds and 27.8 MB.  Sparse fill still
grows with grid size, as expected in two dimensions.  Parallel sparse LDL is a
production extension; the present SuiteSparse LDL kernel is serial, so no
parallel claim is made here.

\begin{figure}[t]
\centering
\includegraphics[width=\textwidth]{figures/adaptive_supg_hpc_meshes.pdf}
\caption{Representative conforming meshes at $\epsilon=.005$ and source
$y=.65$.  All policies use the same 80-coordinate quotient budget.  Flow-based
enrichments allow hc and hpc to retain a coarser mesh in regions where h and hp
must resolve the transported structure geometrically.  The error shown above
each panel is audited in the common fine-grid $L^2$ norm.  Color shows mean
local squared residual energy on a common logarithmic scale; bright cells carry
the largest unresolved contribution.}
\label{fig:adaptive-supg-meshes}
\end{figure}

\begin{figure}[t]
\centering
\includegraphics[width=.88\textwidth]{figures/adaptive_supg_hpc_errors.pdf}
\caption{Adaptive SUPG error behavior. Left: convergence against intrinsic
quotient dimension for the representative strongly convective case. Right:
terminal errors across the locked diffusion and source-location sweep; points
are medians and bands span the two source locations.  Hc and hpc separate from
their polynomial counterparts during refinement and retain the advantage as
diffusion decreases.}
\label{fig:adaptive-supg-errors}
\end{figure}

\begin{figure}[t]
\centering
\includegraphics[width=.92\textwidth]{figures/noncoercive_convergence_cost.pdf}
\caption{Error convergence and measured selection cost for the locked
noncoercive suites. Curves show medians and interquartile ranges. The
operator-adapted geometry improves the convergence trajectory, not only its
terminal point. Reported time includes construction of the 24-mode Helmholtz
deflation or convection incomplete factorization, cached operator images, and
forty selection steps; it excludes the independent reference audit.}
\label{fig:noncoercive-convergence-cost}
\end{figure}

The implementation caches all candidate operator images and vectorizes the
Schur gain calculation. Against the original sequential implementation, it
reproduces every selected direction and terminal error exactly in six parity
tests, with median speedup $2.22$ (range $2.21$--$2.26$). At forty directions,
the median raw/operator-adapted totals are $.160/.185$ seconds for Helmholtz
and $.185/.189$ seconds for convection--diffusion. Cached trial and operator
images require approximately $9.7$ MB and $8.2$ MB, respectively; the 24
Helmholtz deflation vectors add about $.4$ MB. Thus the present small-grid
study shows a substantial implementation speedup and modest incremental cost
for the adapted geometry, but not an end-to-end speed advantage over mature
PDE solvers.

No adaptive-mesh comparison is reported for these suites because all methods
use the same structured finite-difference grid and global candidate functions;
drawing distinct ``with c'' and ``without c'' meshes would be misleading.
Meaningful mesh figures require a local DPG or Trefftz finite-element
discretization in which h, p, and paired c-actions genuinely generate different
trial--test meshes. This remains separate from the function-space mechanism
validated here.

\paragraph{Local optimal-test Helmholtz prototype.}
As a first local test, we discretize a one-dimensional forced Helmholtz problem
in a fine ambient finite-element space and define the trial-to-test map
$Tv=R^{-1}Av$ using the $H^1$ test Riesz matrix. Starting from eight cells, the
adaptive policies may split a cell (h), add a quadratic bubble (p), or add the
real or imaginary part of a windowed local Helmholtz wave (c). Every proposal
is ranked by its quotient residual gain and rejected if the normalized minimum
singular value of the enlarged trial-to-test map falls below $10^{-3}$.

\begin{table}[t]
\centering
\caption{Local optimal-test Helmholtz prototype. Terminal values use forty
actions. The matched row reports the first state below relative $L^2$ error
$.1$.}
\begin{tabular}{lrrrrr}
\toprule
policy & trial dimension & cells & relative error & stability ratio & time (s)\\
\midrule
h, terminal & 47 & 48 & .140 & $7.97\,10^{-3}$ & .587\\
hp, terminal & 47 & 26 & .0921 & $4.40\,10^{-3}$ & .613\\
c, terminal & 23 & 8 & .830 & $2.58\,10^{-2}$ & .084\\
hpc, terminal & 47 & 15 & .0775 & $1.09\,10^{-3}$ & .807\\
hp, first below $.1$ & 46 & 25 & .0927 & $4.53\,10^{-3}$ & .577\\
hpc, first below $.1$ & 30 & 10 & .0919 & $2.23\,10^{-3}$ & .317\\
\bottomrule
\end{tabular}
\label{tab:local-dpg-helmholtz}
\end{table}

\begin{figure}[t]
\centering
\includegraphics[width=.9\textwidth]{figures/local_dpg_helmholtz_mesh_convergence.pdf}
\caption{Local optimal-test Helmholtz convergence and terminal meshes. The c
action enriches a cell without splitting it. At matched error $.1$, hpc uses
35\% fewer trial coordinates and 60\% fewer cells than hp, and is $1.82$
times faster in this prototype. The c-only control fails, showing that local
waves become useful only after polynomial and spatial innovations represent
the complementary response.}
\label{fig:local-dpg-helmholtz}
\end{figure}

This is the first experiment here with genuinely different adaptive meshes,
but it remains a controlled one-dimensional optimal-test model. The fine
ambient test solve is global, its cost is included but not matrix-free, and the
reported stability ratio is discrete. A publishable superiority claim still
requires two-dimensional broken-test DPG or Trefftz-DG, local Fortin control,
and matched mature implementations.

We then locked the local rule across twelve cases: wavenumbers $24,36,48$, two
source locations, and constant or layered media. Under the original $H^1$
test geometry, median terminal hp/hpc errors are $.717/.303$ and hpc is better
in ten cases. The two failures identify separate defects: oscillatory actions
can displace a delayed h-refinement under a fixed budget, and near resonance a
smaller $H^1$ residual need not imply smaller $L^2$ error.

The repair follows the inf--sup diagnosis. We replace the selection norm by an
$L^2$-targeted residual transform that inverts the twelve discrete modes
nearest resonance and uses one bulk scale on their complement. No candidate,
budget, or stability threshold is changed. Median hp/hpc errors fall to
$.238/.0926$, and hpc is better in eleven of twelve cases.

\begin{table}[t]
\centering
\caption{Locked local Helmholtz suite (12 cases). The shadow policy retains
the hp and hpc paths and deploys the one with smaller $L^2$-targeted residual.}
\begin{tabular}{lrrr}
\toprule
geometry and policy & median error & cases below hp & target $.2$ attained\\
\midrule
$H^1$ residual, hp & .717 & --- & 4/12\\
$H^1$ residual, hpc & .303 & 10/12 & 6/12\\
$L^2$-targeted, hp & .238 & --- & ---\\
$L^2$-targeted, hpc & .0926 & 11/12 & ---\\
$L^2$-targeted hp/hpc shadow & .0926 & 11/12, hp fallback in 1 & ---\\
\bottomrule
\end{tabular}
\label{tab:local-dpg-l2-patch}
\end{table}

The shadow decision is not based on the reference error. In every case, the
smaller transformed residual selects the smaller independently audited $L^2$
error. In the sole remaining hpc failure ($k=48$, right source, constant
medium), the transformed residuals are $.521$ for hp and $.584$ for hpc, so
the rule returns hp; their audited errors are respectively $.286$ and $.463$.
This restores the exact hp special case operationally while allowing c when
its conditional innovation is useful. It presently doubles path management;
shared factorization, candidate images, and hp-prefix computations should
reduce that overhead.

\paragraph{Two-dimensional broken-test DPG control.}
To determine whether the preceding conclusions survive a genuine local test
construction, we also solve a constant-velocity transport--reaction problem
in ultraweak form on the unit square. The trial variables comprise cellwise
polynomials and shared skeleton traces. The test space is discontinuous across
cells, and every test-space Riesz problem is solved element by element; no
global residual transform is used. Uniform meshes with $n=4,6,8,10$ provide
the h-sequence, polynomial degrees one and two provide p, and the c-family
contains cell-local flow-coordinate layers. Candidate selection uses the
$p+4$ broken test space. A birth is accepted only if it preserves the discrete
inf--sup threshold, removes at least one percent of the current residual
energy, and independently passes the same gain test in a nested $p+6$ test
space. The exact solution is used only for the final $L^2$ audit.

\begin{table}[t]
\centering
\caption{Prospective 2D broken-test DPG control for degree-two trial
polynomials. The hp and hpc columns use the same mesh and test spaces. The
nested audit accepts c only when its stable residual gain exceeds one percent;
zero accepted modes therefore means that the certified hpc method returns hp
exactly.}
\begin{tabular}{rrrrrr}
\toprule
$n$ & hp DOFs & hp $L^2$ error & accepted c & hpc $L^2$ error & inf--sup ratio\\
\midrule
4  & 240  & .06612 & 1 & .06614 & .04577\\
6  & 540  & .04717 & 0 & .04717 & .03200\\
8  & 960  & .03146 & 0 & .03146 & .02449\\
10 & 1500 & .02308 & 0 & .02308 & .01981\\
\bottomrule
\end{tabular}
\label{tab:broken-test-dpg-transport}
\end{table}

\begin{figure}[t]
\centering
\includegraphics[width=.86\textwidth]{figures/broken_test_dpg_transport.pdf}
\caption{Independent solution error and discrete inf--sup audit for the 2D
broken-test transport problem. Polynomial refinement converges, while the
certified hpc policy increasingly returns hp as the flow layer becomes
resolvable by the polynomial space. The near coincidence of the stability
curves confirms that rejected c-functions are not needed to maintain the
trial--test geometry.}
\label{fig:broken-test-dpg-transport}
\end{figure}

This experiment changes the interpretation of the earlier prototype in two
ways. First, SR certification remains operational with genuinely broken tests
and local Riesz solves. Second, nonpolynomial availability does not imply
nonpolynomial use: on this smooth aligned transport layer, p-refinement soon
absorbs the proposed c-directions. The one coarse accepted direction changes
the independent error by less than $3\times10^{-5}$, so it is not evidence of
application gain. A convincing DPG advantage for c-enrichment must therefore
use a problem whose operator-compatible structure remains expensive for hp,
not merely a transport problem whose layer is already resolved under h.

\paragraph{Conservative channelized-flow positive control.}
We next replace the constant velocity by locked random and SPE10-derived Darcy
fields \cite{christieblunt2001}. Each permeability is resolved on the trial
mesh, the pressure is recomputed, and neighboring cells share one normal flux
per face. The resulting RT0-type reconstruction has discrete divergence below
$3\,10^{-11}$. A c-proposal is a coupled volume--trace block: the global
profile $\phi(P(x))$ and its projected skeleton trace enter together.
Selection uses an adjoint graph norm, a nested richer test space, the frozen
one-percent gain rule, and the frozen global inf--sup floor $10^{-2}$.

\begin{table}[t]
\centering
\caption{All accepted coupled volume--trace enrichments in the 30-case
channelized DPG sweep. The other 26 cases return hp exactly. One accepted
global coordinate in each listed case produces a substantial independent
error reduction.}
\begin{tabular}{lrrrrr}
\toprule
field & $(n,p)$ & inf--sup & hp error & hpc error & reduction\\
\midrule
random--23 & $(4,1)$ & .0211 & .1179 & .0466 & 60.5\%\\
random--23 & $(4,2)$ & .0155 & .0841 & .0511 & 39.3\%\\
SPE10--20  & $(4,1)$ & .0250 & .1116 & .0444 & 60.2\%\\
SPE10--50  & $(6,1)$ & .0231 & .0671 & .0180 & 73.2\%\\
\bottomrule
\end{tabular}
\label{tab:channelized-broken-dpg}
\end{table}

The experiment isolates three necessary ingredients: conservative face fluxes
make the volume and skeleton operators compatible; the coupled proposal moves
the field and trace together; and the inf--sup gate prevents an uninformative
residual decrease from being mistaken for solution-error control. This is an
operator-aligned positive control: the Darcy field is prescribed by the PDE,
the candidate parameters come from a frozen grid, and the manufactured exact
solution enters only the independent audit. It is not yet evidence for an
arbitrary transport-generated family or a comparison with production DPG software.
Its value is diagnostic: four stable global innovations are
accepted and reduce independent error by $39$--$73\%$, whereas all 26
unsupported proposals reduce exactly to hp.

The present search is not yet computationally competitive.  Across the four
accepted cases, total prototype time is $32$--$39$ times the matched hp solve
(median $34.9$).  The median overhead over the 26 rejected cases is $1.00$
because the inf--sup gate terminates most searches before candidate assembly,
although a stable candidate that is subsequently rejected can still incur a
large search cost.  The bottleneck is the dense exhaustive projection of the
global candidate family, not the element-local DPG Riesz solves.  Matrix-free block
application, batched candidate images, and a coarse-to-fine screening rule are
therefore required before an application-scale cost claim is justified.

\paragraph{Natural many-query transport under one DPG operator.}
To remove the manufactured-solution and operator-mismatch limitations, we
finally fix one fine broken-test DPG operator for each of three SPE10 fields.
Twelve broad inflow modes are propagated offline through that operator.  Nine
previously unseen sharp inflows (three locations for each field) are then
solved in equal-dimensional global-polynomial, transported-response, and mixed
spaces.  The full DPG solution supplies the independent reference.  Thus every
reduced space uses the same bilinear form, test norm, right-hand side, and
reference operator; only its trial coordinates differ.  Conditional-residual
greedy selection is identical for all three candidate families.

\begin{table}[t]
\centering
\caption{Natural common-operator DPG study over nine unseen inflows. Entries
are median relative $L^2$ errors; parentheses give the observed range. The
mixed space beats the matched polynomial space in all nine cases at every
budget.}
\begin{tabular}{rrrr}
\toprule
coordinates & polynomial & transported & mixed\\
\midrule
8  & .3722 $(.1677,.7355)$ & .0292 $(.0158,.2374)$ & .0253 $(.0158,.0770)$\\
12 & .3729 $(.1949,.7628)$ & .0224 $(.0058,.4374)$ & .0171 $(.0047,.1315)$\\
16 & .3796 $(.1907,.7557)$ & .0224 $(.0058,.4374)$ & .0129 $(.0057,.0869)$\\
\bottomrule
\end{tabular}
\label{tab:natural-common-operator-dpg}
\end{table}

\begin{figure}[t]
\centering
\includegraphics[width=.98\textwidth]{figures/common_operator_dpg_solution_comparison_n16_spe10-50.pdf}
\caption{Representative SPE10--50 target on the refined $16\times16$ broken-test
mesh.  The full DPG reference, the 12-coordinate polynomial approximation, and
the 12-coordinate SR-selected mixed approximation use the same operator and
test norm.  Shared color limits make the failure of generic polynomials to
follow the channel-distorted front visible; the lower row shows that the mixed
space removes nearly all of that error.  Each discontinuous bilinear field is
evaluated at ten points per coordinate direction inside every element; the
visible interelement jumps therefore belong to the ultraweak trial space and
are not a cell-centre plotting artifact.}
\label{fig:common-operator-dpg-solutions}
\end{figure}

\begin{figure}[t]
\centering
\includegraphics[width=.88\textwidth]{figures/common_operator_dpg_convergence_cost.pdf}
\caption{Accuracy and measured many-query cost.  Left: median error over nine
held-out inflows as the reduced coordinate budget grows.  Right: cumulative
cost of repeated full DPG solves versus twelve offline transported responses
followed by reduced online solves.  The prototype amortizes its offline stage
after approximately twelve to thirteen targets.}
\label{fig:common-operator-dpg-cost}
\end{figure}

At 12 coordinates the median mixed-to-polynomial error ratio is $.0623$; at
16 it is $.0528$.  The median discrete inf--sup ratios at 12 coordinates are
$.0701$, $.0526$, and $.0530$ for polynomial, transported, and mixed spaces,
respectively, so the error separation is not obtained by accepting an unstable
trial space.  Offline construction of twelve transported responses takes
$.381$ seconds in this prototype.  A full DPG solve takes $.0314$ seconds,
whereas the median 12-coordinate online solve takes $.000240$ seconds, about
$131$ times faster.  Ignoring reusable assembly, the measured offline cost is
recovered after approximately 13 target solves.

This is the strongest two-dimensional evidence in the paper.  It tests a
natural zero-source transport boundary-value problem with no manufactured
interior solution, locked candidate families, held-out inflows, matched dimensions,
and one common DPG operator.  A pre-specified resolution check on SPE10--50
repeated the three held-out targets on a $16\times16$ mesh.  At twelve
coordinates the median polynomial and mixed errors were $.2162$ and $.0176$,
respectively, while the corresponding discrete inf--sup ratios were $.0431$
and $.0896$.  The mixed advantage therefore persists when the number of DPG
elements is quadrupled; it is not an artifact of the original $8\times8$
display.  The fixed operator was assembled once and reused exactly across
right-hand sides.  Its twelve offline solves took $20.18$ seconds, a full
target solve $1.69$ seconds, and a reduced online solve $.000924$ seconds.
The scope remains many-query linear transport.  A multi-field fine-mesh audit,
diffusion, independently implemented DPG controls, and matrix-free offline
construction remain open.

To separate memory scalability from approximation scalability, we also replaced
the dense broken-test matrix by an exact sparse block assembly and applied
right-diagonally preconditioned LSQR to the full reference solves.  Sparse and
dense operators agree exactly on the validation mesh; the preconditioner merely
rescales coordinates and leaves the least-squares problem unchanged.  This
permits the same locked SPE10--50 target to be evaluated through the native
$40\times40$ coefficient resolution.  Figure~\ref{fig:common-operator-resolution}
reports the best candidate satisfying the pre-specified discrete inf--sup floor
$.05$ at each resolution.

\begin{figure}[t]
\centering
\includegraphics[width=.96\textwidth]{figures/common_operator_dpg_resolution.pdf}
\caption{Resolution and sparse-cost audit for one locked SPE10--50 target.
Left: best stable error among budgets 8, 12, and 16 for one fixed polynomial
candidate family and twelve fixed transported responses.  Centre: the number of
coordinates accepted before the prospective stability floor rejects every
remaining direction.  Right: sparse full, offline, and reduced online costs.
The transported and mixed spaces improve on generic polynomials at every
resolution, but their errors are not monotone: refining the ambient DPG mesh
reveals channel and front structure not contained in the fixed broad-response
candidate family.  Thus sparse assembly resolves the memory bottleneck, while
resolution-robust enrichment remains an approximation problem.}
\label{fig:common-operator-resolution}
\end{figure}

This experiment prevents two overstatements.  First, a stability certificate
does not by itself make a fixed reduced family complete as the ambient mesh
is refined.  Second, a small decrease in DPG residual among already stable
coordinates need not decrease the independently measured physical error.  The
appropriate next theorem and algorithm must therefore couple the present
Schur--Riesz stability test to a resolution-aware approximation indicator; the
certificate should reject unstable directions, whereas approximation marking
must introduce new localized or sharper transported responses.

We tested that conclusion prospectively rather than tuning further on
SPE10--50.  Twelve boundary centres and the three mesh-scaled widths
$.10,.05,.025$ were frozen at $40\times40$.  Six unseen boundary families---two
fronts of unseen location and width, a Gaussian, a compact bump, a two-front
band, and an oscillatory trace---were crossed with two untouched coefficient
fields (random--23 and SPE10--70).  Table~\ref{tab:locked-resolution-audit}
compares twelve accepted coordinates under the same $.05$ stability floor.

\begin{table}[t]
\centering
\caption{Prospective native-resolution audit over twelve untouched
field--boundary pairs.  ``Wins'' compare the resolution-aware library with the
corresponding column method.  The stability ratio of every accepted
resolution-aware space exceeds $.05$ (median $.122$).}
\begin{tabular}{lrrr}
\toprule
method & median relative $L^2$ error & maximum error & resolution-aware wins\\
\midrule
global polynomial & .9780 & 1.2655 & 12/12\\
broad transported & .2635 & .9919 & 10/12\\
resolution-aware transported & .0602 & .2741 & ---\\
\bottomrule
\end{tabular}
\label{tab:locked-resolution-audit}
\end{table}

The result supports mesh-scaled enrichment but also identifies the remaining
gap.  In two of the twelve prospective cases the enlarged stable library
lowered the DPG residual yet increased the independent $L^2$ error relative to
the broad library (ratios $1.16$ and $1.04$).  Consequently neither the
Schur--Riesz floor nor residual decrease alone orders two already stable trial
spaces by physical error.  This paper is therefore competitive and
stable at native resolution, but not yet monotonically safe.  Closing that gap
requires an error quantity equivalent to the desired physical norm---for
example a verified DPG test norm or a goal-oriented dual correction---rather
than another candidate-selection heuristic.

The required bridge can be stated sharply.  If, for each admissible trial
space $W$, computable constants $0<\beta_W\leq M_W$ satisfy
\[
 \beta_W\|u-u_W\|_{L^2}\leq \|f-Bu_W\|_{V'}
 \leq M_W\|u-u_W\|_{L^2},
\]
then the comparison
\begin{equation}
 \frac{\|f-Bu_{W_{\rm new}}\|_{V'}}{\beta_{W_{\rm new}}}
 <
 \frac{\|f-Bu_{W_{\rm old}}\|_{V'}}{M_{W_{\rm old}}}
 \label{eq:physical-error-acceptance}
\end{equation}
certifies strict reduction of the physical error: the left side is an upper
bound for the new error and the right side a lower bound for the old error.
This implication separates two constants that must not be conflated.  The
reported Schur--Riesz ratio measures coefficient stability within $W$; it is
not, by itself, the global residual--$L^2$ equivalence pair
$(\beta_W,M_W)$.  Even the optimistic substitution of the normalized
Schur--Riesz ratios into Eq.~\eqref{eq:physical-error-acceptance} accepts none
of the thirty audited field--boundary pairs.  The fully rigorous comparison
is therefore non-operational with the current norm.  A near-optimal DPG test
norm or a dual-weighted residual must supply the missing sharp physical
constants.

For the discrete least-squares formulation there is a direct route to an
intrinsic algebraic quantity.  Write $Z=R_V^{-1/2}B$, let $c_W$ be the reduced
coefficient vector, and set $r_W=y-Zc_W$.  A correction representative
$d_W$ satisfies
\begin{equation}
 (Z^*Z)d_W=Z^*r_W, \qquad
 \|[d_W]\|_{H_0}:=\inf_{z\in\ker Z}\|d_W+z\|_{H_0},
 \label{eq:dual-correction}
\end{equation}
where $H_0$ is the volume $L^2$ Gram operator.  This quotient norm equals the
physical finite-element error only when the coefficient-to-function map is
compatible with $\ker Z$.  More generally, computable constants
$0<a_J\leq b_J$ must establish
\[
 a_J\|[d]\|_{H_0}\leq\|Jd\|_X\leq b_J\|[d]\|_{H_0},
\]
where $J$ maps coefficients to trial functions.  Thus a preconditioned Krylov
iterate first estimates an intrinsic quotient correction.  If a quotient
coercivity estimate gives
\[
 \|d_W-d_{W,k}\|_{H_0}\leq
 C_W\|Z^*(r_W-Zd_{W,k})\|,
\]
then
\[
 E_{W,k}^{\pm}=\|d_{W,k}\|_{H_0}
 \mathbin{\pm} C_W\|Z^*(r_W-Zd_{W,k})\|
\]
is a certified quotient interval (with the lower endpoint truncated at zero).
A physical discrete interval is $[a_JE^-_{W,k},b_JE^+_{W,k}]$, and
Theorem~\ref{thm:discrete-continuum-enclosure} adds $\eta_{\rm disc}$ to reach
the PDE solution.  This is the precise point at which quotient--Riesz
stability re-enters: it must control both the inverse normal operator and the
coefficient-to-function map, rather than merely screen the selected
candidate family.

\begin{figure}[t]
\centering
\includegraphics[width=.92\textwidth]{figures/physical_dual_diagnostic.pdf}
\caption{Physical dual-correction diagnostic at $40\times40$.  Solid curves
are Krylov estimates of the new-to-incumbent $L^2$ error ratio; dotted lines
are independent audited ratios and the black line is equality.  Clear
beneficial enrichments are identified within two to five iterations.  The
large harmful enrichment is rejected immediately.  The subtle case oscillates
near equality until roughly fifty iterations.  The figure evaluates a point
diagnostic; it does not establish quotient-to-function compatibility or the
remainder and discretization bounds required for a PDE-error certificate.}
\label{fig:physical-dual-diagnostic}
\end{figure}

We also tested the most direct remainder bound rather than assuming it was
sharp.  Lanczos estimates of the smallest nonzero singular value of the
right-scaled full DPG operator give global factors $C_W=5.17\,10^3$ for
random--23 and $8.91\,10^4$ for SPE10--70.  These make every correction
interval overlap.  Restricting the Rayleigh calculation to the
residual-generated Krylov spaces is still insufficient: the Ritz factor grows
as weak directions enter the space, and the intervals for the incumbent and
proposal overlap at every tested iteration.  Thus the point estimator in
Figure~\ref{fig:physical-dual-diagnostic} is predictive, but the naive
worst-case certification is vacuous.

This negative result identifies the appropriate Schur object.  Partition the
correction into physical volume and auxiliary trace coefficients,
$d=(d_0,d_\Gamma)$, and eliminate $d_\Gamma$ by least squares.  If
$P_\Gamma$ denotes orthogonal projection onto $\operatorname{range}Z_\Gamma$,
the physically relevant operator is
\[
 S_0=Z_0^*(I-P_\Gamma)Z_0
\]
on the $H_0$ quotient.  Its generalized coercivity constant measures the
least observable physical volume change after optimal trace compensation.
Unlike the full smallest singular value, it does not penalize weak trace or
gauge directions that carry no $L^2$ volume error.  A sharp certificate must
therefore estimate the complementary Schur pair $(S_0,H_0)$, preferably by a
matrix-free generalized Lanczos procedure.  The preceding results show why
neither the full-spectrum constant nor a residual-only Ritz restriction can
replace it.

We implemented this complementary calculation matrix-free, using a reusable
regularized pseudoinverse for the singular trace normal operator.  The trace
projection defects were $3.95\,10^{-7}$ and $1.35\,10^{-6}$ and the symmetry
defects were below $7.1\,10^{-15}$.  The physical remainder factors fell from
$5.17\,10^3$ to $5.71$ on random--23 and from $8.91\,10^4$ to $875.5$ on
SPE10--70.  Table~\ref{tab:physical-schur-decisions} shows the first iteration
at which the resulting intervals separate.

\begin{table}[t]
\centering
\caption{Physical complementary-Schur certification at $40\times40$.  The
decision uses only disjoint correction intervals.  The iteration entries are
unpreconditioned/preconditioned.}
\begin{tabular}{llrrr}
\toprule
field and target & decision & quotient ratio & audited ratio & iterations\\
\midrule
random--23, Gaussian & prefer resolution-aware & .177 & .177 & 1 / 1\\
random--23, oscillatory & prefer broad & 1.163 & 1.163 & 5 / 5\\
SPE10--70, compact & prefer resolution-aware & .109 & .087 & 50 / 10\\
SPE10--70, narrow front & prefer resolution-aware & .951 & 1.036 & 50 / 20\\
\bottomrule
\end{tabular}
\label{tab:physical-schur-decisions}
\end{table}

Here ``prefer'' has a precise but deliberately limited meaning: it identifies
the candidate having the smaller certified correction to the frozen full DPG
solution.  The last column gives unpreconditioned complementary-Schur CG and
cell-block preconditioned CG iterations, respectively.  The preconditioner is
assembled from the independent $4\times4$ physical volume blocks; it changes
only the Krylov trajectory, while the remainder is bounded with the same
complementary-Schur coercivity constant.  It therefore cannot weaken the
certificate.  It is neutral on the already easy random-field cases and reduces
the two SPE10 decisions from 50 iterations to 10 and 20, respectively.

The narrow-front row exposes the boundary of the quotient statement.  The
resolution-aware approximation has the smaller quotient-minimal correction,
but its directly audited volume error relative to the chosen full DPG
representative is $3.6\%$ larger than that of the broad approximation.  Thus a
Schur correction becomes a physical finite-element error certificate only
after quotient compatibility and the trial-norm inf--sup constants have been
verified.  Theorem~\ref{thm:discrete-continuum-enclosure} then adds the second
bridge, from the full discrete solution to the continuum PDE solution.  Until
both bridges are present, the operational rule must abstain rather than
conflate quotient distance, discrete physical error, and continuum error.

We next tested whether the discretization bridge can be made operational
without inserting an empirical calibration constant.  On each cell, write
$Z_c$ and $Z_e$ for the coarse and enriched whitened trial-to-test matrices.
The minimal-norm local Fortin map is computed from
\[
 Z_c^*P_K=Z_e^*,\qquad P_K=(Z_c^*)^\dagger Z_e^*.
\]
Thus $\|P_K\|_2$ is a directly computable local Fortin constant, while
$\|Z_e^*-Z_c^*P_K\|_2/\|Z_e^*\|_2$ tests the commuting equation itself.
Table~\ref{tab:fortin-audit} reports the native-grid calculation.  The defects
are at roundoff, so the local Fortin operators exist.  Their norms, however,
show that test degree five is not sharp enough on the heterogeneous field.
Raising only the test degree from five to seven reduces the worst constant
from $28.3$ to $3.53$ on SPE10--70, without changing the trial space.

\begin{table}[t]
\centering
\caption{Local Fortin audit at $40\times40$.  The enriched test degree is
eight.  Max and p95 are cellwise norms of the minimal commuting map.}
\label{tab:fortin-audit}
\begin{tabular}{lrrrr}
\toprule
field & coarse degree & max $C_F$ & p95 $C_F$ & commuting defect\\
\midrule
random--23 & 5 & 8.56 & 7.35 & $5.5\,10^{-14}$\\
random--23 & 7 & 2.53 & 2.32 & $3.3\,10^{-14}$\\
SPE10--70  & 5 & 28.30 & 7.18 & $1.3\,10^{-12}$\\
SPE10--70  & 7 & 3.53 & 2.05 & $3.6\,10^{-13}$\\
\bottomrule
\end{tabular}
\end{table}

This is a theory-driven discretization decision rather than a refinement
heuristic: the trial approximation is held fixed and the Fortin constant
selects the test resolution required for reliability.  It does not yet close
the continuum certificate.  The degree-eight enriched residual is only
$1.5\%$--$4.5\%$ larger than the degree-five residual in the audited cases,
but a small observed increment is not a bound on the uncomputed test-space
tail.  Until data oscillation or a saturation constant controls that tail,
$\eta_{\rm disc}$ in Theorem~\ref{thm:discrete-continuum-enclosure} remains
unavailable and the correct decision is still to abstain.

We also tested the tempting two-level alternative in
Proposition~\ref{prop:two-level-saturation}.  Degrees five, seven, and eight
were solved with the same trial space and audited against an independently
computed $160\times160$ conservative finite-volume reference.  Table~
\ref{tab:saturation-audit} shows that no uniform $q<1$ exists for this locked
family.  Test enrichment improves the variational geometry but need not
decrease physical $L^2$ error, so an observed hierarchy cannot replace a
reliability proof.

\begin{table}[t]
\centering
\caption{Falsification of an assumed physical-error saturation factor.  An
entry counts cases in which the richer-test solution has smaller audited
$L^2$ error; each field contains six fixed inflows.}
\label{tab:saturation-audit}
\begin{tabular}{lrrrr}
\toprule
field & improved 5 to 7 & improved 7 to 8 & range $q_{5,7}$ & range $q_{7,8}$\\
\midrule
random--23 & 0/6 & 0/6 & 1.004--1.021 & 1.001--1.011\\
SPE10--70  & 5/6 & 2/6 & .609--1.003 & .952--1.041\\
\bottomrule
\end{tabular}
\end{table}

The conclusion is narrower and stronger than selecting a favorable test
degree: local Fortin control is now computable, but the physical
discretization term must be supplied by a saturation-free transport
majorant---for example an equilibrated graph-residual reconstruction with a
proved trace bound.  That construction, rather than further degree tuning, is
the remaining continuum-certification problem.

Theorem~\ref{thm:graph-residual-majorant} provides that construction at the
functional level.  We first validate it in one dimension, where the trace
lifting is explicit, before introducing the additional approximation error of
a two-dimensional $H(\operatorname{div})$ lifting.  A degree-120 graph-Riesz
solve supplies an independent numerical value of the dual norm.  The lower
column in Table~\ref{tab:graph-majorant-validation} is the Galerkin Riesz norm;
the upper column is the saturation-free majorant in
Eq.~\eqref{eq:graph-residual-majorant}.  Every enclosure is valid and the
upper effectivity is $1.0112$.

\begin{table}[t]
\centering
\caption{Controlled one-cell validation of the saturation-free
graph-residual majorant.  The independently resolved dual norm is
$1.1603090054$.}
\label{tab:graph-majorant-validation}
\begin{tabular}{rrrr}
\toprule
test degree & Galerkin lower bound & graph majorant & effectivity\\
\midrule
1 & 1.1356990 & 1.1733137 & 1.011208\\
2 & 1.1601894 & 1.1733137 & 1.011208\\
4 & 1.1603074 & 1.1733137 & 1.011208\\
8 & 1.1603090 & 1.1733137 & 1.011208\\
\bottomrule
\end{tabular}
\end{table}

The same test on a two-dimensional tensor cell, with an affine
$H(\operatorname{div})$-compatible transport lifting and an independently
resolved degree-18 graph-Riesz reference, gives upper effectivity $1.01254$;
degrees one through eight all remain enclosed.  This confirms that the
majorant is not a one-dimensional integration artifact.  It is still a
controlled cell test: the natural heterogeneous experiment must compute the
lifting from its assembled face residuals rather than from a prescribed
polynomial lifting.

That distinction is material on the natural problem.  We attempted to lift
the four assembled degree-one face traces by a single tensor polynomial of
degree five on every cell.  Exact edge-moment compatibility failed on all
$1600$ cells.  Across the six inflows, median relative defects range from
$6.3\,10^{-4}$ to $.0127$ on random--23 and from $.0019$ to $.0192$ on
SPE10--70; the corresponding worst defects range from $.408$ to $.490$ and
from $.490$ to $.541$.  Increasing the polynomial least-squares accuracy does
not remove the obstruction: independently represented edge traces need not be
the boundary trace of one globally polynomial scalar lifting.

This negative result determines the admissible construction.  A transport
graph space permits piecewise characteristic liftings whose discontinuity
interfaces are tangent to $\boldsymbol\beta$; such jumps do not create a
distributional $\boldsymbol\beta\cdot\nabla$ term.  Equivalently, one may
compute the minimum-norm graph-Riesz lifting directly.  Either route must
carry a characteristic-integration or local-solve error bound.  Substituting
an incompatible polynomial fit would give a smaller number but would destroy
the guarantee.

A first characteristic implementation traces every quadrature point in all
$1600$ cells to valid inflow and outflow faces.  Unsplit spatial quadrature
initially suggested a mere integration problem.  Flow-coordinate quadrature
removes that ambiguity: it integrates the cell areas to within
$1.6\,10^{-14}$, but on the random--23 Gaussian case the majorant still grows
from $.5872$ to $.6690$, $.7188$, and $.7543$ at orders 6, 10, 14, and 18.
The divergence is concentrated on arbitrarily short characteristics near
corners whose incident edge polynomials have incompatible limiting values.

This is a structural discretization failure, not a quadrature defect.  The
prototype uses independently labelled flux-weighted $L^2$ edge coefficients;
they need not belong to the trace quotient of the continuous transport graph
space and therefore need not possess a finite compatible lifting.  A rigorous
continuum certificate requires either a conforming transport-trace space or
the proper quotient trace norm.  Characteristic tuning cannot repair an
inadmissible trace space.

We implemented the sufficient conforming repair.  Every degree-one edge trace
is now generated from shared vertex values, prescribed inflow values use the
same nodal interpolation, and the minimum-residual problem is re-solved in
that trace subspace.  Flow-coordinate quadrature then becomes stable.
Table~\ref{tab:conforming-trace-majorant} reports deliberately different
profiles and coefficient fields.  No characteristic fails, and the majorant
lies only $.007\%$--$.105\%$ above the independently enriched residual.

\begin{table}[t]
\centering
\caption{Natural 2D graph majorant after conforming trace repair.  Change is
the absolute order-10 to order-14 quadrature change; area defect audits the
flow-coordinate Jacobian.}
\label{tab:conforming-trace-majorant}
\begin{tabular}{lrrrr}
\toprule
field/profile & majorant & effectivity & change & area defect\\
\midrule
random--23/Gaussian & .1825275 & 1.000070 & $2.5\,10^{-7}$ & $1.6\,10^{-14}$\\
random--23/oscillatory & .3355416 & 1.000197 & $7.1\,10^{-13}$ & $1.6\,10^{-14}$\\
SPE10--70/Gaussian & .8924271 & 1.000141 & $1.1\,10^{-6}$ & $8.1\,10^{-6}$\\
SPE10--70/compact & .5634784 & 1.001052 & $5.4\,10^{-8}$ & $8.1\,10^{-6}$\\
SPE10--70/narrow front & 2.5944850 & 1.000179 & $3.4\,10^{-6}$ & $8.1\,10^{-6}$\\
\bottomrule
\end{tabular}
\end{table}

This closes the structural lifting failure and supplies an operational
saturation-free graph-energy majorant on the natural grid.  We next remove
the remaining characteristic-integration uncertainty.  Since the repaired
edge data share their vertex values, each cell admits a bilinear lifting.
The resulting majorant integrand is polynomial, so the selected
positive-weight Gauss rule is exact in exact arithmetic.  Only
floating-point evaluation and summation remain.  Under the standard IEEE
model, with unit roundoff $\epsilon$ and a conservative explicit operation
count $n$, the computed nonnegative sum is enclosed using
$\gamma_n=n\epsilon/(1-n\epsilon)$ and outward propagation through the
square root.  This encloses evaluation for the fixed assembled coefficients;
the stored floating-point coefficients define the computed approximant.
Crucially, no estimate of their distance from the exact least-squares
minimizer is required.  The reconstructed residual certifies the realized
approximant itself, so early termination of the algebraic solver appears
directly as a larger residual rather than through a condition-number-amplified
coefficient-error estimate.  Only uncertainty in forming the stored
coefficients, as opposed to algebraic nonconvergence, would require an
additional enclosure.

\begin{table}[t]
\centering
\caption{Polynomial-exact bilinear lifting after conforming trace repair.
The majorant is the order-8 value, effectivity uses its IEEE upper endpoint,
and $\delta_{\rm fp}$ is its conservative rounding radius.  Orders 8--12
agree to the displayed precision.}
\label{tab:bilinear-interval-majorant}
\begin{tabular}{lrrr}
\toprule
field/profile & majorant & effectivity & $\delta_{\rm fp}$\\
\midrule
random--23/Gaussian & .19865785 & 1.08845 & $2.3\,10^{-8}$\\
random--23/oscillatory & .40143134 & 1.19660 & $2.3\,10^{-8}$\\
SPE10--70/Gaussian & .98560758 & 1.10457 & $2.3\,10^{-8}$\\
SPE10--70/compact & .77047425 & 1.36879 & $2.3\,10^{-8}$\\
SPE10--70/narrow front & 2.97469383 & 1.14675 & $6.8\,10^{-8}$\\
\bottomrule
\end{tabular}
\end{table}

Table~\ref{tab:bilinear-interval-majorant} makes the
sharpness--auditability tradeoff explicit.  The characteristic lifting in
Table~\ref{tab:conforming-trace-majorant} is sharper; the bilinear lifting
removes quadrature truncation and supplies the formal upper endpoint.  It is
already a physical error certificate in the induced energy norm.  For any
other stated physical norm, Corollary~\ref{cor:physical-norm-conversion}
divides the endpoint by an independently proved continuous $\beta_X$.

The fixed upper value is informative: the explicit affine trace lifting is
reliable but is not the minimum-norm representation of the unresolved
functional.  Quotienting null representations, or equivalently minimizing
$\|a\|^2+\|b\|^2$ subject to the trace constraint, can only sharpen it.  Thus
the theorem closes the saturation gap; the conforming trace construction
makes its natural two-dimensional graph-energy specialization operational.

\paragraph{Controlled two-dimensional extension.}
We finally transfer the same ingredients to a heterogeneous square-domain
Helmholtz problem on a $31\times31$ ambient grid. A nested quadtree supplies h
actions, smooth cell bubbles supply p, and the c-family contains windowed local
waves, partition-of-unity global waves, and source- and inclusion-centered
fundamental solutions. All locations derive from the prescribed source and
coefficient field. The first local-bubble family failed because every
wave vanished on coarse-cell boundaries and could not modify transmitted
phase; adding functions that cross the coarse skeleton repaired that defect.

The remaining issue is test geometry. With 20 near-resonant modes, hpc error
is $.815$. Increasing the transformed subspace gives $.659$, $.496$, and
$.401$ for 40, 80, and 160 modes, approaching the ideal-norm value $.391$.
At the common 105-dimensional budget, the 160-mode results are $.589$ for h,
$.653$ for hp, $.484$ for c, and $.401$ for hpc. The corresponding hpc space
retains the initial 16 cells, while hp uses 88; hpc selects global and local
wave directions together with 14 polynomial bubbles and no h-split.

\begin{table}[t]
\centering
\caption{Controlled 2D Helmholtz sensitivity to the number of explicitly
treated near-resonant modes. All hpc rows use the same 105-dimensional budget.}
\begin{tabular}{lrr}
\toprule
test geometry & hpc relative $L^2$ error & terminal cells\\
\midrule
20-mode transform & .815 & 22\\
40-mode transform & .659 & 19\\
80-mode transform & .496 & 16\\
160-mode transform & .401 & 16\\
ideal transform & .391 & 25\\
ideal hp control & .550 & 58\\
\bottomrule
\end{tabular}
\label{tab:local-2d-helmholtz}
\end{table}

\begin{figure}[t]
\centering
\includegraphics[width=.92\textwidth]{figures/local_2d_helmholtz_mesh_convergence.pdf}
\caption{Controlled 2D local optimal-test experiment using the 160-mode
transform. The enriched policy represents the propagated field through
polynomial and oscillatory functions on the initial mesh, whereas hp expends
its budget on 88 cells. This is evidence of representation compression, not
yet a scalable DPG result: the explicit resonant subspace is too large and the
ambient test computation is global.}
\label{fig:local-2d-helmholtz}
\end{figure}

The 2D test therefore confirms the mechanism but also sharpens the remaining
gap. A small global deflation adequate in one dimension is not adequate in two;
the effective fragile subspace grows. Patchwise broken tests, local spectral
compression, or matrix-free recycling must approximate this geometry without
forming 160 global modes.

We tested the most direct localization: overlapping additive-Schwarz test
solves on $6\times6$ through $16\times16$ patches. One-level patch norms fail,
with hpc errors between $.919$ and $.971$ even when their internal residual is
small, because independent Dirichlet patches do not see global cavity modes.
A two-level construction that is exact on 20 global resonant modes improves
the $16\times16$-patch error to $.594$, but remains well above $.401$; raising
the coarse dimension to 40 gives $.628$, showing that the greedy path need not
improve monotonically under an altered approximate norm. Thus naive patch
localization is not the scalable solution. The missing result is a
spectrally-equivalent broken test norm, likely requiring impedance transmission
conditions or an energy-minimizing coarse space rather than independent
Dirichlet patch inverses.

Replacing artificial Dirichlet boundaries by discrete impedance transmission
conditions confirms this diagnosis. With 20 global coarse modes, the physical
Robin scale gives error $.482$ at 80 actions, compared with $.594$ for the
Dirichlet two-level norm; a sweep over relative impedance parameters
$.25,.5,1,2,4$ selects one. Increasing the coarse correction to 40 modes gives
$.469$. This remains above global-160 error $.401$, but stores approximately
$8.8\,10^4$ factor and coarse entries rather than $1.54\,10^5$ entries for the
global eigenvectors. Its present sequential implementation takes 28 seconds,
so it is a stability/accuracy improvement rather than a speed result. Patch
solves are independent and can be parallelized, while the remaining gap points
to an operator-compatible coarse space rather than further tuning of the Robin
coefficient. Two controlled falsification tests make that qualification
important. A coercive energy-minimizing construction produces errors $.913$
and $.947$ (without and with incident-phase constraints), despite driving its
own transformed residual as low as $.084$. Rank-20, 40, and 80 Arnoldi spaces
recycled from a nearby right-hand side give errors $.848$, $.731$, and $.833$.
Thus neither low surrogate energy nor one previous Krylov trajectory ensures
coverage of the near-resonant directions governing a new load. Together with
the positive Robin--spectral result, these controls isolate the open
requirement: cross-load spectral equivalence on the fragile operator subspace.

We tested that requirement prospectively on four unseen loads.  Twelve training
loads generated resolvent snapshots; exact snapshot redundancies were removed,
and the remaining directions were ordered by solution energy per unit
operator-image energy.  Table~\ref{tab:crossload-coarse} reports the frozen
audit.  The rank-20 learned space is substantially better than either
single-load construction, but does not match the spectral control.  Protecting
30 or 35 resonant modes and filling the remaining 40-dimensional budget with
conditional snapshot innovations also fails to improve upon 40 spectral modes.
The result is informative rather than positive: average amplification on the
training loads is not a substitute for uniform coverage of the fragile
subspace.

\begin{table}[t]
\centering
\caption{Locked four-load audit of coarse test geometries.  Errors are relative
$L^2$ errors after the same 80-action hpc budget.  Training resolvents and audit
loads are disjoint.}
\label{tab:crossload-coarse}
\begin{tabular}{lrrrr}
\toprule
coarse geometry & dimension & median error & worst error & stored entries\\
\midrule
cross-load resolvent & 20 & .507 & .558 & 68,763\\
30 spectral + 10 innovations & 40 & .489 & .534 & 89,183\\
35 spectral + 5 innovations & 40 & .497 & .504 & 89,183\\
Robin--spectral control & 40 & .472 & .488 & 87,583\\
global spectral control & 160 & .400 & .415 & 153,760\\
\bottomrule
\end{tabular}
\end{table}

\paragraph{Principal numerical finding.}
The clearest result is heterogeneous Darcy flow.  Across eleven prescribed
SPE10 layer--load cases, pc improves upon p at equal dimension in every case,
with a maximum relative energy-error reduction of $41.6\%$.  For Model 1,
23 pc coordinates attain error $.514$, while the first h-adaptive P1 and P2
spaces below that threshold have intrinsic dimensions 79 and 156.  The staged
hc calculation reaches error $.432$ with 52 coordinates: the retained
coefficient-adapted functions represent global channel structure, while h
refinement acts on the remaining localized residual.  The stochastic Darcy,
high-contrast, Allen--Cahn, and transport examples examine how far this
mechanism extends; they provide supporting rather than principal evidence.

\endsecondarynumericsD

\section{Numerical evidence: automatic transfer and safe fallback}
\label{sec:main-numerics}
The principal experiments test the complete width-to-decision chain rather
than a single candidate family.  In every case an incumbent space is frozen
before automatic transfer modes are generated.  Construction uses the PDE
operator and a declared training load or trajectory family; the marking and
audit families are disjoint.  Comparisons are dimension matched, and the safe
variant returns the incumbent whenever the prospective transfer test fails.

Table~\ref{tab:auto-transfer-crossclass} gives the main result.  For layered
Helmholtz, twelve randomized quotient--Green modes replace twelve additional
expert-designed directions and approximately halve unseen-source error over
five sketches.  For coercive Darcy, energy-transfer modes generated from a
fixed load quadrature outperform h/p/compositional selection over ten unseen
permeability fields.  Maxwell supplies the necessary counterexample: automatic
trajectory modes help on the coarser compatible discretization, but are worse
on the finer level, where the safe rule rejects all five proposals and returns
the incumbent exactly.

\begin{table}[H]
\centering
\caption{Matched-dimension automatic-transfer audit.  Errors are means with
sample standard deviations where random repetitions are available.  The safe
column reports how often the automatic proposal is deployed.}
\label{tab:auto-transfer-crossclass}
\begin{tabular}{lrrrrr}
\toprule
problem & runs & incumbent & automatic & safe & accepted\\
\midrule
layered Helmholtz & 5 & $.8646$ & $.4313\pm.0388$ & $.4313\pm.0388$ & 5/5\\
coercive Darcy & 10 & $.4764\pm.0434$ & $.3285\pm.0351$ & $.3285\pm.0351$ & 10/10\\
Maxwell, coarse level & 5 & $.3145\pm.0709$ & $.2620\pm.0957$ & $.2888\pm.0767$ & 2/5\\
Maxwell, fine level & 5 & $.1478\pm.0347$ & $.1915\pm.0485$ & $.1478\pm.0347$ & 0/5\\
\bottomrule
\end{tabular}
\end{table}

The same selection mechanism was frozen and tested in three dimensions.  The
heterogeneous Helmholtz discretization has $20^3=8000$ unknowns.  Thirty-two
training loads construct operator-response modes, 16 independent marking
loads select among matched spaces, and 24 untouched loads per seed audit the
physical solution error.  Selection and coefficient fitting use independent
rank-128 and rank-256 adjoint Petrov sketches, respectively, formed by solving
$B^*z_j=s_j$ for isotropic probes $s_j$; no audit solution enters either step.
Table~\ref{tab:helmholtz-3d-adjoint} reports averages over three locked seeds.

\begin{table}[H]
\centering
\caption{Locked 3D heterogeneous Helmholtz audit.  The spaces
have equal dimension.  Errors are mean relative physical solution errors over
72 untouched right-hand sides.}
\label{tab:helmholtz-3d-adjoint}
\begin{tabular}{rrrr}
\toprule
dimension & spectral-polynomial & operator-adapted & reduction (\%)\\
\midrule
8  & 1.0060 & .7535 & 25.1\\
12 & 1.0085 & .6664 & 33.9\\
16 & 1.0135 & .6096 & 39.8\\
24 & 1.0212 & .5379 & 47.3\\
\bottomrule
\end{tabular}
\end{table}

The independent rule selected the operator-adapted space or a mixed space
with the same numerical span in all 12 seed--dimension cases.  Worst test
error fell by $12.8$--$33.8\%$.  A single block adjoint solve constructs both
sketches in 1.26 seconds and 46.9 MiB; an operator-keyed cache reduces reuse
time to .009 seconds without changing either sketch.  The reduced
online solve required approximately $.35$ ms per right-hand side.  This tests
transfer beyond the two-dimensional development.  To separate dimension
efficiency from a weak baseline, we next compare with native MFEM
p-refinement on tetrahedral meshes.  All approximations are prolonged into a
common p4 or p5 reference space and audited in its mass norm.  For each
wavenumber, Table~\ref{tab:helmholtz-3d-mfem-target} reports the first tested
configuration below the declared mean-error target.  Since p is discrete,
MFEM sometimes overshoots the target; the attained errors are shown.

\begin{table}[H]
\centering
\caption{Matched-target native MFEM comparison in three dimensions.  The ratio
is MFEM dimension divided by SR dimension.  Online times are per right-hand
side.  Memory is the deployed sparse operator and factor for MFEM, and the
basis, fitting sketch, and reduced operator for SR.  Offline time excludes
common mesh/operator assembly and all audit work; it includes the selected
MFEM factorization or the SR reference factorization, training solves,
compression, independent sketches, and selection.}
\label{tab:helmholtz-3d-mfem-target}
\small
\begin{tabular}{rrrrrrrrr}
\toprule
$k$ & target & method & dim. & ratio & error & off. (s) & on. (ms) & MiB\\
\midrule
8  & .15 & MFEM-p2 & 1331 &      & .1299 & .018 & .156 & 4.09\\
   &     & SR      &   24 & 55.5 & .1478 & 8.36 & .237 & 33.43\\
12 & .25 & MFEM-p3 & 4913 &      & .1049 & .167 & 1.044 & 37.66\\
   &     & SR      &   24 & 204.7 & .2354 & 8.39 & .231 & 33.43\\
17 & .20 & MFEM-p4 & 6859 &      & .1386 & .174 & 1.640 & 54.11\\
   &     & SR      &   96 & 71.4 & .1930 & 6.57\,(1.30) & .099 & 47.46\\
\bottomrule
\end{tabular}
\end{table}

The result identifies the intended regime rather than universal superiority.
At $k=8$, native p2 is more accurate, faster, and smaller.  At $k=12$ and 17,
the compact response space is approximately $4.5$ and $1.4$ times faster
online and uses $11.2\%$ and $12.3\%$ less deployment memory, but has larger
offline cost.  For $k=17$, the first number is cold construction and the
parenthesized number reuses the operator-keyed adjoint cache.  The measured
online savings amortize the cold and cached differences after approximately
4,200 and 730 right-hand sides, respectively.

We also performed a native 3D adaptive control on a
manufactured heterogeneous Helmholtz family.  The exact solution is used only
for the final physical-error audit.  MFEM's recovered-flux indicator marks
elements, after which native nonconforming h-refinement and variable element
orders perform h-only, p-only, or alternating h--p updates.  Each MFEM path is
adapted separately to its held-out right-hand side; SR must instead use one
shared response space trained independently of all three audit seeds.
Table~\ref{tab:helmholtz-3d-adaptive-hp} reports the first state below relative
$L^2$ error $.15$, or the terminal state when the target is missed.

\begin{table}[H]
\centering
\caption{Native 3D residual-marked MFEM refinement versus a
shared SR space at target error $.15$ (three locked seeds).  ``Pass'' counts
audits reaching the target.  The alternating h--p row follows MFEM's standard
marked p/h iteration pattern; it is a control, not a claim to represent every
specialized hp decision rule.  For SR, ``cold'' is the first offline number and
includes assembly of the previously unseen discrete PDE operator, generation
and solution of the training loads, response compression, adjoint-sketch
construction, and cache writes, with no reusable numerical state.  The two
parenthesized numbers cache only the fixed operator and then the complete
operator/response/adjoint package, reducing offline cost from $1.013$ s to
$.365$ s ($2.8\times$ faster) and $.008$ s (approximately $127\times$ faster),
respectively.}
\label{tab:helmholtz-3d-adaptive-hp}
\begin{tabular}{lrrrrrr}
\toprule
method & pass & mean dim. & mean error & off. (s) & on. (ms) & MiB\\
\midrule
adaptive h & 3/3 & 878 & .0967 & 1.101 & .1080 & 2.15\\
adaptive p & 3/3 & 1109 & .1180 & 1.279 & .2105 & 6.17\\
adaptive h--p & 3/3 & 5072 & .1009 & 2.225 & 2.6090 & 55.96\\
SR shared space & 3/3 & 48 & .1092 & 1.013\,(.365/.008) & .0158 & 9.61\\
\bottomrule
\end{tabular}
\end{table}

For MFEM, offline time is the complete residual-marked adaptive path for one
right-hand side and online time is its terminal direct solve.  For SR, offline
time constructs one space and cached adjoint geometry shared by the entire
family; online time is the reduced solve for one new right-hand side.  The SR
offline entries report a genuinely cold construction, construction of a new
response space with the fixed PDE operator cached, and full reuse of the
operator, response basis, and adjoint geometry, respectively.  These
columns therefore expose, rather than conceal, the different amortization
models.

This control strengthens the dimension-efficiency claim: the shared SR space
meets the target with about $18$ times fewer coordinates than adaptive h, the
best polynomial control by dimension, and its reduced online solve is about
$6.8$ times faster.  It does not imply universal computational superiority.
For a single right-hand side, adaptive h stores less data and has lower
construction cost; SR amortizes its response basis and adjoint geometry across
a family.  Moreover, alternating marked h and p updates over-refine this smooth,
localized family.  A comparison with a specialized hp smoothness classifier
and mature parallel adaptive-hp implementation remains future work.
The reported SR construction uses vectorized nodal evaluation followed by the
consistent mass matrix.  Relative to repeated elementwise Python quadrature
callbacks, vectorized load formation first reduced the high-frequency offline
time from 45.69 to 8.66 seconds.  Batched independent sketches and
operator-keyed caching then reduce it to 6.57 seconds cold and 1.30 seconds
when the fixed operator geometry is reused.  The three-seed physical error is
$.1930\pm.0220$, consistent with the earlier result.  These reductions are
implementation improvements, not changed selection rules or relaxed audits.

The Maxwell physical diagnostics support the same interpretation.  On the
coarse level, automatic transfer reduces mean Poynting-flux error from $.3513$
to $.2301$ and sensor-trace error from $.2842$ to $.1338$.  On the fine level,
the incumbent already gives lower energy error; rejection therefore preserves
its result.  A redundant-proposal control has zero quotient--Riesz lower bound,
negative marking gain, and also triggers exact fallback.  These outcomes test
selection rather than universal superiority of one representation.

Across the 25 prospective decisions in Table~\ref{tab:auto-transfer-crossclass},
the marking gate predicts the lower-error audit space in 23 cases.  Every one
of the 17 accepted transfers improves independent audit error, so there are no
false acceptances; the two errors are conservative rejections on coarse
Maxwell.  This is evidence that the functional gate is predictive rather than
post-hoc, but the sample is still too small to estimate a universal error rate.

The remaining examples isolate individual components of the framework.  The
transport study tests compositional discovery and construction error; the
nonnormal study tests sharp complementary certification; and the sparse
two-dimensional study tests matrix-free evaluation.  Broader finite-element,
DPG, scaling, and falsification results appear in
the Supplementary Material.

The bulk theorem changes only experiments that repeatedly choose among a
candidate pool.  To avoid relabeling unrelated bounds, Table~\ref{tab:bulk-applicability}
states the treatment of every empirical class in the paper.  The two locked
iterative suites are rerun under bulk selection with cost and memory
instrumentation.  One-shot construction tests, error-bound audits, and sparse
linear-algebra scaling retain their original protocols because they have no
successive candidate-pool decision.
\begin{table}[H]
\centering
\caption{Applicability of the bulk-coverage certificate to the empirical
evidence.  ``Unchanged'' means mathematically inapplicable, not omitted.}
\label{tab:bulk-applicability}
\begin{tabular}{p{.29\linewidth}p{.18\linewidth}p{.43\linewidth}}
\toprule
empirical class & status & reason and reported quantities\\
\midrule
iterative Helmholtz and 2D Darcy candidate selection & rerun &
$\alpha_m$, $\theta_m$, contraction, batch size, time, memory, outer cycles\\
one-shot compositional discovery and prospective map acceptance & unchanged &
constructs or audits one block; there is no successive pool to cover\\
nonnormal eigenvalue and complementary-Schur error bounds & unchanged &
certifies a fixed approximation rather than selecting refinements\\
sparse bordered-inertia and matrix-free scaling & unchanged &
measures implementation of the fixed-space certificate\\
historical PDE scope and falsification studies in the appendix & scalar ablation &
retained to test candidate richness and failure modes, but not presented as
results of the final bulk-certified algorithm\\
\bottomrule
\end{tabular}
\end{table}

\subsection{Variational discovery and certification of transported innovations}
We begin with the proposal stage.  For an advection problem, numerical characteristics are computed solely from
the supplied velocity field.  Profiles composed with that map form the
candidate block.  To test the new calculus rather than insert the known
characteristic amplitude, we hide that amplitude from the optimizer and
maximize $G(\theta)$ on five locked pairs of source locations.  The assembled
operator remains available; the hidden amplitude and held-out source
locations are used only for audit.

Table~\ref{tab:main-composition-discovery} shows that quotient-gain discovery
recovers the true amplitude $.55$ with mean absolute error $.0169$, versus
$.0664$ for an otherwise matched objective that omits projection against the
current operator image.  On held-out sources its residual ratio $.5470$ is
close to the operator-reference value $.5460$ and below the unprojected
control's $.5507$; straight coordinates give $.7585$.  The analytic derivative
in Eq.~\eqref{eq:compositional-gain-gradient} agrees with a centered finite
difference to $1.9\,10^{-8}$ relative error.  Thus quotient geometry
materially improves recovery of the transferable structural parameter, while
the modest residual difference limits the predictive-performance claim.

\begin{table}[H]
\centering
\caption{Compositional discovery over five locked training-source pairs and
fifteen held-out evaluations.  The true flow amplitude $.55$ is hidden during
optimization.  Parentheses give standard deviations.}
\label{tab:main-composition-discovery}
\begin{tabular}{lrrr}
\toprule
method & selected amplitude & amplitude MAE & held-out residual ratio\\
\midrule
quotient gain & $.5371\;(.0153)$ & $.0169$ & $.5470\;(.0952)$\\
unprojected control & $.5420\;(.0799)$ & $.0664$ & $.5507\;(.0951)$\\
operator reference & $.5500$ & $0$ & $.5460\;(.0975)$\\
straight coordinate & $0$ & $.5500$ & $.7585\;(.0470)$\\
\bottomrule
\end{tabular}
\end{table}

\secondarynumericsF
We then enlarge the search to three continuous parameters---flow amplitude,
phase, and transverse scale---and the four discrete depths $1,\ldots,4$.
Each depth is optimized separately on a fixed rank-16 quotient and the depths
are compared by exact gain.  Across three locked source splits,
Table~\ref{tab:main-multivariate-discovery} shows that quotient discovery
again gives the best learned held-out ratio, $.5824$, compared with $.5871$
for unprojected discovery and $.8961$ for a matched polynomial block.  It
remains close to the operator reference, $.5778$.

The parameter audit is deliberately more cautious.  Quotient discovery
selects depths $1,1,2$ across the three splits, while transverse scale
sometimes reaches the search boundary.  Hence the useful functional range is
more identifiable than the individual compositional parameters.  This agrees
with quotient invariance but provides no evidence that extra depth is
uniformly advantageous.

\begin{table}[H]
\centering
\caption{Mixed discrete--continuous discovery.  Ratios are averaged over nine
held-out evaluations.  The depth column reports the three independently
selected depths.}
\label{tab:main-multivariate-discovery}
\begin{tabular}{lrr}
\toprule
method & selected depths & held-out residual ratio\\
\midrule
quotient discovery & $1,1,2$ & $.5824$\\
unprojected control & $4,1,1$ & $.5871$\\
operator reference & $1,1,1$ & $.5778$\\
straight coordinate & $1,1,1$ & $.7908$\\
polynomial control & --- & $.8961$\\
\bottomrule
\end{tabular}
\end{table}

To isolate whether depth can itself provide compression, we finally use a
controlled one-dimensional reaction--diffusion operator whose potential
contains an unknown recursively composed tent coordinate.  The solution has
the same structural depth together with additional oscillatory and localized
components that are absent from every candidate block.  The assembled
operator and load are supplied to the method; the generating depth and exact
solution are retained only for audit.  Six candidate directions are compared
at each depth $1,\ldots,6$ by quotient gain.

The correct depth is selected in all 15 combinations of five generating depths
and three mismatched loads.  Table~\ref{tab:main-recursive-depth} also gives a
rank-matched compression audit.  The median number of additional sine
directions required to match the six-function compositional error grows from
$6$ at depth one to $192$ at depth four and exceeds $192$ at depth five.
This is a controlled demonstration of depth-dependent representation
compression, not evidence that natural PDEs generally possess the required
recursive structure.

\begin{table}[H]
\centering
\caption{Controlled recursive-depth audit.  Every row summarizes three loads;
the selected block always contains six directions.}
\label{tab:main-recursive-depth}
\begin{tabular}{rrrr}
\toprule
true depth & selected depths & mean error ratio
& median p-directions to match\\
\midrule
1 & $1,1,1$ & $.3908$ & $6$\\
2 & $2,2,2$ & $.2784$ & $12$\\
3 & $3,3,3$ & $.2369$ & $48$\\
4 & $4,4,4$ & $.2053$ & $192$\\
5 & $5,5,5$ & $.1446$ & $>192$\\
\bottomrule
\end{tabular}
\end{table}

As a natural-coefficient check, we next use eight lognormal multiscale
permeability fields that are not generated recursively.  In one-dimensional
Darcy flow the operator supplies the harmonic coordinate
$\xi(x)\propto\int_0^x k(s)^{-1}\,ds$.  Dyadic approximations $\xi_d$ are
genuinely successive compositions: because each is monotone, the incremental
map $\Phi_d=\xi_d\circ\xi_{d-1}^{-1}$ gives
$\xi_d=\Phi_d\circ\cdots\circ\Phi_1$.  No solution enters this construction.

The smallest depth attaining 99\% of the best training gain is seven for
seven fields and eight for one.  On 24 held-out load--field pairs, six
functions composed with the selected coordinate attain mean energy-error
ratio $.308$, compared with $.609$ for six polynomial directions, $.878$ for
six residual-marked local hats, and $.295$ for the exact harmonic-coordinate
reference.  This confirms practical operator-coordinate reuse, but not
exponential compression: a depth-seven dyadic map stores roughly 256
resistance values.  A separate offline--online audit makes the tradeoff more
precise.  The ten-dimensional composed space uses, on average, 6.6 times
fewer trial coefficients than a retrospectively error-matched polynomial
control and 11.1 times fewer than a uniform-h control.  Its intrinsic stored
state is smaller by factors 12 and 35.  Coordinate search adds about $.0085$
seconds once; at 10,000 reused right-hand sides the measured total times are
$.0274$, $.0300$, and $.0408$ seconds.  The classical controls attain lower
held-out errors, so this establishes amortized compression rather than
accuracy superiority.  Detailed results appear in
Appendix~\ref{app:natural-darcy-depth}.

The sharper test is whether the residual can decide what should follow the
first coordinate block.  On eight additional locked fields, we repeat a
solve--estimate--enrich cycle with a conforming hierarchical finite element
control.  Its h-candidates are nested dyadic hats; its p-candidates are
element-supported Legendre bubbles on the dyadic hierarchy.  At every cycle,
the same operational Schur gain per retained direction ranks h, p, and
harmonic-coordinate c-candidates.  No held-out load enters this decision.
At the common dimension 51, the mixed policy attains held-out energy ratio
$.0864\pm.0144$, compared with $.1011\pm.0279$ for c alone,
$.2131\pm.0422$ for hierarchical hp, and $.2211\pm.0313$ for p alone.  It
beats hp on all eight fields and c alone on seven.  Its average composition is
12.6 c-, 6.8 h-, and 28.6 p-directions.  Thus the rule retains ordinary finite
element actions rather than replacing them.  Figure~\ref{fig:darcy-residual-policy}
shows the complementarity: operator coordinates remove coefficient-driven
structure, after which h and p resolve the remaining spatial and smooth error.
\begin{figure}[t]
\centering
\includegraphics[width=.96\linewidth]{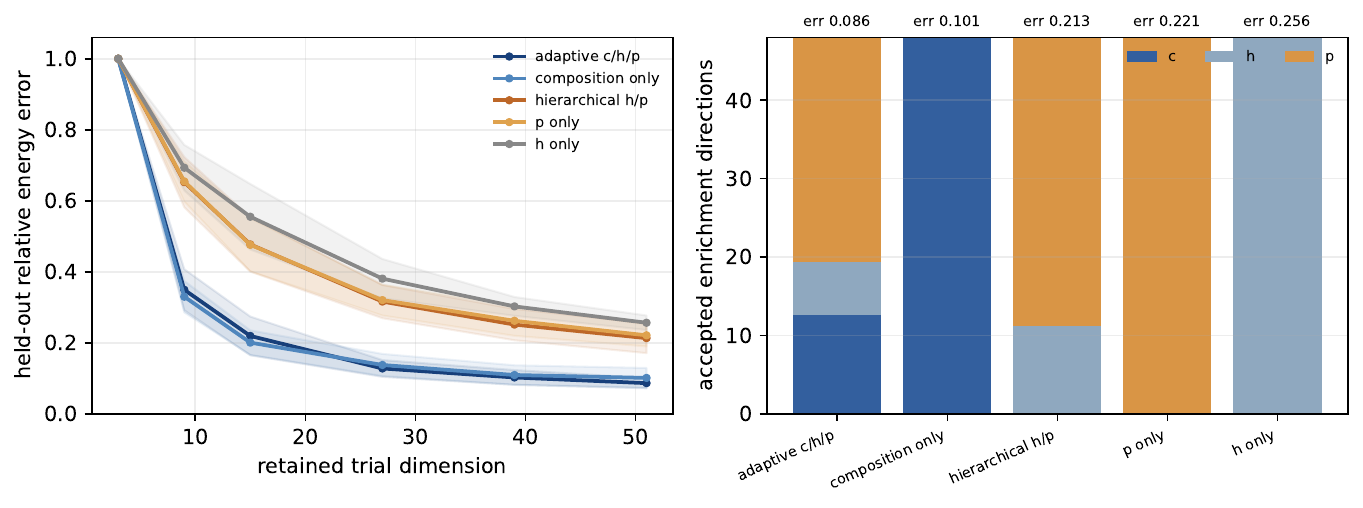}
\caption{Prospective residual selection against a conforming hierarchical hp
control on eight locked lognormal Darcy fields.  Left: mean held-out energy
error; shading is one sample standard deviation.  Right: mean composition of
the 48 accepted enrichment directions.  Every c/h/p decision uses only
pilot-load quotient--Schur gain.}
\label{fig:darcy-residual-policy}
\end{figure}

A stricter two-dimensional replication uses eight new lognormal fields, eight
space-filling pilot well pairs, eight disjoint audit pairs, and exactly 48
scalar enrichments after the same nine-dimensional coarse Q1 space.  The
c/h/p policy attains $.4711\pm.0426$, versus $.7492\pm.0413$ for hierarchical
hp, and wins all eight paired comparisons.  Pure c attains
$.4690\pm.0350$, statistically indistinguishable at this scale; c/h/p is
better on five fields.  Thus the unified policy retains the large
operator-coordinate advantage over hp, but this 2D experiment does not show
that mixing improves upon pure composition.  Its average 48-direction space
contains 31.6 c-, 2.1 h-, and 14.2 p-directions.  A pilot-only incumbent guard
still prefers the mixed path, so the small remaining difference is consistent
with cross-load selection rather than algebraic instability.  Full results
appear in Appendix~\ref{app:natural-darcy-depth}.

We also test the natural first remedy on a third, entirely new locked suite: a
variance-penalized lower-confidence Schur score across pilot loads.  It moves
the selected space toward composition but does not improve transfer.  The
result rules out a generic variance penalty as the missing ingredient and
shows that predictive mixing requires either a richer load model or an
independent task-level validation mechanism.

Proposition~\ref{prop:operator-covariance-gain} supplies that richer model.
On a fourth locked suite, an 81-point well quadrature represents the admissible
load covariance without PDE solves.  Across 24 new fields, covariance-ranked
c/h/p reduces held-out error from $.4963$ for pilot-mean mixing to $.4586$;
the paired improvement is $.0376$ with 95\% confidence interval
$[.0261,.0491]$.  Pure c attains $.4630$, statistically indistinguishable
from covariance c/h/p.  Operator covariance therefore closes the harmful
mixing gap and retains 9.5 p- and 2.5 h-directions on average, but does not yet
establish superiority over pure composition.

A locked structural-regime test supplies the missing architectural evidence.
With unchanged candidates and score, the selected mixture changes with the
coefficient operator: smooth fields use 46.4 of 48 c-directions and no h;
combined channel--inclusion fields use 31.8 c, 4.1 h, and 12.1 p.  Relative to
pure c, expected-gain selection improves channel error by $.0217$ (95\%
interval $[.0102,.0332]$) and combined-field error by $.0123$
($[.0038,.0208]$).  Local inclusions primarily activate cell-supported
p-bubbles rather than h-hats, correcting the simplistic expectation that
every localized coefficient feature calls for h-refinement.

The minimum-residual extension gives the corresponding non-coercive result.
Across 12 heterogeneous Helmholtz cases, covariance-selected polynomial and
wave directions attain mean held-out solution error $.3422$, versus $.4429$
for pure polynomials, $.6315$ for pure waves, and $.9512$ when covariance is
computed in the raw residual geometry.  The stable transformed geometry beats
both pure families in 10 and 12 cases, respectively, and raw covariance in all
12.  This confirms that the expected-gain principle is not restricted to
coercive energy spaces, while also showing that its solution relevance depends
on a stable trial--test norm.

The computational conclusions differ.  On a fixed channel--inclusion Darcy
coefficient, reduced covariance c/h/p solves are 10--17 times faster per new
right-hand side and amortize construction after roughly 2,100--4,000 solves.
For Helmholtz, the present dense 24-mode residual transform makes a reduced
solve 2.5--4.1 times slower than a sparse direct triangular solve on the tested
grids.  Thus the coercive formulation already has a repeated-solve use case;
the non-coercive formulation requires a matrix-free or localized test-space
map before it can claim computational speedup.
\endsecondarynumericsF

After discovery, the current operator image is projected out and the
innovation cutoff determines the retained rank, while the characteristic-map
error supplies $\varepsilon_C$.  The prospective bound is
\[
 \rho_{\rm cert}=\sqrt{1-\gamma}
 +\frac{\varepsilon_C}{\sqrt{A_C}}\sqrt{\gamma}.
\]
Table~\ref{tab:main-transport} reports the decisive cases.  With 241 map
labels, a rank-33 block has predicted residual ratio $.4806$, close to the
independently observed $.4786$.  The larger rank-56 block is rejected when
the map is coarse and accepted after the map is refined.  Incorrect
coordinates are rejected.  Hence the same theorem selects approximation rank
and the accuracy required of the operator coordinate.

\begin{table}[H]
\centering
\caption{Prospective transport certificate.  Acceptance requires
$\rho_{\rm cert}<1$; the observed ratio is evaluated only after the decision.}
\label{tab:main-transport}
\begin{tabular}{lrrrrrr}
\toprule
coordinate & labels & rank & $\gamma$ &
$\varepsilon_C/\sqrt{A_C}$ & $\rho_{\rm cert}$ & observed\\
\midrule
computed & 31  & 33 & .7710 & .1484 & .6089 & .4785\\
computed & 241 & 33 & .7710 & .0023 & .4806 & .4786\\
computed & 31  & 56 & .7789 & .9245 & 1.2861 & .4703\\
computed & 241 & 56 & .7789 & .0144 & .4830 & .4702\\
straight & 241 & 33 & .7710 & 13.6689 & 12.4804 & .5906\\
wrong flow & 241 & 33 & .7710 & 14.1400 & 12.8941 & .9435\\
\bottomrule
\end{tabular}
\end{table}

\subsection{Sharp certification for a nonnormal operator}
Having tested discovery, we hold the candidate construction fixed and examine
the sharpness of certification.  We consider a symmetrizable
convection--diffusion eigenproblem.  The
operator supplies an exponential gauge and a weighted inner product in which
the problem becomes self-adjoint.  The composed block uses gauged sine
functions; the control uses the same number of ordinary sine functions.  No
eigenvector enters either construction.  Complementary Schur elimination
then retains the unresolved spectral distribution rather than replacing it by
a single worst-case gap.

Across 128 matched tests, the composed space is more accurate in every case.
More importantly for the present paper, the resulting upper error bound is
sharp: its median effectivity is $1.000000$ and its 99th percentile is
$1.0026$ for the composed space; see Table~\ref{tab:main-spectrum}.  The result
separates three contributions: composition supplies approximation power, the
operator metric removes artificial nonnormality, and Schur elimination makes
the certificate sharp.

\begin{table}[H]
\centering
\caption{Complementary Schur audit over 128 matched nonnormal eigenproblems.
Effectivity is the certified upper bound divided by observed eigenvalue error.}
\label{tab:main-spectrum}
\begin{tabular}{lrrrr}
\toprule
family & matched wins & median error & median effectivity & 99th percentile\\
\midrule
ordinary & 0/128   & $.00507$ & $1.000001$ & $1.0166$\\
composed & 128/128 & $9.00\,10^{-7}$ & $1.000000$ & $1.0026$\\
\bottomrule
\end{tabular}
\end{table}

\subsection{Two-dimensional sparse scalability}
The preceding one-dimensional calculation exposes the certificate cleanly but
does not test its algebraic implementation at scale.  The final example therefore
moves to two dimensions while leaving the mathematical criterion unchanged. We
apply bordered inertia to a two-dimensional, five-point,
symmetrized convection--diffusion operator with variable potential.  The
method evaluates the complementary spectral gap using the original sparse
matrix and two vectors; it never constructs a dense orthogonal complement.
Nested dissection and one reused symbolic LDL analysis give the results in
Table~\ref{tab:main-sparse-2d}.

At 65,536 unknowns, 38 shifted factorizations require 2.39 seconds and 27.8 MB.
Relative to reverse Cuthill--McKee ordering, this is a $6.69\times$ time and
$4.68\times$ memory improvement.  The experiment does not claim a new sparse
direct solver.  It shows that the sharp complementary certificate is an
operator calculation that remains sparse in two dimensions, rather than a
dense proof device restricted to the 1D prototype.

\begin{table}[H]
\centering
\caption{Sparse bordered-inertia certificate for a two-dimensional operator.
Each run uses 38 numerical LDL factorizations with one symbolic analysis.}
\label{tab:main-sparse-2d}
\begin{tabular}{rrrrrrr}
\toprule
grid & $N$ & RCM time & ND time & speedup & RCM MB & ND MB\\
\midrule
$32^2$  & 1,024  & .0079 & .0063 & $1.25\times$ & .279 & .202\\
$64^2$  & 4,096  & .0716 & .0400 & $1.79\times$ & 2.12 & 1.09\\
$128^2$ & 16,384 & 1.039 & .292 & $3.56\times$ & 16.47 & 5.61\\
$192^2$ & 36,864 & 5.078 & 1.018 & $4.99\times$ & 55.05 & 15.30\\
$256^2$ & 65,536 & 15.994 & 2.390 & $6.69\times$ & 129.87 & 27.78\\
\bottomrule
\end{tabular}
\end{table}

As a separate check of Theorem~\ref{thm:bulk-coverage}, we reused the locked
polynomial--wave candidate pools from four heterogeneous Helmholtz cases.  At
each of six cycles, the joint Schur innovation was formed before selection and
the smallest greedy batch reaching $50\%$ of its gain was retained.  The
coverage and bulk fractions were not tuned by case.  Table~\ref{tab:bulk-audit}
shows that coverage remained positive but became markedly smaller in the
layered medium.  Thus the diagnostic distinguishes a weaker library from an
inadequate selection rule.  Across all 24 cycles, the predicted and observed
squared-residual ratios agreed to at most $7.8\,10^{-16}$; batch sizes ranged
from 7 to 35 (mean 17.7).  Iterating the measured factors as in
Corollary~\ref{cor:finite-run-complexity} gives six-cycle residual envelopes
$.1603$, $.1601$, $.2830$, and $.3215$ for the four locked cases; each equals
the observed product of contraction ratios to rounding.  The full-pool calculation is not free: mean
coverage-plus-selection time is 334 ms per cycle and the peak certificate
workspace is 18.5 MiB in this dense implementation.

\begin{table}[H]
\centering
\caption{Locked bulk-coverage audit.  Cost is joint coverage plus selection
time per cycle; memory is peak certificate workspace, excluding 4.02 MiB of
candidate-image storage shared by all rows.}
\label{tab:bulk-audit}
\begin{tabular}{lrrrrrr}
\toprule
medium & angle & min./mean $\alpha_m$ & batch & 6-cycle envelope & cost (ms) & peak MiB\\
\midrule
Gaussian & $-0.55$ & $.142/.497$ & 17.7 & $.1603$ & 329 & 18.5\\
Gaussian & $ 0.40$ & $.123/.489$ & 16.7 & $.1601$ & 320 & 18.5\\
layered  & $-0.55$ & $.062/.351$ & 18.8 & $.2830$ & 347 & 18.5\\
layered  & $ 0.40$ & $.061/.324$ & 17.5 & $.3215$ & 341 & 18.5\\
\bottomrule
\end{tabular}
\end{table}

Theorem~\ref{thm:mixed-near-oracle} adds a stronger question than contraction:
does the selected mixed sequence use its limited number of coordinates nearly
as well as the best admissible sequence?  We tested this on eight independent
Darcy coefficient fields using a deliberately enumerable prospective pool.
Four screening loads, disjoint from the target load, selected four h, four p,
and four operator-derived candidates subject to the fixed pair-correlation
gate $.92$.  For $k=3$, all $\binom{12}{3}=220$ competitors were refitted to
obtain the exact oracle, while all subsets through size $2k=6$ were enumerated
to obtain $A_{2k}$.  A declared $1\%$ bounded perturbation was then applied to
the gain rankings; its exact magnitude supplied the theorem's $\delta_m$.

\begin{table}[H]
\centering
\caption{Exact oracle audit of three-block mixed refinement over eight Darcy
fields.  Ratios are relative to the best three-block gain in the fixed
prospective pool.  Values are mean $\pm$ sample standard deviation, with the
observed range in parentheses.}
\label{tab:mixed-near-oracle-audit}
\begin{tabular}{lcc}
\toprule
quantity & result & interpretation\\
\midrule
$A_{2k}$ & $.0872\pm.0157$ $(.0669,.1169)$ & restricted joint innovation\\
$B_1$ & $1.0000$ & normalized one-block upper bound\\
$\gamma=A_{2k}/B_1$ & $.0872\pm.0157$ & certified interaction ratio\\
greedy/oracle gain & $.9916\pm.0208$ $(.9405,1.0000)$ & observed efficiency\\
certified/oracle gain & $.0704\pm.0154$ $(.0491,.0996)$ & theorem lower bound\\
accumulated perturbation penalty & $1.90\,10^{-5}\pm1.74\,10^{-5}$ & declared ranking error\\
bound satisfied & $8/8$ & independent coefficient fields\\
\bottomrule
\end{tabular}
\end{table}

The experiment validates the theorem but also identifies its present
conservatism.  Greedy selection captured $99.2\%$ of oracle gain on average,
whereas the restricted lower spectrum certified only $7.0\%$.  The guarantee
is therefore non-vacuous and uniformly correct, but not yet a sharp predictor
of realized efficiency.  The gap comes principally from minimizing a joint
Riesz eigenvalue over every six-block subset, rather than from the controlled
assembly penalty.  This distinction is useful: it points to localized or
path-dependent restricted spectra as the appropriate route to a sharper
certificate, without changing the selected algorithm or its incumbent
fallback.

The certificate must also justify its computational price.  A comparison at
unequal errors would confound approximation power with cost, so we first fix
the held-out relative-energy tolerance $.41$.  The hp control starts from a
conforming tensor $Q_2$ space on $4\times4$ cells and admits conforming h
($Q_1$ on $8\times8$ cells) and p ($Q_3$ on $4\times4$ cells) functions.  We
give this control the favorable retrospectively calibrated stopping index on
its otherwise pilot-selected path.  The SR implementation uses the locked
screened bulk rule: individual gains index a $48$-direction subpool and an
exact Schur calculation certifies that subspace.  Table~\ref{tab:bulk-hp-computation}
reports five-run medians.

\begin{table}[H]
\centering
\caption{Matched-accuracy computational audit on the locked $63\times63$
Darcy operator.  Both methods attain relative error approximately $.405$.
Break-even is measured against a sparse factorization and subsequent LU
backsolves.  Working memory includes cached candidates and operator images;
deployment memory retains only the online representation.}
\label{tab:bulk-hp-computation}
\resizebox{\textwidth}{!}{%
\begin{tabular}{lrrrrrrrr}
\toprule
method & dim. & error & cycles & offline (s) & online/RHS ($\mu$s) & break-even RHS & working MB & deploy MB\\
\midrule
conforming hp & 110 & $.4057$ & 61 & $.1120$ & $5.24$ & 1492 & $19.6$ & $3.42$\\
screened Schur--Riesz & 68 & $.4050$ & 6 & $.1438$ & $3.65$ & 1872 & $30.0$ & $2.09$\\
\bottomrule
\end{tabular}}
\end{table}

\begin{figure}[H]
\centering
\includegraphics[width=.92\textwidth]{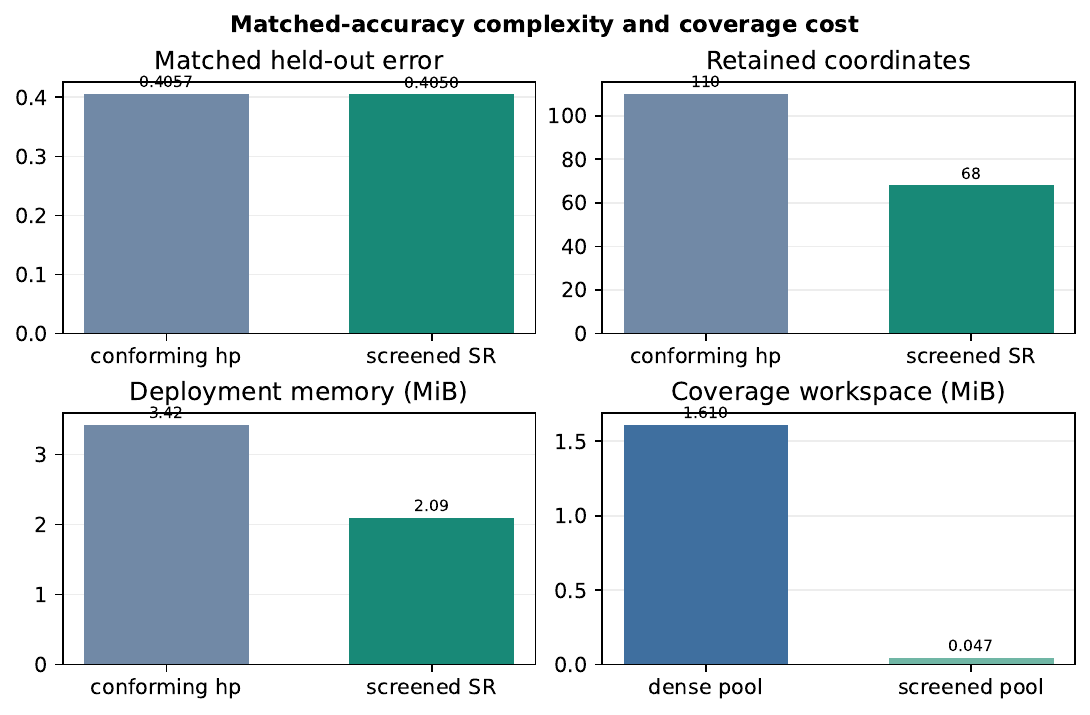}
\caption{The matched-accuracy comparison in
Table~\ref{tab:bulk-hp-computation}, together with the dense-to-screened
coverage ablation.  Accuracy is fixed before comparing representation,
memory, online work, and offline certification.}
\label{fig:bulk-hp-computation}
\end{figure}

At matched accuracy, SR uses $38.2\%$ fewer coordinates, $38.9\%$ less
deployment memory, and $30.1\%$ less measured online time per right-hand side.
Its present offline implementation is $28.4\%$ slower and uses $53.1\%$ more
working memory because it materializes the full mixed candidate family before
screening.  Relative to hp, this extra setup amortizes after approximately
21,000 right-hand sides; relative to the sparse full solve, the separate
break-even counts are 1,492 and 1,872.  These timings are implementation
evidence, whereas the representation counts give the cleaner complexity
comparison.

The coverage implementation has a separate, controlled ablation.  Replacing
the full candidate-pool Schur matrix by the certified $48$-direction subpool
reduces dense-bulk offline time from $.3593$ to $.2697$ s ($24.9\%$) and
coverage workspace from 1.61 MiB to .047 MiB ($97.1\%$).  Total working memory
falls only $6.5\%$ because the present code still materializes the candidate family
and all operator images.  Lazy tree traversal and contiguous cached blocks
would replace that storage by the active subpool; this is the natural
matrix-free implementation implied by the subspace lower bound, but is not
claimed as measured performance here.  Screening also changes the selected
space, increasing error from $.3725$ to $.4050$, so its savings are not a
free accuracy-preserving speedup.

More precisely, with $n$ fine unknowns, $M$ available candidates, screened
block size $s$, and retained dimension $r$, a dense pool stores
$O(nM+M^2)$ numbers for candidate images and its Schur matrix.  Exact flat
screening reduces the certificate workspace to $O(s^2)$ but retains the
$O(nM)$ cache; this is the measured implementation above.  A lazy hierarchical
index reduces active candidate storage to $O(ns+s^2)$ only when every skipped
tree node carries a certified upper bound on all descendant quotient gains.
Without such a bound, tree pruning is heuristic and cannot support the
coverage theorem.  Once the space is selected, hp and SR both have online
projection work and deployment storage $O(nr+r^2)$; the matched dimensions
$110$ and $68$ therefore directly explain the observed online and memory
advantage.

We first examine the noncoercive two-dimensional Helmholtz prototype because
its phase and wave-current diagnostics expose the effect of structured
enrichment most clearly.  The hp and hpc paths, Schur gains, stability gate,
and stopping budget are unchanged; phase, amplitude, gradient, and
wave-current quantities are evaluated only after the spaces have been frozen.
Table~\ref{tab:helmholtz-quality} gives the result for the locked heterogeneous
medium and source.

\begin{table}[H]
\centering
\caption{Posterior Helmholtz wave-quality audit.  The equal-dimension rows use
105 coordinates.  The matched-error hpc row is the first frozen state no
worse in solution error than terminal hp; none of these diagnostics enters
selection or certification.  Larger phase coherence is better.}
\label{tab:helmholtz-quality}
\begin{tabular}{llrrrrrr}
\toprule
comparison & method & dim. & solution & amplitude & phase RMS & coherence & wave current\\
\midrule
equal dimension & hp  & 105 & $.6525$ & $.5477$ & $.7610$ & $.8033$ & $.8454$\\
                & hpc & 105 & $.4010$ & $.3245$ & $.3441$ & $.9519$ & $.5574$\\
matched hp error & hp  & 105 & $.6525$ & $.5477$ & $.7610$ & $.8033$ & $.8454$\\
                 & hpc & 38  & $.6430$ & $.5218$ & $.6906$ & $.8293$ & $.9337$\\
\bottomrule
\end{tabular}
\end{table}

\begin{figure}[H]
\centering
\includegraphics[width=.80\textwidth]{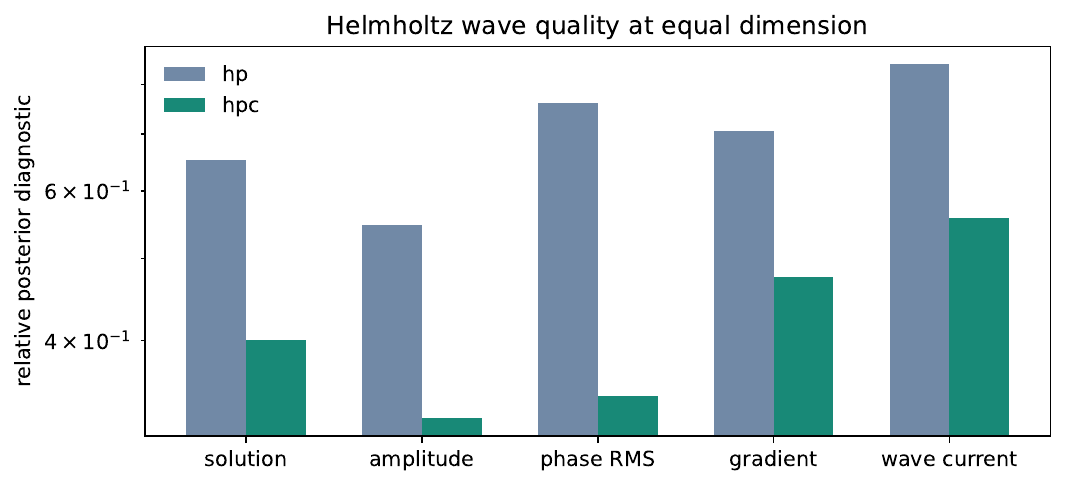}
\caption{Independent wave-quality diagnostics at equal dimension.  Relative
to hp, hpc reduces solution, amplitude, phase, gradient, and wave-current
errors.  At the more aggressive 38-coordinate matched-error point in
Table~\ref{tab:helmholtz-quality}, most quantities remain comparable but the
wave-current error worsens, exposing a physical cost hidden by the global
solution norm.}
\label{fig:helmholtz-quality}
\end{figure}

The equal-dimension result shows that wave-compatible enrichment improves
the oscillatory field rather than only the residual used by the algorithm.
The matched-error result gives the complementary limitation: hpc obtains the
hp solution error with $63.8\%$ fewer coordinates, but such compression is not
uniformly quality preserving.  Posterior PDE diagnostics therefore strengthen
the positive approximation claim while preventing the certificate from being
interpreted as a guarantee for every physical observable.

We next turn to heterogeneous Darcy flow.  Global energy error does not by itself show whether the reduced solution
preserves the transport structures for which the operator-adapted coordinates
were introduced.  We therefore freeze both spaces and draw 64 new right-hand
sides that enter neither calibration nor certification.  Table~\ref{tab:pde-quality}
reports posterior physical diagnostics.  These quantities do not affect any
refinement decision.

\begin{table}[H]
\centering
\caption{Independent PDE-specific audit at matched global accuracy.  Errors
are relative; larger high-flux cosine is better.  ``Worst balance'' is the
$\ell^\infty$ discrete conservation residual normalized by the forcing
$\ell^\infty$ norm.  Means are over 64 fresh loads.}
\label{tab:pde-quality}
\begin{tabular}{lrrr}
\toprule
diagnostic & conforming hp & screened SR & relative SR change (\%)\\
\midrule
energy error & $.3316$ & $.3298$ & $-0.6$\\
all-face flux error & $.3349$ & $.3287$ & $-1.9$\\
top-decile flux error & $.3118$ & $.2785$ & $-10.7$\\
top-decile flux cosine & $.9500$ & $.9592$ & $+1.0$\\
source-response error & $.1140$ & $.1134$ & $-0.6$\\
mean absolute local balance & $2.9778$ & $2.9958$ & $+0.6$\\
worst local balance & $2.6115$ & $1.9775$ & $-24.3$\\
\bottomrule
\end{tabular}
\end{table}

To identify the source of the Darcy gain without appealing to an informal
``structure match,'' we repeat the same posterior audit after removing whole
candidate classes.  The two SR ablations use 67--68 coordinates and the same
Schur rule; only the admissible library changes.

\begin{table}[H]
\centering
\caption{Candidate-class ablation on the 64-load Darcy audit.  The h/p-only
failure shows that Schur screening by itself does not explain the gain.  The
operator-only result shows that coefficient-derived coordinates supply most
of the compression, while the mixed library gives the best global and
dominant-flux errors.}
\label{tab:darcy-mechanism-ablation}
\begin{tabular}{lrrrr}
\toprule
space & dimension & energy error & top-decile flux error & worst balance\\
\midrule
conforming hp & 110 & $.3316$ & $.3118$ & $2.6115$\\
mixed c/h/p SR & 68 & $.3298$ & $.2785$ & $1.9775$\\
SR h/p-only & 67 & $.6008$ & $.5273$ & $4.6411$\\
SR operator-only & 68 & $.3551$ & $.2817$ & $1.4002$\\
\bottomrule
\end{tabular}
\end{table}

The conclusion is specific.  Operator-derived candidates are necessary for
the observed compression in this library, and ordinary h/p directions improve
their global approximation when combined.  However, flux error restricted to
the top conductivity decile is $.3735$ for hp and $.3784$ for mixed SR.
Therefore the evidence does not support the simpler explanation that SR merely
resolves the largest permeability coefficients; it supports preservation of
the solution's dominant flux pathways, as measured independently after the
solve.

The paired 95\% intervals for the top-decile flux error and worst local
balance exclude zero; SR is better on 42 and 43 of 64 loads, respectively.
The corresponding intervals for total flux, source response, and mean balance
include zero.  Thus this experiment supports a specific physical conclusion,
not a universal dominance claim: at the same global accuracy, the smaller
operator-adapted space preserves dominant channel transport and suppresses
extreme local defects more effectively.  Such posterior diagnostics can also
reveal when the variational certificate is misaligned with a quantity of
interest, without making that quantity part of the certificate.

\paragraph{Compatible Maxwell evolution.}
The preceding tests compare frozen approximation spaces.  We therefore add a
deployment test in which the selected spaces themselves evolve a constrained
two-dimensional Maxwell system.  The reference is a sparse compatible
$N_2/P_1^{\rm disc}$ de~Rham discretization advanced by implicit midpoint;
the mature control is the complete low-order $N_1/P_0$ compatible space, and
the p and c rows enrich that same space by 22 and 24 directions, respectively.
Thus the comparison changes the approximation space but holds the discrete
complex, time integrator, training trajectories, and four unseen audit
trajectories fixed.  The c-functions are generated from the prescribed
material and wave operator, not from the audit solutions.

\begin{table}[H]
\centering
\caption{Deployed Maxwell reduced dynamics.  Errors compare complete evolved
trajectories with the compatible full-order solution.  Drift, skew, and Gauss
defects remain at roundoff for every row.  The level-3 full sparse evolution
takes $.191$ s; the present dense reduced evolutions take $.337$ s (p) and
$.330$ s (c), so this is an accuracy and structure-preservation result rather
than a runtime claim.}
\label{tab:maxwell-reduced-dynamics}
\begin{tabular}{llrrrrr}
\toprule
level & space & dimension & trajectory & Poynting & sensor trace & energy history\\
\midrule
2 & low-order & 255  & $.8794$ & $.9137$ & $.6837$ & $.5677$\\
2 & p         & 277  & $.9031$ & $.8905$ & $.6953$ & $.4544$\\
2 & c         & 279  & $.3766$ & $.3108$ & $.3378$ & $3.25\,10^{-4}$\\
3 & low-order & 1023 & $.6112$ & $.4858$ & $.4799$ & $.1447$\\
3 & p         & 1045 & $.6436$ & $.5718$ & $.4760$ & $.1432$\\
3 & c         & 1047 & $.2700$ & $.2718$ & $.1811$ & $4.68\,10^{-5}$\\
\bottomrule
\end{tabular}
\end{table}

This evolved comparison confirms that the earlier projection gain survives
deployment.  At level 3, c reduces the trajectory error by $58.1\%$ relative
to matched p, the Poynting error by $52.5\%$, and the sensor-trace error by
$62.0\%$.  It also reveals a failure that a static audit would hide: the
selected p-block slightly worsens total trajectory error at both levels.
For c, the reduced generator is skew-adjoint to within $7.8\,10^{-16}$,
relative energy drift is below $6.4\,10^{-14}$, and the normalized Gauss
defect is below $3.4\,10^{-14}$.  Combining p and c selects exactly the c-space,
so the mixed candidate family offers no further gain in this test.  Dense
reduced time stepping ceases to be competitive at level 3; sparse or
matrix-free reduced propagation is therefore required before this example
can support an online-speed claim.

A comparison only with polynomial enrichment is not sufficient for a reduced
dynamics claim.  We therefore repeat the audit against the mature
energy-inner-product POD construction.  Eight training trajectories provide
enough snapshots to construct a genuine rank-48 control; all methods then use
the same midpoint evolution and four unseen trajectories.  Table~\ref{tab:maxwell-pod-control}
reports the larger level.

\begin{table}[H]
\centering
\caption{Matched mature-ROM control at Maxwell level 3.  The POD+c and POD+p
spaces contain 24 POD coordinates and 24 accepted innovations.  The POD-48
row has exactly the same deployed dimension.  Time excludes candidate
selection but includes reduced-operator construction and four trajectory
evolutions; the sparse full solve takes $.187$ s.}
\label{tab:maxwell-pod-control}
\begin{tabular}{lrrrrr}
\toprule
space & dimension & trajectory & Poynting & sensor trace & time (s)\\
\midrule
POD-24   & 24 & $.0754$ & $.0253$ & $.0426$ & $.0219$\\
POD-48   & 48 & $\mathbf{.0364}$ & $\mathbf{.0118}$ & $\mathbf{.0200}$ & $.0271$\\
POD-24+p & 48 & $.0757$ & $.0248$ & $.0426$ & $.0285$\\
POD-24+c & 48 & $.0749$ & $.0216$ & $.0456$ & $.0287$\\
\bottomrule
\end{tabular}
\end{table}

This control changes the conclusion.  The compact POD implementation is
$6.9$ times faster than the sparse full evolution at rank 48 and is markedly
more accurate than either enriched rank-48 space.  The c-block modestly
improves the rank-24 POD trajectory and Poynting errors, whereas p does not,
but neither block supplies the missing empirical covariance directions as
efficiently as 24 further POD modes.  All reduced generators remain
skew-adjoint and Gauss compatible to roundoff.  Thus Schur--Riesz acceptance
correctly identifies stable operator innovations, but stability and residual
novelty do not imply dimension-optimal dynamics when a well-sampled POD space
is available.  Maxwell is consequently retained as a structure-preserving
deployment and diagnostic example, not as evidence that c-enrichment
outperforms a mature reduced-order method.  A prospective advantage would
require an extrapolatory, moving-feature, or data-poor regime in which the
training covariance does not already span the relevant dynamics.

Finally, we independently reproduce the definite Maxwell discretization in
MFEM~\cite{anderson2021mfem}, rather than comparing only implementations built
from the same assembly code.  On identical triangular meshes we solve
$\operatorname{curl}\operatorname{curl}E+E=f$ with first- and second-order
N\'ed\'elec elements, homogeneous tangential data, and a prescribed smooth
solution.  Table~\ref{tab:mfem-external-control} gives the two largest meshes;
the complete audit contains four refinements at each order.

\begin{table}[H]
\centering
\caption{External MFEM 4.8 control.  DOFs, manufactured solution, mesh,
N\'ed\'elec order, and $L^2$ norm are matched.  MFEM uses serial PCG with a
Gauss--Seidel smoother; the paper stack uses sparse direct solution.  Times
therefore document the tested implementations rather than an optimized package
ranking.}
\label{tab:mfem-external-control}
\begin{tabular}{rrlrrr}
\toprule
mesh & order & code & DOFs & $L^2$ error & solve (s)\\
\midrule
$32^2$ & 1 & MFEM        & 3,136  & $.0283323765$ & $.0049$\\
       &   & paper stack & 3,136  & $.0283323760$ & $.0032$\\
$32^2$ & 2 & MFEM        & 10,368 & $.0003263851$ & $.1239$\\
       &   & paper stack & 10,368 & $.0003263806$ & $.0151$\\
$64^2$ & 1 & MFEM        & 12,416 & $.0141692785$ & $.0374$\\
       &   & paper stack & 12,416 & $.0141692778$ & $.0148$\\
$64^2$ & 2 & MFEM        & 41,216 & $.0000816116$ & $.9691$\\
       &   & paper stack & 41,216 & $.0000815997$ & $.1130$\\
\bottomrule
\end{tabular}
\end{table}

The two independent codes agree to at least five significant digits in every
reported error and recover the expected first- and second-order convergence.
This rules out a private baseline implementation as the source of the Maxwell
accuracy and constraint results.  We next implement the SR calculation
directly on the MFEM operator and independently in Firedrake/PETSc
\cite{rathgeber2016firedrake}.  In each case the host package supplies the
mesh, N\'ed\'elec spaces, variational assembly, boundary treatment, and native
p controls.  SR receives only the assembled operator, the p1-to-p2 transfer,
and ten operator-frequency functions; its reduced solve uses dense algebra.

\begin{table}[H]
\centering
\caption{Direct external-code SR comparison at approximately matched error.
The native control is p1 on a $64^2$ mesh; SR uses p1 on a $16^2$ mesh plus
two selected frequency functions.  ``Offline'' includes transfer, candidate
projection, quotienting, and scoring.  Break-even is the number of repeated
solves needed to amortize that measured offline cost.}
\label{tab:external-direct-sr}
\begin{tabular}{llrrrrr}
\toprule
backend & method & dimension & $L^2$ error & offline (s) & solve (s) & break-even\\
\midrule
MFEM & native p1 & 12,416 & $.014169$ & --- & $.0374$ & ---\\
     & native p2 & 672 & $.005217$ & --- & $.00153$ & ---\\
     & SR-c       &    738 & $.014737$ & $.300$ & $.00665$ & 10\\
Firedrake & native p1 & 12,416 & $.014170$ & --- & $.0594$ & ---\\
          & native p2 & 672 & $.003047$ & --- & --- & ---\\
          & SR-c       &    738 & $.011254$ & $1.898$ & $.00742$ & 40\\
\bottomrule
\end{tabular}
\end{table}

The result isolates both the gain and its cause.  On the $16^2$ MFEM mesh,
selecting 64 ordinary high-order coordinate functions changes the p1 error
only from $.056615$ to $.056586$; generic SR screening is therefore not the
source of compression.  Two operator-frequency functions reduce it to
$.014737$.  Firedrake independently gives the same qualitative separation,
reducing $.056690$ to $.011254$.  Moreover, the normalized innovation lower
bound remains $.896$ (MFEM) and $.597$ (Firedrake) for two functions, but
falls below $4\,10^{-8}$ when eight nearly duplicate frequencies are retained.
Thus the Riesz gate detects precisely where apparently richer enrichment
becomes redundant.  The comparison is intentionally favorable to
Trefftz-style enrichment because the operator frequency is declared.  It
demonstrates portability, compression, and a useful rejection diagnostic; it
does not imply an advantage for arbitrary Maxwell data.  Parallel MFEM,
Firedrake matrix-free projection, partial assembly, and package-specific
preconditioner tuning remain outside the timing claim.

Figure~\ref{fig:external-maxwell-convergence} supplies the missing mesh study.
The SR-c curve is approximately parallel to p1: the two accepted functions
improve the error constant, not the observed asymptotic exponent.  Native p2
already attains lower error with fewer coordinates on this smooth solution.
Hence no accuracy/mesh crossover against the best polynomial control is
observed; the dimension and repeated-query crossovers in
Table~\ref{tab:external-direct-sr} are explicitly relative to p1.

\begin{figure}[H]
\centering
\includegraphics[width=.68\textwidth]{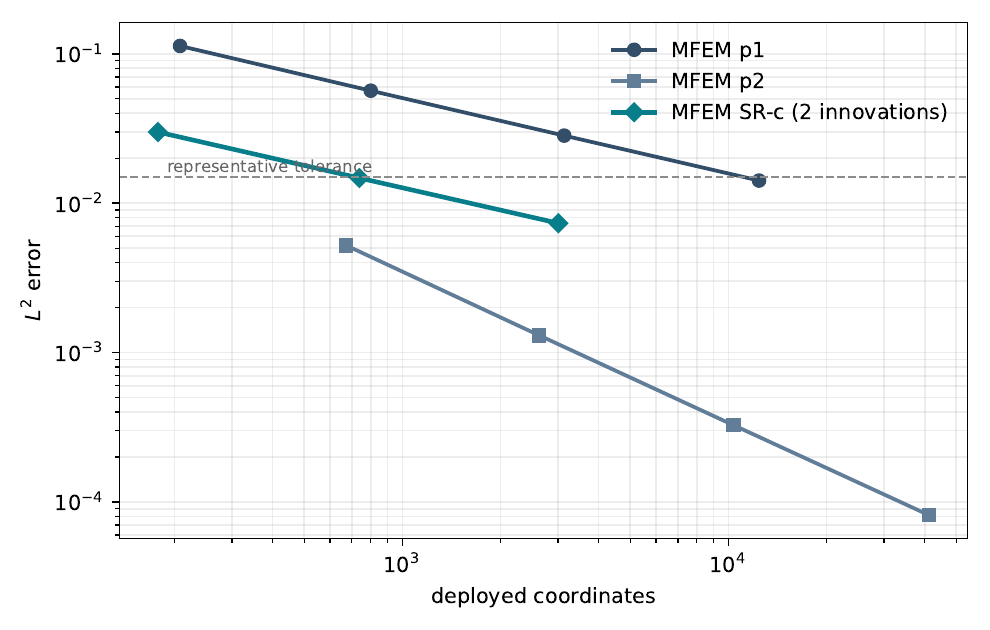}
\caption{External Maxwell convergence.  SR-c shifts the p1 curve downward but
does not change its observed slope; p2 remains superior for this smooth
manufactured solution.  The result is therefore a mechanism and portability
test, not evidence against mature p-refinement.}
\label{fig:external-maxwell-convergence}
\end{figure}

\begin{figure}[H]
\centering
\includegraphics[width=.78\textwidth]{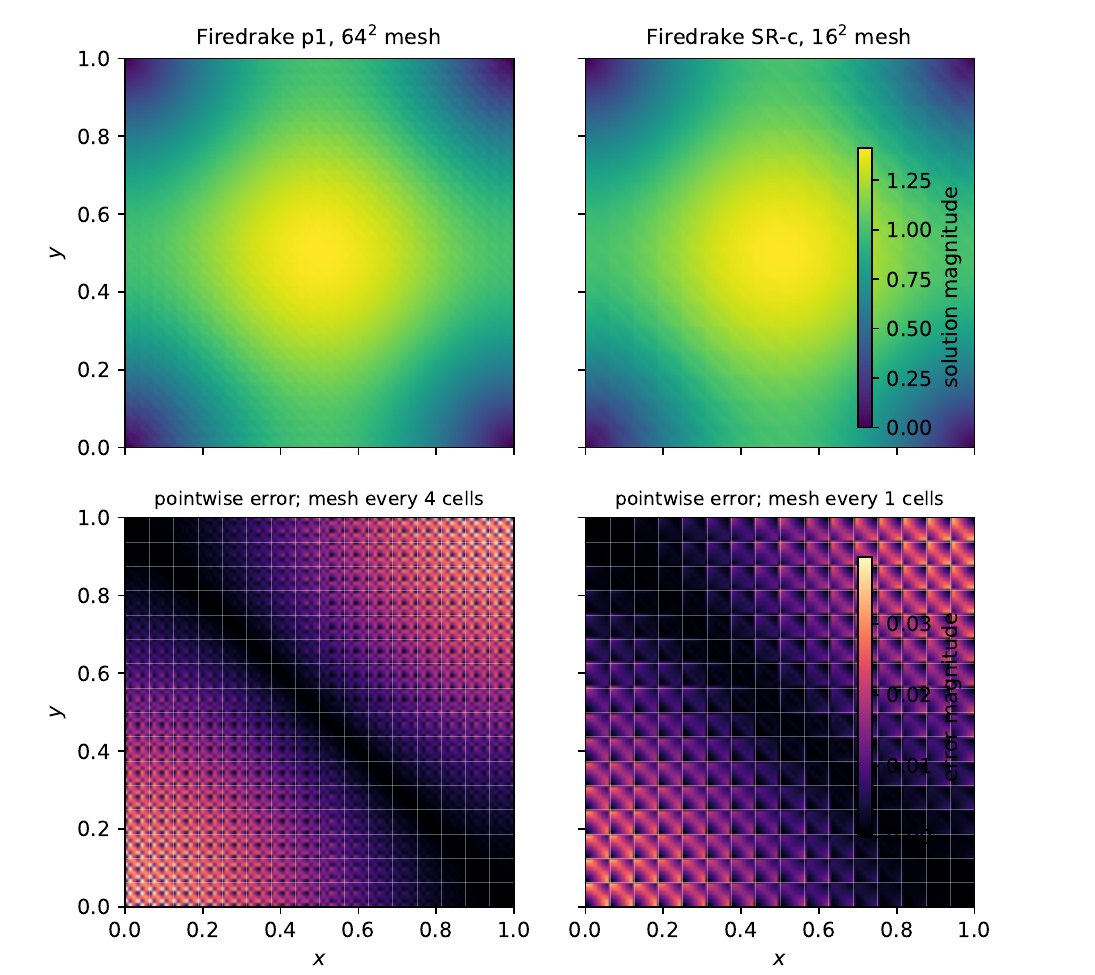}
\caption{Matched Firedrake fields.  Native p1 uses a $64^2$ mesh, whereas SR-c
uses a $16^2$ mesh and two operator-frequency functions.  Common color scales
show comparable field structure and the different pointwise error patterns;
mesh lines are subsampled on the fine-grid panel for visibility.}
\label{fig:external-maxwell-fields}
\end{figure}

The query count in Table~\ref{tab:external-direct-sr} also requires a precise
scope.  Offline work is reusable only while the operator, mesh, boundary
treatment, candidate family, and certified right-hand-side class remain
fixed.  Each new right-hand side must pass a residual-coverage check.  Failure
of that check triggers enrichment or reselection and resets part of the
amortization.  Figure~\ref{fig:external-maxwell-amortization} consequently
shows a conditional deployment calculation, not a universal lookup cost.

\begin{figure}[H]
\centering
\includegraphics[width=.68\textwidth]{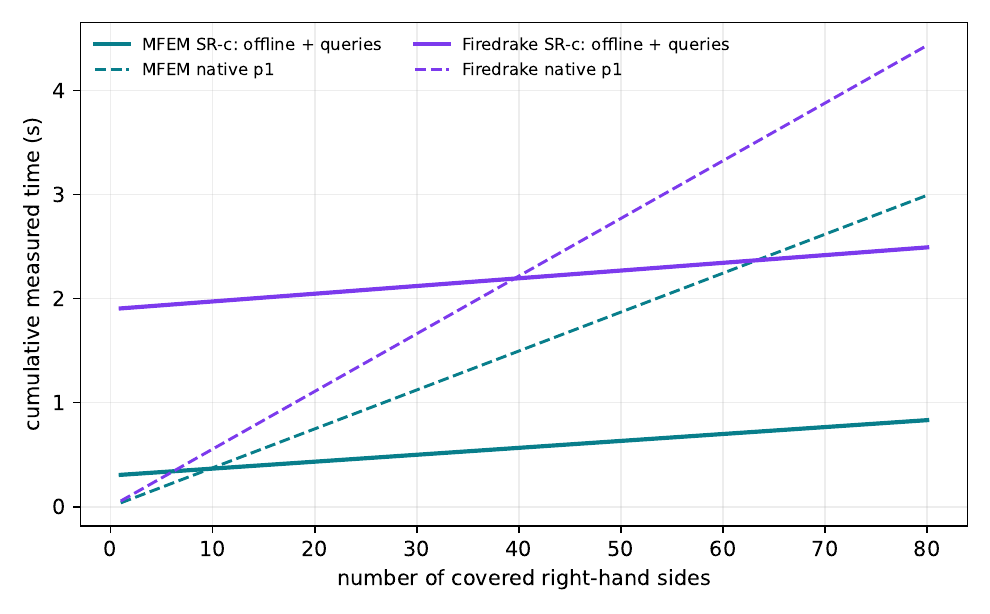}
\caption{Conditional offline--online amortization against the p1 control.
Only right-hand sides covered by the frozen operator-adapted space follow
these lines.  Changed operators or failed coverage tests require new offline
work.  At the tested sizes, one-shot SR remains slower once setup is included.}
\label{fig:external-maxwell-amortization}
\end{figure}

Together, these examples validate the paper's operational claims:
operator structure can generate compact candidates without solution leakage;
quotient--Schur--Riesz quantities distinguish stable novelty from residual
utility and yield sharp bounds; pool coverage and bulk retention predict
finite-step contraction; and the decisive complementary calculation can be
performed sparsely on a two-dimensional operator.

\secondarynumericsA
\begin{table}[h]
\centering
\caption{Evidence hierarchy.  The categories distinguish demonstrated
applications from promising, controlled, secondary, and negative evidence.}
\begin{tabular}{p{.24\textwidth}p{.20\textwidth}p{.42\textwidth}}
\toprule
evidence class & problem & principal finding and limitation\\
\midrule
1. principal demonstrated application & heterogeneous Darcy flow, especially
SPE10 & improvement in all 11 Model 1/2 layer--load cases, up to $41.6\%$ at
equal dimension; $3.4$--$6.8$ times fewer intrinsic coordinates than the first
matched P1/P2 meshes in Model 1; strongest natural result\\
2. leading unresolved application & L-shaped Poisson & certificate $.418$ at
34 enriched DOFs versus $.747$ at 47 hp DOFs; matched-reference audit remains\\
3. transfer evidence & stochastic permeability ensemble & pc improves five of
six fields, median $15.6\%$; one finite-budget greedy failure\\
4. controlled compression & Allen--Cahn, high contrast, and asymptotic layers
& 12 versus 129 interface coordinates, up to $24.3\%$ high-contrast gain, and
near-exact operator-derived outflow layers; not broad natural benchmarks\\
5. supplementary example & blinking-vortex transport & improves all eight
prescribed variants, median $4.5\%$; mature characteristic FEM remains faster\\
6. mixed wave evidence & Helmholtz and frozen transport & local wave enrichment
improves the locked 2D wavefield, whereas other plausible rich families are
rejected or underperform; candidate-family compatibility is not automatic\\
\bottomrule
\end{tabular}
\label{tab:evidence-hierarchy}
\end{table}

\begin{table}[h]
\centering
\caption{Terminal refinement histories for Poisson problems on the L-shaped
domain.  The first four problems use prescribed exact weak solutions with,
respectively, smooth, corner-singular, oscillatory, and recursively structured
behavior; the final row solves $-\Delta u=1$ with homogeneous Dirichlet data.
Cost counts free scalar coordinates and recorded c-factor evaluations.}
\resizebox{\textwidth}{!}{%
\begin{tabular}{llrrrrl}
\toprule
problem & method & cycles & certificate & effectivity & cost & actions (h,p,c)\\
\midrule
smooth & h/hc/hp/hpc & 3 & 0.3094 & 1.0000 & 3 & $(3,0,0)$\\
corner & h & 3 & 0.7052 & 1.0004 & 6 & $(3,0,0)$\\
corner & hp/hpc & 5 & 0.5807 & 1.0011 & 5 & $(1,4,0)$\\
oscillatory & h & 5 & 1.2475 & 1.0000 & 14 & $(5,0,0)$\\
oscillatory & hp & 6 & 1.2364 & 1.0000 & 15 & $(3,3,0)$\\
oscillatory & hpc & 9 & 1.2355 & 1.0000 & 22 & $(3,4,2)$\\
factorized & hc/hpc & 1 & 1.3028 & 1.0002 & 12 & $(0,0,1)$\\
unit-load L-shape & h/hc/hp/hpc & 1 & 0.3017 & --- & 5 & $(1,0,0)$\\
\bottomrule
\end{tabular}
}
\end{table}

The controls prevent a universal-win interpretation.  Smooth and corner
problems reject c.  On the oscillatory problem, c is admissible but does not
improve on hp at the prescribed tolerance.  It is useful only on the explicitly
factorized positive control.  On the standard, non-manufactured problem
$-\Delta u=1$ on the L-shaped domain with homogeneous Dirichlet data, all four
policies choose the same h-action and reject both p and c.  No exact solution
or reference saturation enters that decision.

\subsection{Geometry-informed enrichment at a re-entrant corner}
The re-entrant corner provides a natural opportunity for enrichment because
its leading singular form is determined by the domain geometry.  For the
problem $-\Delta u=1$, we add the classical corner form
$\chi(r)r^{2/3}\sin(2\theta/3)$, where the cut-off $\chi$ has four proposed
radii.  Its exponent, angle, and location are determined entirely by the
domain geometry.  Neither an exact solution nor a fine reference enters
selection.  The ordinary hp refinement rule and the enriched refinement rule use the
same quotient-flux estimator and compare exact variational gains per added
coordinate.

\begin{table}[h]
\centering
\caption{Preliminary matched-cost L-shaped Poisson test.  Certificates are
normalized by their common initial value.  The rows report the closest states
below approximately 35 and 50 free coordinates.}
\resizebox{\textwidth}{!}{%
\begin{tabular}{lrrrr}
\toprule
coordinate cap & hp DOFs & hp certificate & enriched DOFs & enriched certificate\\
\midrule
about 35 & 22 & .689 & 34 & .418\\
about 50 & 47 & .747 & 41 & .538\\
best visited & 8 & .638 & 34 & .418\\
\bottomrule
\end{tabular}
}
\label{tab:lshape-singular-preliminary}
\end{table}

The early result is promising but not yet a principal theorem-validation
claim.  The enriched trajectory reaches a substantially smaller reliable
upper bound with fewer coordinates, whereas later estimator values are not
monotone despite monotone Galerkin energy.  A common fine-reference energy
audit, repeated initial meshes, and a matched wall-time comparison are needed
before attributing a definitive complexity advantage.  This is therefore the
leading unresolved application rather than evidence on the same footing as
SPE10.

\subsection{Transported enrichment for a blinking vortex}
Transport provides a complementary source of economical non-polynomial
structure.  A coherent scalar is
alternately rotated by two localized vortices while diffusion acts at every
step.  The resulting thin, curved material structures are non-polynomial but
have a compact description through the flow.  Compatible and characteristic finite element
methods provide the relevant structure-preserving context
\cite{cotter2023,douglasrussell1982}; our narrower question is whether a
transported function supplies approximation content not already available to
ordinary localization and order enrichment.

We discretize the entire trajectory by a space--time least-squares form.  This
makes the time-dependent operator coercive and permits direct use of the
innovation identity.  An h-candidate is a packet localized on one time slab, a
p-candidate is a higher spatial Fourier mode on one slab, and a c-candidate is
a localized material packet transported coherently through all slabs by the
numerically integrated prescribed velocity.  All policies begin from the same
low-order space and accept twelve directions.  The fine discrete trajectory is
not used to form or rank candidates; it is retained only to audit the predicted
decrease.

\begin{table}[h]
\centering
\caption{Prescribed blinking-vortex robustness study.  ``Horizon'' changes the
number of time slabs; the remaining rows change one parameter from the medium
case.  Errors are relative space--time least-squares energy errors after twelve
accepted directions.  The final column is the relative reduction from hp to
hpc, so positive values favor the augmented space.}
\begin{tabular}{lrrr}
\toprule
case & hp error & hpc error & relative reduction (\%)\\
\midrule
short horizon & 0.6190 & 0.6152 & 0.6\\
medium horizon & 0.6528 & 0.6224 & 4.7\\
long horizon & 0.6543 & 0.6348 & 3.0\\
low diffusion & 0.6546 & 0.6235 & 4.7\\
high diffusion & 0.6488 & 0.6211 & 4.3\\
weak vortex & 0.6588 & 0.6169 & 6.3\\
strong vortex & 0.6481 & 0.6270 & 3.3\\
finer spatial grid & 0.6924 & 0.6394 & 7.7\\
\bottomrule
\end{tabular}
\label{tab:blinking-robustness}
\end{table}

Table~\ref{tab:blinking-robustness} is not a claim of universal dominance: the
gains are deliberately reported even when small.  Their value is that hpc
improves on hp in every prescribed variation while retaining an exact, computable
account of why each direction was accepted.  The median relative reduction is
$4.5\%$.  Pure c also improves on the better of pure h and pure p in all eight
cases, with median relative reduction $5.9\%$.  In the 20-direction medium
run, hpc selects seven transported, twelve order, and one localized direction
and reduces error from $0.6523$ for hp to $0.5870$.  For every accepted block,
the exact-gain identity agrees with the independent space--time energy audit to
numerical precision.  Thus the evidence supports complementarity, not
replacement: the transported enrichment follows coherent history, while p resolves
conventional oscillation and h remains available for spatially localized
residuals.

Equal enrichment count measures representation size but not evaluation work.
We therefore repeat the medium case without changing candidates or PDE data
under three declared budgets.  Table~\ref{tab:cost-audit} separates the
mathematical compression from its present implementation.

\begin{table}[h]
\centering
\caption{Cost audit for the medium blinking-vortex case.  A compact record
stores the parameters defining a function and shares the prescribed flow map;
the direct-support budget instead counts every sampled basis value evaluated by
the present dense numerical implementation.}
\resizebox{\textwidth}{!}{%
\begin{tabular}{llrrr}
\toprule
budget & policy & accepted & selected c & relative energy error\\
\midrule
coefficients & hp  & 12 & 0 & 0.6528\\
coefficients & hpc & 12 & 5 & 0.6224\\
direct support evaluations & hp  & 13 & 0 & 0.6486\\
direct support evaluations & hpc & 13 & 0 & 0.6486\\
compact representation records & hp  & 15 & 0 & 0.6354\\
compact representation records & hpc & 15 & 2 & 0.6133\\
\bottomrule
\end{tabular}
}
\label{tab:cost-audit}
\end{table}

Under equal coefficients, hpc improves on hp by $4.7\%$; under compact record
cost, it improves by $3.5\%$.  Under direct dense evaluation cost, no
c-direction is selected and the two methods coincide exactly.  The conclusion
is therefore specific and actionable: the transported functions supply useful
low-description directions, but a factorized or matrix-free evaluator is
required to turn that representation saving into computational saving.  The
paper does not count an unimplemented fast evaluator as observed speedup.

We subsequently remove the assembled space--time Gram matrix.  The
matrix-free formulation stores the retained functions, their PDE images,
the residual, and shared characteristic geometry; candidate vectors are
generated one at a time from compact descriptors.  On the medium case it
reproduces the hp and hpc errors in Table~\ref{tab:cost-audit} to machine
precision.  It avoids a $5.12$ MB dense Gram matrix, retaining $0.672$ MB of
Galerkin state and $0.0128$ MB of shared characteristic coordinates, and
reduces the indicative hp/hpc selection times from $3.58/4.13$ seconds to
$1.96/2.39$ seconds.  These modest-grid timings validate the operator
formulation rather than an asymptotic claim.  Direct evaluation still visits
the full trajectory of a transported function, which explains why the
support-evaluation policy rejects c and identifies the remaining kernel-level
optimization target.

Finally, the least-squares energy is the norm controlled by the theorem, but
it is not itself a physical observable.  After freezing every selection, we
therefore reconstruct hp and hpc in a separate pass and compare both with an
independent sparse sequential solve of the same discrete PDE.  These values do
not enter candidate construction or ranking.

\begin{table}[h]
\centering
\caption{Independent physical-norm audit after twelve selected directions.
Entries are relative errors; percentages give the reduction from hp to hpc.}
\begin{tabular}{lrrrrrr}
\toprule
grid/time & \multicolumn{2}{c}{space--time $L^2$} &
\multicolumn{2}{c}{terminal $L^2$} & \multicolumn{2}{c}{$H^1$ seminorm}\\
& hp & hpc & hp & hpc & hp & hpc\\
\midrule
$16^2/12$ & .821 & .674 & .889 & .722 & .920 & .815\\
$24^2/12$ & .812 & .642 & .883 & .684 & .924 & .813\\
$32^2/16$ & .825 & .655 & .904 & .736 & .933 & .817\\
\bottomrule
\end{tabular}
\label{tab:physical-audit}
\end{table}

Across the three resolutions, hpc reduces space--time $L^2$ error by
$17.9$--$20.9\%$, terminal $L^2$ error by $18.6$--$22.5\%$, and the spatial
$H^1$ seminorm error by $11.4$--$12.4\%$.  The absolute errors remain large
because both policies are restricted to twelve new directions.  Accordingly,
Table~\ref{tab:physical-audit} establishes transfer of the certified
least-squares improvement to physical discrete norms; it is not a
high-accuracy convergence study or a continuous-PDE error bound.

\subsection{Comparison with characteristic finite elements}
Characteristic finite elements provide the natural established comparison
for this transported structure.  We therefore use a conventional characteristic
Galerkin method on periodic conforming triangular meshes.  Backward
characteristics transport the previous solution, diffusion is treated by a
Galerkin solve, and periodic degrees of freedom are identified exactly.  A
fine quadratic characteristic solve supplies an independent audit trajectory.

\begin{table}[h]
\centering
\caption{Genuine P1/P2 characteristic finite-element baseline for the
12-step blinking-vortex problem.  The errors are measured against a periodic
P2 solve on a $40\times40$ mesh.}
\resizebox{\textwidth}{!}{%
\begin{tabular}{lrrrrr}
\toprule
degree & mesh & periodic DOFs & space--time $L^2$ & terminal $L^2$ & $H^1$ seminorm\\
\midrule
P1 & $8^2$  & 64   & .443 & .495 & .796\\
P1 & $12^2$ & 144  & .237 & .252 & .562\\
P1 & $16^2$ & 256  & .136 & .143 & .384\\
P2 & $8^2$  & 256  & .108 & .101 & .319\\
P2 & $12^2$ & 576  & .039 & .049 & .132\\
P2 & $16^2$ & 1024 & .018 & .025 & .059\\
\bottomrule
\end{tabular}
}
\label{tab:characteristic-fem}
\end{table}

This control is deliberately unfavorable to an overstated conclusion.  The
characteristic method is much more accurate and faster than the present
low-rank space--time prototype.  Its spatial DOF count is not identical to the
global trajectory-coordinate count in Table~\ref{tab:physical-audit}, because
the characteristic solver evolves a complete spatial state at every step.
Even with that qualification, Table~\ref{tab:characteristic-fem} shows that the
current experiments establish reliable low-rank enrichment, not a competitive
replacement for mature transport discretization.  The appropriate next test
is therefore to add residual-selected transported macro-functions to the
characteristic finite-element space itself and ask whether they reduce its
mesh or polynomial requirements.

\subsection{Interface and multiscale compression}
Compact non-polynomial enrichment also arises naturally from interfaces and
heterogeneous coefficients.  We apply the same innovation calculation to two
coercive problems whose useful coordinates arise from different PDE
structure.  The first is a radial Allen--Cahn boundary-value problem, for which
the energy residual proposes localized transition profiles.  The second is a
high-contrast elliptic problem with a procedurally generated coefficient, for
which coefficient-dependent harmonic coordinates generate the c-candidates.
Neither selection rule uses the exact solution.

\begin{table}[h]
\centering
\caption{Non-transport controls.  ``Coordinates'' counts the retained scalar
directions.  The Allen--Cahn column reports the energy error against an
independently resolved solution; the elliptic column reports relative energy
error.  All enrichments are selected by the discrete residual and exact
Schur-innovation gain.}
\resizebox{\textwidth}{!}{%
\begin{tabular}{llrrr}
\toprule
problem & approximation space & contrast & coordinates & error\\
\midrule
Allen--Cahn & P2, 64 radial elements & -- & 129 & $8.19\,10^{-4}$\\
Allen--Cahn & P1, 8 elements + selected profile & -- & 12 & $6.90\,10^{-4}$\\
high-contrast elliptic & physical sine modes & $10^2$ & 20 & .303\\
high-contrast elliptic & physical + harmonic-coordinate modes & $10^2$ & 20 & .262\\
high-contrast elliptic & physical sine modes & $10^4$ & 20 & .812\\
high-contrast elliptic & physical + harmonic-coordinate modes & $10^4$ & 20 & .615\\
constant-coefficient control & either function family & 1 & 20 & .0558\\
mismatched-coordinate control & physical + shifted harmonic modes & $10^4$ & 20 & .879\\
\bottomrule
\end{tabular}
}
\label{tab:nontransport-controls}
\end{table}

The Allen--Cahn enrichment reaches slightly lower error than a globally
quadratic discretization with roughly one tenth as many coordinates.  This is
a compression result, not a claim for a new approximation space: once its
location and scale have been selected, the added hyperbolic-tangent profile is
algebraically the same enrichment used by an optimized partition-of-unity or
extended finite element method.  The contribution here is that its parameters
are proposed from the variational residual and admitted by the same exact gain
identity used for h and p candidates.

For the elliptic problem, adding harmonic-coordinate candidates reduces the
error of physical-coordinate modes by $13.4\%$ at contrast $10^2$ and $24.3\%$
at contrast $10^4$.  The constant-coefficient control collapses to the same
space, as it should, and quotienting removes the duplicate representation.
Conversely, deliberately mismatched coordinates are worse than the ordinary
modes.  Thus the improvement is attributable to PDE-compatible coordinates,
not merely to a larger candidate set.  The coefficient field is procedural rather
than a standard permeability benchmark, and comparisons with GMsFEM and
localized orthogonal decomposition remain necessary before drawing a broader
multiscale conclusion.

\subsection{Operator-compatible enrichment for a convection layer}
Known operator asymptotics provide especially economical functions for a
convection-dominated problem.  We solve
$-\epsilon\Delta u+\partial_xu=0$ with unit inflow, zero outflow, and
homogeneous transverse Neumann conditions.  A fixed lifting $1-x$ imposes the
nonhomogeneous boundary data before any candidate is ranked.  The correction
space contains ordinary trigonometric modes and exponential outflow layers
whose scales are fixed multiples of $\epsilon$; no reference solution is used
to place them.  Exponentials localized at the inflow form the mismatched
control.  The nonsymmetric equation is evaluated through its coercive
least-squares form.

\begin{table}[h]
\centering
\caption{Controlled convection-layer errors after twenty directions.  Values
in parentheses are relative least-squares residuals.}
\begin{tabular}{rrrr}
\toprule
$\epsilon$ & p & pc & misplaced c\\
\midrule
.020 & .846 (.938) & $<5\,10^{-6}$ ($1.1\,10^{-4}$) & .948 (.990)\\
.010 & .915 (.971) & $3.0\,10^{-5}$ ($1.25\,10^{-3}$) & .967 (.996)\\
.005 & .943 (.983) & $1.21\,10^{-3}$ (.0298) & .975 (.998)\\
\bottomrule
\end{tabular}
\label{tab:convection-layer}
\end{table}

The sharp separation between outflow and misplaced profiles confirms the
compatibility principle in a controlled setting: asymptotic information from
the operator can define an economical approximation class, but localization
on the wrong boundary is useless.  The result is intentionally
aligned with the known asymptotic solution class and therefore validates mechanism rather than broad
application performance.  Extending the same comparison to oblique flow,
curved internal layers, and an inf--sup formulation remains open.

\endsecondarynumericsA
\secondarynumericsE

\subsection{Principal application: SPE10 heterogeneous permeability}
SPE10 provides the clearest natural demonstration of coefficient-adapted
enrichment.  Model 1 of the tenth SPE comparative project provides a
$100\times20$ permeability field ranging from $0.001$ to $998.9$ millidarcys
\cite{christieblunt2001}.  We solve three single-phase Darcy problems on this
fixed field: diagonal, offset, and three-well load configurations.  One
harmonic-coordinate function family is computed from permeability alone and reused
unchanged for all three loads.  A cyclically shifted permeability field defines
the mismatched-coordinate control.  Reference pressures are never used for
candidate construction or ranking.

\begin{table}[h]
\centering
\caption{SPE10 Model 1 relative energy errors after twenty accepted
directions.  The last column gives the reduction from p to the combined pc
space.}
\begin{tabular}{lrrrrr}
\toprule
load & p & c & pc & mismatched c & pc reduction (\%)\\
\midrule
diagonal wells & .880 & .514 & .514 & .926 & 41.6\\
offset wells   & .945 & .882 & .882 & .964 & 6.7\\
three wells    & .963 & .936 & .931 & .987 & 3.4\\
\bottomrule
\end{tabular}
\label{tab:spe10}
\end{table}

\begin{figure}[p]
\centering
\includegraphics[width=\textwidth]{figures/spe10_model1_fields.pdf}
\caption{How permeability-adapted enrichment changes the SPE10 approximation.
Read the top row from left to right: the permeability field (a) generates the
operator-adapted functions visualized by their coordinate lines (b); the
reference pressure (c) is used only for the post hoc audit.  The bottom row
compares the p-only approximation (d) with its pointwise error (e) and the pc
error (f).  Panels (e) and (f) use the same color scale, so the darker pc panel
shows the reduction directly.  The lines in (b) describe enrichment functions,
not a deformed finite-element mesh.}
\label{fig:spe10-fields}
\end{figure}

\begin{figure}[p]
\centering
\includegraphics[width=.96\textwidth]{figures/spe10_model1_convergence.pdf}
\caption{SPE10 energy-error histories for three prescribed right-hand sides.  The
permeability-adapted function family is unchanged across panels.  Its benefit is
largest for diagonal wells but remains positive for all three loads, whereas
coordinates from the shifted permeability field are consistently harmful.}
\label{fig:spe10-convergence}
\end{figure}

At equal budgets, pc improves over p for all three loads.  The $41.6\%$
diagonal-well reduction is the strongest natural-PDE result in this paper; the
smaller $6.7\%$ and $3.4\%$ reductions show that its magnitude is load
dependent.  The mismatched control is worse than p in every case.  Together
these observations separate operator compatibility from task utility:
permeability defines potentially useful directions, while the variational
residual determines which of them matter for a particular load.

To show how the approximation spaces differ geometrically, we next compare
the meshes selected by hp and hc on the same SPE10 problem.  The comparison
makes the central distinction visible: hc uses compact coefficient-adapted
functions for extended channel structure, whereas hp must express unresolved
structure through local polynomial resolution.  We
embed conforming P1 and P2 spaces into the same $100\times20$ discrete energy
operator used for pc.  Starting from a common coarse triangulation, local
algebraic residual energy marks cells by D\"orfler bulk marking.  Thus the
reference pressure is used only to report error, not to refine the mesh.

\begin{table}[h]
\centering
\caption{SPE10 Model 1 narrow-band comparison at target relative energy error
$.48$, using D\"orfler fraction $.10$.  Each row is the first discrete space
below the target; unequal terminal errors reflect the remaining discrete
refinement increments.}
\resizebox{\textwidth}{!}{%
\begin{tabular}{lrrrrr}
\toprule
space & error & labelled DOFs & quotient DOFs & triangles & preceding error/DOFs\\
\midrule
pc reference (not target-matched) & .514 & 23 & 23 & fixed coefficient grid & ---\\
hc & .475 & 42 & 42 & 56 & .508/36\\
adaptive P1 & .459 & 72 & 71 & 163 & ---\\
adaptive P2 & .442 & 137 & 130 & 78 & ---\\
\bottomrule
\end{tabular}
}
\label{tab:spe10-matched-mesh}
\end{table}

\begin{figure}[p]
\centering
\includegraphics[width=\textwidth]{figures/spe10_matched_mesh.pdf}
\caption{SPE10 narrow-band approximation-space comparison.  The white and magenta wavy
lines in the left panel are level curves of two permeability-adapted
coordinates $X$ and $Y$; they are neither mesh edges nor error contours.
The left-panel color is $\log_{10}$ permeability.  In the three mesh panels,
color instead denotes the common-scale $\log_{10}$ local residual indicator;
these quantities are identified by separate color bars.
Polynomial modes evaluated in $(X,Y)$ form the c enrichment.  From left to
right: the first panel supplies permeability and c-coordinate geometry; hc
retains the residual-selected c block and then refines h; adaptive P2 and P1
use only piecewise quadratic or linear polynomials.  The attained errors
$.475$, $.442$, and $.459$ lie in a narrow band around the common target $.48$.
Mesh panels are colored by logarithmic local residual energy.  The comparison
concerns solution-space dimension, not removal of the coefficient grid.}
\label{fig:spe10-matched-mesh}
\end{figure}

\begin{figure}[p]
\centering
\includegraphics[width=.76\textwidth]{figures/spe10_matched_convergence.pdf}
\caption{SPE10 Model 1 error convergence measured in one common Darcy energy
norm.  The horizontal coordinate is the numerical rank of the represented
solution space, so exact sampled redundancies in P2 are not counted.  The pc
curve records successive residual-selected permeability-adapted functions;
P1 and P2 record successive residual-marked meshes.  The curves therefore
compare approximation efficiency, not assembly time.}
\label{fig:spe10-matched-convergence}
\end{figure}

The narrow-band meshes make the compression claim more defensible.  At
comparable errors, hc uses 42 intrinsic coordinates, versus 71 for adaptive P1
and 130 for adaptive P2, corresponding to reductions of $41\%$ and $68\%$.
These are representation counts, not end-to-end speedups.  Harmonic-coordinate
construction and fine coefficient integration remain part of hc cost, while
mature sparse FEM assembly benefits the polynomial methods.
The hc mesh concentrates near the well residual while the retained
c block continues to represent the global permeability channels.  This is the
intended division of labor: c carries coherent operator-induced structure and
h resolves the remaining localized error.

Figures~\ref{fig:spe10-matched-mesh} and
\ref{fig:spe10-matched-convergence} also clarify where c- or hc-refinement
helps.  Its clearest advantage occurs when the PDE operator generates a
coherent, non-polynomial feature extending across many cells, as the channel
geometry does here.  One accepted function can then replace many local mesh
coordinates.  Pure h-refinement remains preferable for sharply localized
geometric defects, while p-refinement remains preferable for elementwise
smooth behavior already represented efficiently by polynomials.  The combined
hc method is useful in the intermediate case: h resolves genuinely local
features and c represents operator-induced structure that would otherwise be
repeated across the refined mesh.  This is a conditional advantage rather
than a universal one.  The residual gain and quotient test are essential
because they reject c whenever this compact structure is absent, returning the
method to ordinary h- or hp-refinement.

\subsubsection{Conforming matched-tolerance and scaling study}
The preceding comparison isolates approximation efficiency.  We next impose a
stricter common protocol in which hp and hc use the same conforming fine P1
discretization, coefficient-aligned triangulations, load projection, bulk
parameter $\theta=.55$, stopping tolerance, and refinement budget.  The hp
method may choose h- or p-enrichment; the hc method may choose h-refinement or
a small block of permeability-adapted functions.  Every accepted space
contains the preceding space algebraically.  After each enrichment, an energy
Gram decomposition removes exact redundancy, so the reported dimension is
intrinsic rather than a count of labelled functions.

Two complementary bounds are evaluated.  The Galerkin residual Riesz norm is
the exact error reducible within the common fine space and supplies the
stopping rule.  An equilibrated $RT_2/P_1^{\rm disc}$ flux gives a fully
computable continuum majorant; the fine reference solution is used only to
audit this bound.  The methods therefore do not inspect the unknown pressure
when choosing between h, p, and c.

\begin{table}[h]
\centering
\caption{Conforming SPE10 comparison under one refinement and stopping
protocol.  The target is relative to the initial reducible Galerkin error.
Terminal errors differ because a discrete enrichment can overshoot the common
threshold.  Wall times are indicative single-process measurements including
candidate search; memory is peak traced allocation.}
\resizebox{\textwidth}{!}{%
\begin{tabular}{llrrrrrrrrr}
\toprule
grid & method & target & energy error & majorant & effectivity & dimension & elements & quadrature work & time (s) & memory (MB)\\
\midrule
$50\times10$  & hp & .60 & .421 & .950 & 2.26 & 45  & 117  & 468  & .038 & 2.76\\
               & hc & .60 & .514 & .995 & 1.93 & 39  & 73   & 292  & .051 & 2.09\\
$100\times20$ & hp & .48 & .330 & .778 & 2.36 & 169 & 549  & 2196 & .322 & 42.8\\
               & hc & .48 & .401 & .810 & 2.02 & 131 & 311  & 1244 & .213 & 22.7\\
$150\times30$ & hp & .38 & .284 & .532 & 1.88 & 337 & 1175 & 4700 & 1.706 & 193.4\\
               & hc & .38 & .317 & .551 & 1.74 & 253 & 639  & 2556 & .738 & 98.6\\
\bottomrule
\end{tabular}}
\label{tab:spe10-conforming-scaling}
\end{table}

The coordinate saving grows from $13.3\%$ to $24.9\%$ as the coefficient grid
is enlarged; the corresponding element and quadrature-work savings grow from
$37.6\%$ to $45.6\%$.  Hc is slower on the smallest case, but is respectively
$1.51$ and $2.31$ times faster on the two larger cases and uses approximately
half the peak allocation.  On this rough field hp proposes p-enrichments but
selects h at every accepted cycle.  Hc first accepts coherent c-blocks and
then switches to h, which is precisely the intended separation between global
operator-induced structure and residual localization.

The computational claims can therefore be audited individually rather than
inferred from approximation error alone.  Table~\ref{tab:computational-gain-audit}
collects the strongest controlled evidence.  ``System entries'' is the dense
reduced Galerkin matrix implied by the intrinsic dimension; it does not mean
that the ambient sparse PDE matrix disappears.  Peak allocation is measured,
whereas the reduced-system entry count is algebraic.  Online timing excludes
the separately reported offline construction.

\begin{table}[h]
\centering
\caption{Claim-by-claim computational audit.  Reductions compare the enriched
method with its stated baseline under the corresponding locked protocol.}
\label{tab:computational-gain-audit}
\begin{tabular}{p{.25\textwidth}p{.25\textwidth}rrr}
\toprule
quantity & controlled comparison & baseline & enriched & gain\\
\midrule
intrinsic DOFs & SPE10 $150\times30$, hp/hc & 337 & 253 & $24.9\%$ fewer\\
reduced-system entries & same, $d^2$ & 113,569 & 64,009 & $43.6\%$ fewer\\
peak traced memory (MB) & same, hp/hc & 193.4 & 98.6 & $49.0\%$ less\\
online solve ($\mu$s) & Darcy $63^2$, full/reduced & 83.9 & 4.9 & $17.1\times$ faster\\
refinement cycles & factorized control, h/hc & $>24$ & 1 & $>24\times$ fewer\\
\bottomrule
\end{tabular}
\end{table}

The first three rows are mutually consistent but not interchangeable: the
$24.9\%$ coordinate reduction yields a $43.6\%$ reduction in dense reduced
matrix entries, while measured peak allocation falls by $49.0\%$ because mesh
and candidate workspaces also shrink.  The repeated-solve gain is independently
measured and appears only after offline amortization; at this Darcy resolution
the observed break-even point is 2,173 right-hand sides.  Here one right-hand
side---or query---means one new forcing vector for the same fixed
$3,969\times3,969$ Darcy operator.  The full baseline is treated favorably: it
assembles and factorizes the sparse matrix once, then performs one LU backsolve
per forcing.  The enriched method constructs its 57-dimensional space once,
then projects and reconstructs one reduced solution per forcing.  The measured
costs are
\[
 T_{\rm full}(R)=.004331+83.919\,10^{-6}R,
 \qquad
 T_{\rm red}(R)=.175939+4.945\,10^{-6}R
 \quad\hbox{seconds}.
\]
Equating them gives
$R=(.175939-.004331)/(83.919-4.945)10^{-6}=2172.98$.
The timings use a batch of 200 distinct forcing vectors only to reduce timer
noise; division by 200 reports the per-forcing cost.  Thus 2,173 means 2,173
independent right-hand-side solves with a shared operator, not 2,173 matrix
factorizations.  Finally, the dramatic
cycle reduction is intentionally confined to the recursively factorized
positive control in Table~\ref{tab:factorized-cycle-cap}.  Natural SPE10 shows
fewer elements and less work, but does not yet establish a general theorem or
broad empirical claim of fewer adaptive cycles.

This separates algorithmic complexity from prototype overhead.  If hp needs
$n(\varepsilon)$ stable coordinates to reach tolerance $\varepsilon$ and the
certified enriched space needs $r(\varepsilon)$, then reduced storage changes
from $O(n(\varepsilon)^2)$ to $O(r(\varepsilon)^2)$, dense reduced
factorization from $O(n(\varepsilon)^3)$ to $O(r(\varepsilon)^3)$, and each
reduced backsolve from $O(n(\varepsilon)^2)$ to
$O(r(\varepsilon)^2)$.  Reconstruction costs $O(Nr(\varepsilon))$ for an
ambient dimension $N$.  SR discovery adds an offline candidate-construction
and certification cost.  Hence a slower present implementation is compatible
with lower online work and memory whenever $r(\varepsilon)\ll
n(\varepsilon)$; it is not evidence of an unconditional lower complexity
order.  The recursive-depth experiment supplies precisely such a favorable
approximation regime, while the negative controls show that it need not occur.
Exact caching of the invariant candidate--load cross matrix already reduces
selection time by $2.8\times$ without changing the selected space or its
error, confirming that part of the measured overhead is implementational.

\begin{figure}[p]
\centering
\includegraphics[width=.96\textwidth]{figures/spe10_sinum_theorem_validation.pdf}
\caption{Direct validation of the theoretical quantities over the conforming
adaptive cycles.  The upper-left panel compares actual error with the exact
fine-space residual bound and the reference-independent continuum majorant.
The upper-right panel reports the smallest accepted Schur-innovation
eigenvalue before quotient compression.  The lower-left panel compares the
predicted variational gain with the realized squared-error decrease; the
maximum relative discrepancy is $4.8\times10^{-14}$.  The lower-right panel
shows the observed contraction after each accepted enrichment.}
\label{fig:spe10-theorem-validation}
\end{figure}

\begin{figure}[p]
\centering
\includegraphics[width=.96\textwidth]{figures/spe10_sinum_scaling.pdf}
\caption{SPE10 scaling under the common stopping protocol.  Left: terminal
intrinsic dimensions for hp, hc, and the first ideal energy-minimizing
multiscale space reaching the hc error.  Right: hc quantities divided by hp
quantities; values below one favor hc.  Matrix-free candidate action and
quotient compression turn the growing representation saving into time and
memory savings at the two larger resolutions.}
\label{fig:spe10-scaling}
\end{figure}

The comparison with established multiscale approximation is deliberately
strong.  We compute an ideal energy-minimizing coarse space, the global-corrector
limit underlying localized orthogonal decomposition and constrained
energy-minimizing GMsFEM.  It uses the factored fine operator and global
correctors, and is therefore not a scalable localized implementation; its
construction time must not be compared directly with the adaptive timings in
Table~\ref{tab:spe10-conforming-scaling}.  It instead asks whether the proposed
c-family is already representation-optimal.

\begin{table}[h]
\centering
\caption{Ideal operator-adapted baseline.  ``Preceding'' is the last ideal
coarse space above the terminal hc error; ``first below'' is the first one
below it.  The ideal global space is often more compact, showing both the
strength of established multiscale approximation and the opportunity to admit
localized LOD/GMsFEM blocks through the same variational selection rule.}
\begin{tabular}{lrr|rr|rr}
\toprule
grid & hc dim. & hc error & preceding dim. & error & first-below dim. & error\\
\midrule
$50\times10$  & 39  & .514 & 10  & .758 & 21  & .471\\
$100\times20$ & 131 & .401 & 21  & .539 & 36  & .330\\
$150\times30$ & 253 & .317 & 105 & .442 & 171 & .193\\
\bottomrule
\end{tabular}
\label{tab:spe10-ideal-multiscale}
\end{table}

Thus the new contribution is not a claim that harmonic c-functions dominate
ideal multiscale bases.  It is a common quotient and variational principle that
measures the genuinely new contribution of heterogeneous candidate families,
rejects redundant additions, and retains a reliable error bound.  The ideal
baseline suggests the next constructive step: use localized
energy-minimizing multiscale blocks as c-candidates, while preserving the same
gain identity, quotient test, and stopping rule.

Model 2 tests that conclusion on the substantially more channelized Tarbert
and Upper Ness formations.  We retain the full $60\times220$ horizontal grid
on layers 20, 50, 70, and 85 and use the anisotropic horizontal permeability
components.  Two prescribed well configurations and twelve-direction budgets give
eight cases.  Coordinates are computed independently on each layer from its
permeability; cyclically shifted layer coordinates form the negative control.

\begin{table}[h]
\centering
\caption{SPE10 Model 2 equal-budget results.  Entries give the relative energy
errors of p and pc.  The final column is the percentage reduction in error
from the p-only baseline to pc, $100(E_p-E_{pc})/E_p$; positive values favor
pc.}
\begin{tabular}{llrrr}
\toprule
layer/formation & load & p & pc & reduction from p to pc (\%)\\
\midrule
20/Tarbert & diagonal & .920 & .835 & 9.3\\
20/Tarbert & offset   & .673 & .532 & 21.0\\
50/Upper Ness & diagonal & .998 & .987 & 1.1\\
50/Upper Ness & offset   & .999 & .930 & 6.9\\
70/Upper Ness & diagonal & .966 & .870 & 10.0\\
70/Upper Ness & offset   & .998 & .937 & 6.1\\
85/Upper Ness & diagonal & .999 & .994 & 0.5\\
85/Upper Ness & offset   & .816 & .578 & 29.1\\
\bottomrule
\end{tabular}
\label{tab:spe10-model2}
\end{table}

\begin{figure}[p]
\centering
\includegraphics[width=\textwidth]{figures/spe10_model2_layers.pdf}
\caption{SPE10 Model 2 permeability layers and pc improvement over p at equal
budgets.  The method improves all eight prescribed layer--load cases; the variation
from $0.5\%$ to $29.1\%$ confirms that operator-adapted geometry supplies
possible directions while the load controls their value.}
\label{fig:spe10-model2}
\end{figure}

Model 2 therefore extends rather than merely repeats the Model 1 result: pc
improves all eight cases, with median reduction $8.1\%$ and maximum $29.1\%$.
The mismatched control is worse in every case.  Some absolute errors remain
close to one because twelve global directions cannot resolve localized
channel flow; the claim is comparative approximation at a fixed small budget,
not a high-fidelity reservoir solve.

\subsection{Transfer across stochastic permeability fields}
Stochastic permeability fields show that the benefit extends beyond the
single SPE10 geometry.  We draw six prescribed lognormal permeability fields from
a Gaussian random-field model.  Three samples have contrast $10^2$ and
three have contrast $10^4$.  A fixed injection--production well pair defines
the load.  The c-candidates are sine modes in global harmonic coordinates
computed from the observed permeability; they use neither the reference
pressure nor its residual geometry.  As a negative control, a second set of
coordinates is computed from an independent permeability sample.  Each
policy accepts twenty directions, and a full fine-grid Darcy solve is used
only for the energy-error audit.

\begin{table}[h]
\centering
\caption{Prescribed stochastic Darcy-flow ensemble.  Entries are relative energy
errors after twenty accepted directions.  The final column is the relative
error reduction from p to pc; negative values indicate that greedy pc selection
followed an inferior finite-budget path.}
\begin{tabular}{rrrrrr}
\toprule
contrast & sample & p & c & pc & pc reduction (\%)\\
\midrule
$10^2$ & 11 & .302 & .323 & .259 & 14.3\\
$10^2$ & 23 & .209 & .236 & .174 & 16.8\\
$10^2$ & 37 & .328 & .359 & .318 & 2.8\\
$10^4$ & 11 & .565 & .513 & .455 & 19.5\\
$10^4$ & 23 & .402 & .398 & .320 & 20.5\\
$10^4$ & 37 & .557 & .561 & .566 & $-1.7$\\
\bottomrule
\end{tabular}
\label{tab:stochastic-darcy}
\end{table}

\begin{figure}[p]
\centering
\includegraphics[width=\textwidth]{figures/stochastic_darcy_fields.pdf}
\caption{The strongest prescribed stochastic Darcy sample (contrast $10^4$,
sample 23).  Top: observed permeability with the actual structured
computational grid, coefficient-adapted harmonic coordinate lines, and the
independent reference pressure.  Bottom: the polynomial approximation and
pointwise errors of the p and combined pc spaces after twenty accepted
directions.  The curved coordinate lines are an enrichment geometry, not an
adaptively deformed mesh.  The pc policy reduces relative energy error from
$0.402$ to $0.320$.}
\label{fig:darcy-fields}
\end{figure}

\begin{figure}[p]
\centering
\includegraphics[width=.94\textwidth]{figures/stochastic_darcy_convergence.pdf}
\caption{Relative energy error against accepted-direction count for all six
prescribed stochastic Darcy samples.  Thin lines show individual fields and
thick lines show medians.  The combined space generally improves early and
retains its median advantage, while the crossing in one contrast-$10^4$
sample records the finite-budget greedy failure reported in
Table~\ref{tab:stochastic-darcy}.}
\label{fig:darcy-convergence}
\end{figure}

The combined space improves on p in five of six samples, with median
relative improvement $15.6\%$.  Coordinates from the independent permeability
field are worse than p in all six cases, so the gain is not explained by an
arbitrary nonlinear coordinate change.  Pure c is not consistently superior;
the useful space combines ordinary and coefficient-adapted directions.  The
single negative pc result is also informative.  Exact one-step gain does not
make a finite-horizon greedy sequence globally optimal: an early c-choice can
exclude a better later p-sequence under a hard cardinality cap.  Block
look-ahead or exchange steps are therefore required for a monotone
finite-budget comparison with p.  The ensemble is a natural stochastic Darcy
model, but it remains procedurally generated rather than a field permeability
benchmark.

\subsection{Irregular-domain scaling and the algebraic solver boundary}
The preceding Darcy tests use rectangular coefficient grids.  To test geometry
transfer, we lock the same spectral cutoff, two-patch oversampling rule,
contraction-guided controller, and quotient-flux majorant on an L-shaped domain
and a square with an interior polygonal hole.  Coefficient seeds, well
orientations, and contrasts up to $10^6$ vary without retuning.  The local
spectral functions are obtained by constrained energy minimization over
adjacency-defined patches; neither Cartesian separability nor the fine
reference solution enters construction, selection, or stopping.

Across mesh levels four through six, spectral hc has lower post-hoc energy
error than hp in all 18 paired cases at the common certified target and lower
intrinsic dimension in 13.  Median hc/hp ratios are $0.943$ for energy error,
$0.877$ for intrinsic dimension, and $0.536$ for elements.  The functional
majorant has effectivity range $1.005$--$1.635$ and median $1.043$.

\begin{figure}[t]
\centering
\includegraphics[width=.96\textwidth]{figures/locked_hc_scaling.pdf}
\caption{Locked hp--spectral-hc scaling on heterogeneous L-shaped and
perforated domains.  Marker shape gives the fine-mesh level.  Panels
(a)--(c) show post-hoc energy error against intrinsic dimension, elements,
and total construction-plus-solve time at the common certified target;
panel (d) shows majorant effectivity.  Spectral hc generally reduces error and
representation size, whereas its present construction is slower.}
\label{fig:locked-hc-scaling-submission}
\end{figure}

\endsecondarynumericsE
\secondarynumericsC
This approximation result is distinct from algebraic acceleration.  We use a
symmetric geometric V-cycle as the baseline and wrap it in the balanced
quotient correction defined in \cref{eq:balanced-cem-vcycle}.  The resulting
iteration counts remain bounded as the fine system grows from 705 to 60,736
unknowns:
\begin{table}[h]
\centering
\caption{PCG iterations for the locked irregular-domain multigrid stress test.}
\label{tab:cem-vcycle-submission}
\begin{tabular}{rrrr}
\toprule
level & unknowns & geometric V-cycle & balanced CEM--V-cycle\\
\midrule
4 & 705--872       & 14--15 & 11--12\\
5 & 2,945--3,664   & 15--16 & 13--14\\
6 & 12,033--15,008 & 15--16 & 14\\
7 & 48,641--60,736 & 16--17 & 14--15\\
\bottomrule
\end{tabular}
\end{table}
The result validates mesh-independent convergence and a modest improvement in
the contraction constant, but not a speedup.  At level six, eight independently
assembled loads take approximately $7$--$10$ ms per solve with geometric
multigrid and $12$--$15$ ms with the balanced correction; spectral setup adds
$0.5$--$0.7$ s.  Reuse across right-hand sides therefore has no break-even
point in the present implementation because each corrected solve remains
slower.  Ordinary multigrid is the preferred solver, while spectral functions
are retained for their demonstrated approximation-space value.
\endsecondarynumericsC

\secondarynumericsB
\subsection{A scope test for frozen transport enrichment}
The L-bend problem provides a useful test of when transported enrichment
should be declined.  A conforming diffusive transport solution is frozen at
the final time and used as the target of the stationary coercive projection
\[
 \min_{v\in V}\ \|u_* -v\|_{L^2}^2
       +0.04^2\|\nabla(u_*-v)\|_{L^2}^2.
\]
The transported c-candidate is generated independently by a coarser characteristic map.
On the initial 40-element mesh, h and hc errors are respectively $0.075261$
and $0.075252$; hp and hpc errors are $0.040452$ and $0.040448$.  The squared
error reductions equal the quotient--innovation gains to rounding error, but
the relative error improvements are only $0.012\%$ and $0.008\%$ while the
feature records 81 flow factors.  Hence the static refinement rule should reject
this enrichment block.

This result materially changes the interpretation of the earlier transient
bend experiment.  Its advantage came from characteristic transport as an
evolution operator, not from a compact new direction in the frozen spatial
approximation space.  Transport is therefore natural, but this particular
frozen c-candidate family does not close the natural-application gap.

The same rejection lesson appears in a variable-index Helmholtz test.  An
optical-path coordinate was formed from the refractive index and used to warp
Fourier modes.  Constant index reproduces the ordinary function family exactly,
which the quotient identifies.  For variable index and wave numbers 12, 18,
and 24, however, the combined finite-budget errors are respectively $.991$,
$.979$, and $.975$, compared with $.968$, $.958$, and $.965$ for ordinary
modes.  Thus an intuitively plausible nonlinear coordinate is not sufficient:
the tested phase map neither captures the relevant Green-function structure
nor improves the prescribed successive path.  We retain this negative result because
it demonstrates why the method must evaluate rich functions variationally
rather than presuming that non-polynomial structure is useful.

\begin{table}[h]
\centering
\caption{Strict factorized certificate cap, target relative level $0.65$.
Nonattainment is reported at the prescribed 24-cycle cap.}
\label{tab:factorized-cycle-cap}
\begin{tabular}{lrrrr}
\toprule
method & attained & terminal certificate & cost & wall seconds\\
\midrule
h   & no  & 1.5743 & 33 & 91.2\\
hp  & no  & 1.5119 & 35 & 222.7\\
hc  & yes & 1.3028 & 12 & 0.63\\
hpc & yes & 1.3028 & 12 & 0.75\\
\bottomrule
\end{tabular}
\end{table}

Across the original factorized histories in this table, certificates are monotone and audited effectivity
lies in $[1.0000,1.0011]$.  Observed one-step contraction ratios range from
$0.639$ to values close to one.  The latter show that the abstract uniform
contraction hypothesis is not yet numerically sharp; they do not invalidate
Galerkin monotonicity.
\endsecondarynumericsB

\section{Conclusions and outlook}
This paper develops Schur--Riesz refinement as a general route from
width-theoretic approximation structure to an adaptive numerical method.
Classical h/p blocks, expert-designed Trefftz or multiscale functions, and
automatically generated operator responses enter through the same interface.
Each proposal is viewed through the trial-to-test map, quotiented by the
incumbent operator image, tested for stable innovation, and assigned an exact
variational gain.  The method therefore separates three questions that are
often conflated: whether a direction is new, whether it is stable, and whether
it is useful for the declared problem family.

The principal theoretical bridge is the conditional-width identity: the
Schur spectrum is the Kolmogorov-width spectrum of the available candidate
ellipsoid after incumbent content has been removed.  Operator or randomized
transfer methods make those directions computable.  The safe
automatic-transfer theorem then compares a proposed space with the incumbent
on independent functional data and supplies a population non-regression bound;
rejection recovers the incumbent exactly.  Riesz stability remains a separate
admissibility requirement and prevents a useful projection from being encoded
by unstable coefficients.

For adaptive enlargement, ideal residual capture $\gamma$, stable novelty
$A_C$, and coordinate error $\varepsilon_C$ are kept separate before the
enlarged problem is solved.  Exact gain and a reliable global estimator play
distinct roles in Algorithm~\ref{alg:hpc}: the former selects an action, while
the latter decides when refinement may stop.
The bulk-coverage theorem supplies the global bridge from those local
decisions to convergence: the joint pool must capture a fixed fraction of the
residual, and the selected batch must retain a fixed fraction of that gain.
Both fractions are computable Schur diagnostics, and together they imply
geometric convergence in the induced residual norm without assuming a
separate estimator-reduction inequality.  Corollary~\ref{cor:finite-run-complexity}
turns the observed fractions into a product-form error envelope and explicit
batch, dimension, and work certificates; a failed coverage floor is returned
as a diagnosed limitation of the proposed candidate family.
The conditional-width identity explains the approximation content of each
available block, and the factorized residual majorant supplies a reusable
route from discrete innovation to certified remaining error.  Detailed
covariance, transport-lifting, and physical-norm variants are supporting
extensions rather than additional premises of the algorithm.

The cross-class experiments support this logic without implying that automatic
transfer always wins.  At matched dimension it approximately halves held-out
error for layered Helmholtz and reduces mean energy error by about $31\%$ for
coercive Darcy.  It improves coarse-grid Maxwell field and physical
diagnostics, but loses to the incumbent on the finer Maxwell discretization;
the safe rule rejects all five inferior fine-grid proposals.  A deliberately
redundant proposal is also rejected.  Supporting experiments show that
quotient-based transport discovery improves a hidden flow parameter,
complementary Schur elimination gives nearly unit-effectivity nonnormal bounds,
and bordered inertia evaluates sparse gaps up to 65,536 unknowns without a
dense complement.  A locked 24-cycle audit additionally verifies the
bulk-coverage identity and detects materially weaker candidate coverage in
layered media.
The matched 2D Darcy audit gives the corresponding computational tradeoff.
At the common error level $.405$, screened SR uses 68 coordinates rather than
110 for conforming hp, reducing deployment memory by $38.9\%$ and measured
online work per right-hand side by $30.1\%$.  Exact indexed screening reduces
SR offline time to $.1438$ s, still $28.4\%$ above hp; the extra setup
amortizes after approximately 21,000 right-hand sides.  Against sparse
full-order backsolves, hp and SR amortize after 1,492 and 1,872 right-hand
sides.  Separately, screened rather than dense coverage reduces certificate
workspace by $97.1\%$ and the measured dense-bulk offline calculation by
$24.9\%$.  On 64 independent loads, the smaller SR space also reduces
dominant-channel flux error by $10.7\%$ and the worst local-balance diagnostic
by $24.3\%$, while aggregate flux and source response remain statistically
tied.  In the locked Helmholtz case, equal-dimension hpc also reduces phase
RMS from $.7610$ to $.3441$ and wave-current error from $.8454$ to $.5574$;
at matched solution error it uses 38 rather than 105 coordinates, although
wave-current fidelity then deteriorates.  In compatible Maxwell evolution,
c reduces level-3 trajectory error from $.6436$ for matched p to $.2700$ and
preserves energy, skew-adjointness, and the Gauss constraint to roundoff; the
current dense reduced implementation is nevertheless slower than the sparse
full-order solve at that level.  A compact POD control restores a $6.9$-fold
deployment speedup and gives lower equal-rank error than POD+c, so Maxwell
supports safe structural deployment but not superiority over mature ROM.
Thus the present evidence is strongest for representation, online
complexity, and localized transport fidelity; lazy certified candidate
generation remains necessary to remove the working-memory overhead of storing
the complete candidate family.
The broader appendix shows both application gains and failures;
it supports breadth, not a claim that nonpolynomial enrichment always wins.

The resulting contribution is therefore not c-refinement alone, nor a new
candidate family.  It is a common, auditable decision calculus: quotient
projection tests novelty, the Riesz bound gates stability, and certified
variational gain ranks the admissible refinements.  The principal remaining tasks
are extensions rather than missing premises of the present theory.  On the analytical side, a
goal-oriented extension would augment residual gain by dual-weighted gains for
declared observables such as flux, phase, or far-field response.  Such an
extension must preserve the global stability gate and be validated on data
independent of both primal and adjoint selection; the posterior diagnostics in
this paper do not already provide that theorem.  On the implementation side,
the priorities are matrix-free candidate action, lazy generation, parallel
block evaluation, and a hierarchical index whose skipped nodes carry valid
upper bounds on every descendant quotient gain.  That last condition is
essential: an uncertified tree would reduce cost by changing the algorithm.
Sharper application-specific coverage bounds and tests with changing
operators and geometries remain necessary before making a general efficiency
claim.  The fixed-operator 3D Helmholtz audit establishes
scalability of the independent Petrov mechanism, but a matched comparison with
a mature adaptive 3D hp implementation remains future work.

\appendix
\numberwithin{figure}{section}
\numberwithin{table}{section}
\section{Supporting variants of the main theory}
\label{app:secondary-theory}
The main text retains the results needed to define and certify
Algorithm~\ref{alg:hpc}.  This appendix records useful extensions for
many-query load laws, two-level estimation, transport graph reconstruction,
and physical-norm calibration.  They broaden the framework but do not alter
the current refinement decision.
\storedtheoryA

\section{Proof of the mixed-refinement near-oracle theorem}
\label{app:mixed-near-oracle-proof}

\begin{proof}[Proof of Theorem~\ref{thm:mixed-near-oracle}]
Fix $S=S_m$ and a comparator $O$ with $|O|\leq k$.  If
$\mathcal F(S)\geq\mathcal F(O)$, the one-step estimate below is immediate.
Otherwise, let $b_q$ and $D_{q\mid S}$ be the residual vector and conditional
Schur Gram operator of block $q\notin S$.  The classical projection identity
gives
\[
 G_m(q)=\mathcal F(S\cup\{q\})-\mathcal F(S)
 =\langle b_q,D_{q\mid S}^{\dagger}b_q\rangle
 \geq B_1^{-1}\|b_q\|^2.
\]
Concatenate the blocks in $O\setminus S$.  Their joint conditional Gram
operator has lower bound $A_{2k}$ because
$|S\cup O|\leq2k$.  Hence their joint gain satisfies
\[
 \mathcal F(S\cup O)-\mathcal F(S)
 \leq A_{2k}^{-1}\sum_{q\in O\setminus S}\|b_q\|^2.
\]
Combining the last two displays and using monotonicity gives
\[
 \sum_{q\in O\setminus S}G_m(q)
 \geq\gamma\{\mathcal F(S\cup O)-\mathcal F(S)\}
 \geq\gamma\{\mathcal F(O)-\mathcal F(S)\}.
\]
There are at most $k$ terms, so the largest exact marginal gain is at least
$\gamma\{\mathcal F(O)-\mathcal F(S)\}/k$.

Let $q_m$ maximize the computed marginal and let $q^*$ maximize the exact
marginal.  Proposition~\ref{prop:main-safe-ranking} implies
\[
 G_m(q_m)\geq \widetilde G_m(q_m)-\Xi_m(q_m)
 \geq G_m(q^*)-\Xi_m(q^*)-\Xi_m(q_m)
 \geq \frac{\gamma}{k}
       \{\mathcal F(O)-\mathcal F(S_m)\}-\delta_m.
\]
Thus, with $\Delta_m=\mathcal F(O)-\mathcal F(S_m)$,
$\Delta_{m+1}\leq\rho\Delta_m+\delta_m$.  Iteration from
$\mathcal F(S_0)=0$ yields Eq.~\eqref{eq:mixed-near-oracle}.  Finally,
$\rho^k=(1-\gamma/k)^k\leq e^{-\gamma}$ gives
Eq.~\eqref{eq:mixed-near-oracle-k}.
\end{proof}

\section{Proof of the safe automatic-transfer theorem}
\label{app:safe-transfer-proof}

\begin{proof}[Proof of Theorem~\ref{thm:safe-automatic-transfer}]
The argument has four steps.

\emph{Step 1: convert approximation error into captured energy.}
Orthogonality gives, for every closed finite-dimensional $Z\subset Y$,
\[
 \|g-P_Zg\|_Y^2=\|g\|_Y^2-\|P_Zg\|_Y^2.
\]
After taking expectations,
$\mathcal E(Z_A)-\mathcal E(Z_I)=-\Delta$, where
$\Delta=\mathbb E[J_{Z_A}(g)-J_{Z_I}(g)]$.

\emph{Step 2: bound the marking variable.}
Let $X=J_{Z_A}(g)-J_{Z_I}(g)$.  Orthogonal projections are contractions, so
$0\leq J_Z(g)\leq\|g\|_Y^2\leq M$.  Consequently $-M\leq X\leq M$.

\emph{Step 3: obtain a one-sided population bound.}
Hoeffding's inequality for the independent variables
$X_i=J_{Z_A}(g_i)-J_{Z_I}(g_i)$ yields
\[
 \mathbb P\!\left\{\Delta<\widehat\Delta_n-t\right\}
 \leq \exp\!\left(-\frac{nt^2}{2M^2}\right).
\]
Choosing $t=\tau_{n,\alpha}$ shows that, with probability at least
$1-\alpha$, $\Delta\geq\widehat\Delta_n-\tau_{n,\alpha}$.

\emph{Step 4: apply the decision rule.}
If the automatic space is accepted, the preceding event and the strict
acceptance inequality imply
$\Delta\geq\widehat\Delta_n-\tau_{n,\alpha}>0$.
Step 1 then gives the second inequality in
Eq.~\eqref{eq:safe-transfer-nonregression}, and hence non-regression.  If either
the gain test or the Riesz admissibility test fails, the rule returns $Z_I$
itself, so both its space and its approximant are unchanged.  This proves the
first inequality in all cases.
\end{proof}

\end{document}